\documentclass[11pt]{article}
\usepackage[a4paper,margin=25mm,headheight=14pt]{geometry}
\usepackage{amsmath,amssymb,amsthm,mathtools,esint,booktabs,array}
\usepackage[T1]{fontenc}
\usepackage{lmodern,microtype,xcolor,needspace}
\usepackage[topmarks]{titleps}
\usepackage{hyperref}
\hypersetup{colorlinks=true,linkcolor=blue!40!black,citecolor=blue!40!black,
urlcolor=blue!40!black,pdftitle={Universal H\"older estimates for stable solutions to semilinear equations of fractional order},pdfauthor={},pdfsubject={Interior regularity of stable solutions and optimal dimension and order conditions}}

\newcommand{\ip}[2]{\langle #1,#2\rangle}
\newcommand{\Ree}{\operatorname{Re}}
\newcommand{\Imm}{\operatorname{Im}}
\newcommand{\ii}{\mathrm i}
\newcommand{\dd}{\,\mathrm d}

\newcommand{\R}{\mathbb R}
\newcommand{\Sph}{\mathbb S^{n-1}}
\newcommand{\Cyl}{\mathcal C}
\newcommand{\T}{\mathcal T}
\newtheorem{theorem}{Theorem}[section]
\newtheorem{proposition}[theorem]{Proposition}
\newtheorem{lemma}[theorem]{Lemma}
\newtheorem{corollary}[theorem]{Corollary}
\theoremstyle{remark}
\newtheorem{remark}[theorem]{Remark}
\numberwithin{equation}{section}
\allowdisplaybreaks[1]
\hypersetup{bookmarksdepth=2}
\newpagestyle{manuscript}{
  \settitlemarks{section}
  \sethead{\small\ifx\thesection\empty\else\thesection\quad\fi\sectiontitle}{}{}
  \setfoot{}{\thepage}{}
  \setheadrule{0.3pt}}
\makeatletter
\renewcommand{\@maketitle}{%
  \newpage\null\vskip 1em
  \begin{center}
    {\LARGE\@title\par}
    \ifx\@author\@empty\else\vskip 1em{\large\@author\par}\fi
    \vskip 1em{\large\@date\par}
  \end{center}\vskip 1em}
\makeatother
\title{Universal H\"older estimates for stable solutions\\[0.15em]
to semilinear equations of fractional order}
\author{Yiyang Xiao\thanks{School of Mathematical Sciences, Beijing Normal University, Beijing 100875, China. Email:
202531130042@mail.bnu.edu.cn}}
\date{17 October 2026}
\begin{document}
\maketitle
\begin{abstract}
We establish universal interior H\"older and fractional energy estimates
for stable solutions of $(-\Delta)^su=f(u)$, $0<s<1$, with
$f\in C^1_{\mathrm{loc}}(\R)$ nonnegative and convex. The bounds depend only on a weighted
$L^1$ norm of the solution and hold under optimal conditions on the
dimension and the fractional order. In particular, the half-Laplacian
estimate holds through dimension eight. The estimates extend to stable
finite-energy Dirichlet $L^1$-weak solutions without an a priori local
boundedness assumption, and an explicit logarithmic solution proves
optimality in this class. Applications include global bounds in convex
domains, bounded extremal solutions, and a Liouville theorem.
\end{abstract}

\smallskip
\noindent\textit{2020 Mathematics Subject Classification.}
35B45, 35B65, 35R11, 35J61.
\par\noindent\textit{Keywords.} Fractional Laplacian, stable solutions,
universal estimates, critical dimension, extremal solutions.

\par\medskip
\pdfbookmark[1]{Contents}{contents}
\begingroup
\small
\makeatletter
\c@tocdepth=1\relax
\renewcommand{\l@section}{\@dottedtocline{1}{0pt}{2.2em}}
\noindent\textbf{\normalsize\contentsname}\par\nobreak\smallskip
\@starttoc{toc}
\makeatother
\endgroup
\medskip

\section{Introduction}\label{reading-routes}

We consider stable solutions of
\begin{equation}\label{equation}
 (-\Delta)^s u=f(u)\qquad\text{in }B_2\subset\R^n,
 \qquad 0<s<1,
\end{equation}
where $f:\R\to[0,\infty)$ is convex and belongs to
$C^1_{\mathrm{loc}}(\R)$. All nonlinearities below are defined on the
whole real line with this local regularity. We normalize $(-\Delta)^s$
to have Fourier symbol $|\xi|^{2s}$ and write $L=(-\Delta)^s$.
Stability means that
\[
 \int_{B_2}f'(u)\xi^2\dd x
 \le \int_{\R^n}|(-\Delta)^{s/2}\xi|^2\dd x
 \qquad\text{for every }\xi\in C_c^\infty(B_2).
\]
The exterior values enter the estimates through the weighted norm
\[
 \|u\|_{L_s^1(\R^n)}
   =\int_{\R^n}\frac{|u(x)|}{1+|x|^{n+2s}}\dd x.
\]
The fractional energy and the solution conventions are fixed in
Section~\ref{preliminaries-section}.

For the classical Laplacian, Cabr\'e, Figalli, Ros-Oton, and Serra
\cite{CFRS} proved universal interior H\"older estimates through
dimension nine for nonnegative nonlinearities, without assuming
convexity. Their proof first extracts a partial control from stability
and then recovers the full estimate. Quantitative proofs and related
developments are presented in \cite{CabreQuant,CabreICM}.

For the half-Laplacian, Cabr\'e and Sanz-Perela \cite{CS} obtained
the corresponding interior H\"older estimate through dimension four
for nonnegative convex nonlinearities. Boundedness for $2\le n\le8$
was known for radially decreasing stable $H^{1/2}(\R^n)$ weak
solutions in the ball with zero exterior values
\cite[Theorem~1.2]{SP}. That result assumes a nondecreasing $C^2$
reaction, and its bound may depend on $f$ and $\|u\|_{L^1}$.
Boundedness was also known for exponential extremal solutions in
convex domains symmetric about each coordinate hyperplane \cite{RO};
see also \cite[Section~3.5]{CabreICM}.

The restriction on $n$ and $s$ is dictated by the logarithmic solution.
For $n>2s$, set
\begin{equation}\label{hardy-constants}
 A_{n,s}=4^s\frac{\Gamma(n/2)\Gamma(1+s)}{\Gamma(n/2-s)},
 \qquad
 H_{n,s}=4^s\frac{\Gamma((n+2s)/4)^2}{\Gamma((n-2s)/4)^2}.
\end{equation}
The function $u_*(x)=-2s\log|x|$ satisfies
$(-\Delta)^su_*=A_{n,s}e^{u_*}$, and $H_{n,s}$ is the sharp
fractional Hardy constant. Consequently $u_*$ is stable exactly when
$A_{n,s}\le H_{n,s}$. We prove interior regularity when
$A_{n,s}>H_{n,s}$.
These constants also arise in the fractional Gelfand problem:
Ros-Oton \cite{RO} identified the same condition for its extremal
solutions in convex domains symmetric about each coordinate hyperplane.
Our estimate is interior and
uniform over nonnegative convex nonlinearities and exterior data with
finite weighted $L^1$ norm.

\begin{theorem}[Universal interior estimate]\label{full}
Let $0<s<1$ and $n\ge1$. Assume
\begin{equation}\label{dimension-range}
 n=1\qquad\hbox{or}\qquad
 \bigl(n\ge2\ \hbox{and}\ A_{n,s}>H_{n,s}\bigr).
\end{equation}
Let $u\in L^\infty_{\mathrm{loc}}(B_2)\cap L_s^1(\R^n)$ be a stable
distributional solution of \eqref{equation}. Assume that $f$ is
nonnegative and convex on $\R$, with $f\in C^1_{\mathrm{loc}}(\R)$.
There are
$\alpha\in(0,\min\{1,2s\})$ and $C>0$, depending only on $n,s$, such that
\begin{equation}\label{full-bound}
 \|u\|_{C^\alpha(B_1)}+[u]_{H^s(B_1)}\le C \|u\|_{L_s^1(\R^n)}.
\end{equation}
The constants are independent of $f$ and the initial local bounds of $u$.
\end{theorem}

\begin{remark}
For $s=1/2$, the short proof of Theorem~\ref{explicit-exponents}
already gives the optimal dimension range $n\le8$.
\end{remark}

The local boundedness in this a priori statement only supplies the
qualitative regularity needed to justify the proof. It can be removed
in the finite-energy Dirichlet class introduced in Section~\ref{weak-section}.
Theorem~\ref{we-thm:upgrade} establishes the same estimate for stable
finite-energy Dirichlet $L^1$-weak solutions with exterior data
satisfying \eqref{we-eq:g}.

Proposition~\ref{we-prop:log} verifies that the original logarithmic
profile belongs to this class and is unbounded and stable whenever
$A_{n,s}\le H_{n,s}$. This proves optimality of the dimension and order
range for regularity in that weak class, including failure at equality.

A second result concerns classical harmonic replacement. For a smooth
compactly supported function $v$, the difference between the normal
derivatives of $v$ and its harmonic replacement is an average of
$-\Delta v$ over the ball. Theorem~\ref{intro-poisson-source}
constructs a positive average of this difference whose value is instead
an integral of $(-\Delta)^sv$ over the same ball. The two representations
follow from the boundary limits of the classical and fractional Green
functions. This identity turns the derivative estimate from stability
into an estimate for the nonnegative source $f(u)$.

The proof follows the partial-estimate strategy of \cite{CFRS}.
We apply stability to the corrected radial derivative from harmonic
replacement. Convexity makes the defect in its equation nonnegative,
but the derivative itself can change sign. The local energy controls
it on annuli only in the negative Sobolev norm $H^{s-1}$ of the
logarithmic cylinder, and taking a positive part can destroy
cancellation in that norm. We overcome this difficulty by keeping the
interaction between the positive and negative parts in the stability
inequality.

The corrected derivative intertwines the ordinary heat equation with
a positive operator in polar coordinates. Its heat semigroup controls
the retained interaction and the cost of taking a positive part.
The associated fractional extension provides the quadratic forms used
in the stability comparison and estimates the errors from radial
cutoffs. Thus the operator identity, sign argument, and localization
use the same linear construction.

The positive averaging identity \eqref{intro-fractional-source}
converts the derivative estimate into the weighted source bound of
Proposition~\ref{weighted-partial}. Since $f\ge0$, this bound gives
Morrey decay of the source mass. Riesz potentials then yield a
H\"older estimate in terms of local energy and the weighted tail.
An energy estimate and an iteration on concentric balls absorb the
remaining local energy, giving the universal interior bound.

The dimension restriction enters through the comparison of extension
energies. Under $A_{n,s}>H_{n,s}$, the higher angular modes are
controlled by a positive potential comparison. The radial mode
requires a more precise argument: its difference at zero frequency
is $A_{n,s}-H_{n,s}$. At low frequencies we retain
this lower bound by direct estimates of the Gamma ratios. At higher
frequencies a positive resolvent contribution from the first-order
radial term compensates for the loss in the real potential comparison.
The resulting bound is uniform in the angular degree and persists
under a small positive radial weight. This weight supplies the decay
exponent in the source estimate; the comparisons are proved in
Sections~\ref{gap-section} and~\ref{sp-appendix}.

For weak solutions, we construct bounded stable approximations whose
limit $v$ satisfies $v\le u$. The universal estimate makes $v$ locally
bounded, so it remains to transfer this property to $u$. Equality
in the stability and convexity inequalities forces $f$ to be affine
on the intervals between $v$ and $u$. Consequently, their difference
satisfies $(-\Delta)^s(u-v)=f'(v)(u-v)$. The coefficient $f'(v)$ is
locally bounded, and linear regularity gives local boundedness of $u$.
The universal interior estimate can then be applied directly to $u$;
see Section~\ref{weak-section}.

% MODULE common

\section{Preliminaries}\label{preliminaries-section}

We collect the notation, linear estimates, and cylinder symbols used below.

\subsection{Fractional energy and linear operators}\label{prelim-notation}\label{prelim-sphere}\label{prelim-cylinder}\label{prelim-semigroups}

Throughout, $n\ge1$ and $0<s<1$. We write
$B_R(a)=\{x\in\R^n:|x-a|<R\}$ and $B_R=B_R(0)$; a bar on a set denotes
closure, and $\Omega^c=\R^n\setminus\Omega$. The measure $dx$ is
Lebesgue measure. For $n\ge2$, $\Sph$ is the unit sphere, $d\omega$
is ordinary surface measure, and $|\Sph|$ is its total mass.
The notation $\fint$ denotes an integral divided by the measure of
its domain. For a bounded domain, $d_\Omega(x)=
\operatorname{dist}(x,\partial\Omega)$.

For a real function $v$, let $v_+=\max\{v,0\}$ and
$v_-=\max\{-v,0\}$, so $v=v_+-v_-$.

\Needspace{6\baselineskip}
For $0<\alpha<1$, we use
\[
 [v]_{C^\alpha(A)}=\sup_{\substack{x,y\in A\\x\ne y}}
 \frac{|v(x)-v(y)|}{|x-y|^\alpha},\qquad
 \|v\|_{C^\alpha(A)}=\|v\|_{L^\infty(A)}+[v]_{C^\alpha(A)}.
\]

At higher noninteger orders we use the usual H\"older norms of the
corresponding derivatives. Unspecified constants $C,c>0$ may change
from line to line; subscripts record additional dependence.
We write $A\asymp B$ when $cB\le A\le CB$. In universal estimates
the constants depend only on the parameters stated, never on the
nonlinearity or on a qualitative local bound for the solution.

The Euclidean Fourier transform is unitary:
\begin{equation}\label{prelim-euclidean-fourier}
 \widehat v(\xi)=(2\pi)^{-n/2}
 \int_{\R^n}e^{-\ii x\cdot\xi}v(x)\dd x.
\end{equation}

The operator $L=(-\Delta)^s$ is normalized to have
Fourier symbol $|\xi|^{2s}$. Thus its quadratic form is
\begin{equation}\label{energy-definition}
 [v]_{H^s(\R^n)}^2=\frac{c_{n,s}}2\iint_{\R^n\times\R^n}
       \frac{|v(x)-v(y)|^2}{|x-y|^{n+2s}}\dd x\dd y,
 \qquad
 c_{n,s}=\frac{4^s\Gamma(n/2+s)}{\pi^{n/2}|\Gamma(-s)|}.
\end{equation}
For a measurable set $A$, the seminorm $[v]_{H^s(A)}$ uses the
same normalization, with the integral restricted to $A\times A$.

We also use the weighted $L_s^1$ norm
\[
 \|u\|_{L_s^1(\R^n)}=\int_{\R^n}\frac{|u(x)|}{1+|x|^{n+2s}}\dd x,
 \qquad L_s^1(\R^n)=\{u:\|u\|_{L_s^1(\R^n)}<\infty\}.
\]
This weight is equivalent to the convention in \cite[(1.4)]{CS}.

A locally bounded function $u\in L_s^1(\R^n)$ is a distributional solution
of \eqref{equation} if
\[
 \int_{\R^n}u(-\Delta)^s\phi=\int_{B_2}f(u)\phi
 \qquad\text{for every }\phi\in C_c^\infty(B_2).
\]

It is stable if
\begin{equation}\label{stability}
 \int_{B_2}f'(u)\phi^2\le[\phi]_{H^s(\R^n)}^2
 \qquad\text{for every }\phi\in C_c^\infty(B_2).
\end{equation}
Only interior variations are involved, and the exterior values of $u$
are subject to the weighted $L^1$ norm condition above.

For $v\in H^s(A)$ the full local norm is
$\|v\|_{H^s(A)}^2=\|v\|_{L^2(A)}^2+[v]_{H^s(A)}^2$.
For real functions, polarization of the full-space energy gives
\begin{equation}\label{prelim-bilinear-energy}
 \mathcal E_s(v,w)=\frac{c_{n,s}}2\iint_{\R^n\times\R^n}
 \frac{(v(x)-v(y))(w(x)-w(y))}{|x-y|^{n+2s}}\dd x\dd y.
\end{equation}
Thus $\mathcal E_s(v,v)=[v]_{H^s(\R^n)}^2$ and, for smooth tests,
$\mathcal E_s(v,w)=\int vLw$. Plancherel gives
\[
 [v]_{H^s(\R^n)}^2=\int_{\R^n}|\xi|^{2s}|\widehat v(\xi)|^2\dd \xi.
\]

Under translation and dilation the weighted norm takes the form
\begin{equation}\label{rescaled-tail-definition}
 \|u\|_{L_s^1(a,R)}=R^{2s}\int_{\R^n}
 \frac{|u(y)|}{R^{n+2s}+|y-a|^{n+2s}}\dd y,
 \qquad \|u(a+R\,\cdot)\|_{L_s^1(\R^n)}=\|u\|_{L_s^1(a,R)}.
\end{equation}

For the spherical and cylinder constructions, assume $n\ge2$. We recall the
Dirichlet-to-Neumann operator of the ball. For $g\in C^\infty(\Sph)$, let $H_g$ solve
$\Delta H_g=0$ in $B_1$ with boundary value $g$, and define
\[
 \Lambda g=\partial_\nu H_g\big|_{\Sph}.
\]
Thus $\Lambda$ gives the outward normal derivative of the harmonic
replacement from its boundary values.

Let $\mathcal H_j$ be the restrictions to $\Sph$ of harmonic
homogeneous polynomials of degree $j$. With ordinary surface measure
$d\omega$, choose an orthonormal basis
$Y_{j,\ell}$, $1\le\ell\le\dim\mathcal H_j$. Orthogonality and
completeness give
\[
 L^2(\Sph)=\widehat{\bigoplus}_{j\ge0}\mathcal H_j,
 \qquad g=\sum_{j,\ell}g_{j,\ell}Y_{j,\ell},
 \qquad \|g\|_2^2=\sum_{j,\ell}|g_{j,\ell}|^2;
\]
see \cite[Proposition~5.9 and Theorem~5.12]{ABR}.

The polar-coordinate identity
\[
 \Delta=\partial_r^2+\frac{n-1}{r}\partial_r
                    +\frac1{r^2}\Delta_{\Sph}
\]
applied to the harmonic polynomial $r^jY_{j,\ell}(\omega)$ shows that
\[
 -\Delta_{\Sph}Y_{j,\ell}=j(j+n-2)Y_{j,\ell},\qquad
  H_{Y_{j,\ell}}(r\omega)=r^jY_{j,\ell}(\omega),\qquad
 \Lambda Y_{j,\ell}=jY_{j,\ell}.
\]
The polar formula and this eigenvalue computation are also recorded in
\cite[Lemma~4.1 and Theorem~4.5]{DX}. In particular,
\[
 -\Delta_{\Sph}=\Lambda(\Lambda+n-2),\qquad
 \Lambda=\sqrt{-\Delta_{\Sph}+(n-2)^2/4}-(n-2)/2.
\]
The square root is defined spectrally on this orthonormal basis.

For $0<\rho<1$, the Poisson formula gives
\begin{equation}\label{spherical-poisson}
 \begin{split}
 P_\rho g(\omega):=H_g(\rho\omega)
 &=\int_{\Sph}\frac{1-\rho^2}
 {|\Sph|\,|\omega-\rho\eta|^n}\,g(\eta)\dd \eta\\
 &=\sum_{j,\ell}\rho^j g_{j,\ell}Y_{j,\ell}(\omega).
 \end{split}
\end{equation}
The factor $|\Sph|^{-1}$ corresponds to our unnormalized surface
measure. The kernel is positive and has integral one, so $P_\rho$
preserves positivity and constants. Its spectral formula gives
$P_\rho P_\sigma=P_{\rho\sigma}$. In the time parameter
$h=-\log\rho>0$ we therefore write
\[
 \mathsf P_h=P_{e^{-h}}=e^{-h\Lambda}.
\]
These formulas, with the same surface-measure convention, appear in
\cite[Section~4, (4.3)--(4.5)]{DNS}; see also
\cite[(1.15)--(1.16), Theorem~1.17 and Proposition~1.20]{ABR}.

For example, on the circle ($n=2$), if
$g(\theta)=\sum_{k\in\mathbb Z}g_ke^{\ii k\theta}$, then
\[
 H_g(\rho,\theta)=\sum_{k\in\mathbb Z}\rho^{|k|}g_ke^{\ii k\theta},
 \qquad \Lambda g=\sum_{k\in\mathbb Z}|k|g_ke^{\ii k\theta}.
\]
The degree $j$ in higher dimensions plays the role of $|k|$ here.

Set $r=|x|$, $t=\log r$, and $x=r\omega$. The logarithmic cylinder is
$\Cyl=\R\times\Sph$, with points $X=(t,\omega)$ and measure
$dX=dt\dd \omega=r^{-n}dx$.
Its Laplacian is $\Delta_{\Cyl}=\partial_t^2+\Delta_{\Sph}$, and
$\partial_t=r\partial_r$. We reserve $h>0$ for semigroup time,
distinct from the spatial coordinate $t$.

For $v(t,\omega)=\sum_{j,\ell}v_{j,\ell}(t)Y_{j,\ell}(\omega)$,
\[
 \widehat v_{j,\ell}(\tau)=\frac1{\sqrt{2\pi}}
 \int_\R e^{-\ii t\tau}v_{j,\ell}(t)\dd t.
\]
This is the unitary Fourier transform in $t$, equivalently the Mellin
transform in $r$ with measure $dr/r$.

For every $\rho\in\R$, set
\begin{equation}\label{cylinder-bessel-norm}
 \|v\|_{H^\rho(\Cyl)}^2
 =\sum_{j,\ell}\int_\R
 (1+j(j+n-2)+\tau^2)^\rho
 |\widehat v_{j,\ell}(\tau)|^2\dd \tau.
\end{equation}
We abbreviate $\langle(j,\tau)\rangle=(1+j^2+\tau^2)^{1/2}$ and
$\langle a\rangle=(1+a^2)^{1/2}$; replacing $j(j+n-2)$ by $j^2$
in \eqref{cylinder-bessel-norm} gives an equivalent norm. More precisely,
with $C_n=\max\{1,n-1\}$,
\begin{equation}\label{cylinder-weight-comparison}
 1+j^2+\tau^2\le1+j(j+n-2)+\tau^2
                 \le C_n(1+j^2+\tau^2),\qquad j\ge0.
\end{equation}

The
$L^2$ inner product and its Sobolev-duality extension are denoted by
$\langle\cdot,\cdot\rangle$, using the convention
$\langle v,w\rangle=\int\overline v\,w$. Unless another domain is
displayed, norms and pairings in a cylinder argument use $\Cyl$.
We write $\sigma=1-s$, so $H^{s-1}=H^{-\sigma}$.

For an operator $\mathcal Q$, $\Ree\mathcal Q=
(\mathcal Q+\mathcal Q^*)/2$ in quadratic forms; for scalar symbols,
$\Ree$ and $\Imm$ have their ordinary complex meaning.

For a nonnegative selfadjoint operator $\mathcal Q$, the spectral
theorem and the Gamma integral give
\begin{align}
 \mathcal Q^sv
 &=c_s\int_0^\infty(v-e^{-h\mathcal Q}v)\frac{dh}{h^{1+s}},
 &c_s&=\frac{s}{\Gamma(1-s)},\label{prelim-positive-power}\\
 \mathcal Q^{-\rho}v
 &=\frac1{\Gamma(\rho)}\int_0^\infty
 h^{\rho-1}e^{-h\mathcal Q}v\dd h,
 &\rho&>0.\label{prelim-negative-power}
\end{align}
The first formula holds on the operator domain; the second is used
when $\mathcal Q$ is bounded below by a positive constant. If its heat
semigroup preserves nonnegativity, so do these inverse powers.
For nonnegative $v,w\in D(\mathcal Q^{s/2})$ with disjoint supports,
the fractional-power formula gives the form inequality
$\langle v,\mathcal Q^sw\rangle\le0$; the term $\langle v,w\rangle$
vanishes and each heat pairing is nonnegative. These formulas are
\cite[(1.11) and (2.3)]{ST}.

The Caffarelli--Silvestre extension \cite[\S3]{CaSi} solves
$\operatorname{div}(y^{1-2s}\nabla U)=0$ in $\R^{n+1}_+$ with trace
$U(x,0)=v(x)$. With our Fourier normalization,
\begin{equation}\label{prelim-extension-normalization}
 (-\Delta)^sv=-d_s\lim_{y\downarrow0}y^{1-2s}\partial_yU,
 \qquad d_s=2^{2s-1}\frac{\Gamma(s)}{\Gamma(1-s)};
\end{equation}
see \cite[Theorem~1.1, (1.8)]{ST}.

Finally, for $n>2s$, the sharp fractional Hardy inequality is
\begin{equation}\label{prelim-hardy-inequality}
 H_{n,s}\int_{\R^n}|x|^{-2s}\phi^2\dd x
 \le[\phi]_{H^s(\R^n)}^2\qquad(\phi\in C_c^\infty(\R^n)).
\end{equation}
Here $H_{n,s}$ is the constant in \eqref{hardy-constants}; see
\cite[(2.1)--(2.2) and Proposition~4.1]{Hardy}. Dilation preserves
the quotient, so the same optimal constant is obtained with tests
supported in any fixed ball centered at the origin.

\subsection{Localization and linear regularity}\label{prelim-linear}

\begin{lemma}[Fractional Poincar\'e and localization]
\label{prelim-localization}
Let $A\subset\R^n$ be bounded and open with positive measure. For $v\in H^s(A)$,
\[
 \|v\|_{H^s(A)}\le C_A\bigl([v]_{H^s(A)}+\|v\|_{L^1(A)}\bigr).
\]
If $\eta\in C_c^\infty(A)$, extend $\eta v$ by zero outside $A$.
Then $\|\eta v\|_{H^s(\R^n)}\le C_{A,\eta}\|v\|_{H^s(A)}$.
For functions supported in a fixed compact subannulus of
$\R^n\setminus\{0\}$, the change $x=e^t\omega$ identifies the
Euclidean and cylinder $H^s$ norms up to fixed constants. By duality,
the corresponding localized negative-order norms are equivalent too.
If $\varphi$ is a smooth bounded function with bounded first derivatives,
then, on either full space,
\begin{equation}\label{quantitative-multiplier}
 \|\varphi v\|_{H^\rho}
 \le C\bigl(\|\varphi\|_\infty+\|\nabla\varphi\|_\infty\bigr)
             \|v\|_{H^\rho},\qquad -1\le\rho\le1.
\end{equation}
For $0<s<1$, positive and negative parts preserve $H^s$ on either
space, with a norm bound depending only on $n,s$ and the chosen
equivalent norms.
Moreover, if $a,b\in L^2(\Cyl)$ are nonnegative and $\rho>0$, then
\begin{equation}\label{bessel-positive-pairing}
 \langle a,(1-\Delta_{\Cyl})^{-\rho}b\rangle\ge0.
\end{equation}
\end{lemma}
\begin{proof}
Write $v_A=\fint_Av$. The identity
$\int_A|v-v_A|^2=(2|A|)^{-1}\iint_{A\times A}|v(x)-v(y)|^2$
and $|x-y|\le\operatorname{diam}(A)$ bound the mean-zero part by the
fractional seminorm. The mean is controlled by the $L^1$ norm.

For localization, expand
$\eta(x)v(x)-\eta(y)v(y)$ into an increment of $v$ and an increment
of $\eta$. The latter is bounded by
$C\min\{|x-y|,1\}$, and its squared kernel is integrable because
$s<1$. Interactions with $A^c$ are bounded by the positive distance
from $\operatorname{supp}\eta$ to $A^c$. For cutoffs supported in a
fixed $K\Subset A$, this proof bounds the localization constant by
$C_{A,K}(\|\eta\|_\infty+\|\nabla\eta\|_\infty)$.

We prove the fractional norm assertions by one interpolation
calculation. Let $B$ be a positive selfadjoint operator on a Hilbert
space, bounded below by a positive constant, and put
\[
 K_B(t,f)^2=\inf_{f=f_0+f_1}
       \bigl(\|f_0\|_2^2+t^2\|B^{1/2}f_1\|_2^2\bigr),
 \qquad f_1\in D(B^{1/2}).
\]
Minimization at each spectral value $\lambda$ gives the multiplier
$t^2\lambda/(1+t^2\lambda)$. Substituting $u=t\sqrt\lambda$ and
using Tonelli therefore gives, for $0<\rho<1$,
\begin{equation}\label{quadratic-interpolation}
 \int_0^\infty t^{-2\rho}K_B(t,f)^2\frac{\dd t}{t}
   =c_\rho\|B^{\rho/2}f\|_2^2,
 \qquad
 c_\rho=\int_0^\infty\frac{u^{1-2\rho}}{1+u^2}\dd u\in(0,\infty).
\end{equation}
A linear map bounded on the two underlying $L^2$ spaces and their
form domains is bounded on the corresponding fractional domains:
apply it to each decomposition defining $K_B$, bound both terms by
the larger endpoint norm, and integrate this inequality.

For $B=1-\Delta$ on $\R^n$ or $\Cyl$, its form norm is the full
$H^1$ norm and its fractional norm is the spectral $H^\rho$ norm.
On $\R^n$ this is equivalent to our energy-plus-$L^2$ norm because
$(1+|\xi|^2)^\rho\asymp1+|\xi|^{2\rho}$.
The $L^2$ multiplier bound and the $H^1$ product rule, followed by
\eqref{quadratic-interpolation}, prove \eqref{quantitative-multiplier}
for $0\le\rho\le1$. Both endpoint bounds have first-degree
dependence on $\|\varphi\|_\infty+\|\nabla\varphi\|_\infty$;
duality gives the same dependence for negative orders.

For the coordinate change, use a smooth annular cutoff equal to one
on the prescribed compact support. Pullback by $x=e^t\omega$ with
this cutoff and the reverse coordinate map are bounded on $L^2$
and $H^1$: the Jacobian and $r,r^{-1}$ are bounded there, and
\[
 \partial_t u(e^t\omega)=r\omega\cdot\nabla u(r\omega),\qquad
 |\nabla_{\Sph}u(r\,\cdot)|\le r|\nabla u(r\omega)|.
\]
The reverse derivative formula has the factor $r^{-1}$.
Their compositions are the identity on functions with the stated
support. Equation~\eqref{quadratic-interpolation} proves the
annular $H^s$ equivalence in both directions. For negative orders,
pass to duality: the pulled-back test includes the smooth Jacobian
$r^n$ or its inverse. Multiplication by either localized Jacobian
is bounded by \eqref{quantitative-multiplier}, so the same norm
equivalence follows.

For real $v$, taking real parts does not increase either term in
the decomposition defining $K_{1-\Delta}$, so we may take $v_0,v_1$
real. For positive parts write
\[
 v_+=g_0+g_1,\qquad g_1=(v_1)_+,\qquad g_0=v_+-(v_1)_+.
\]
The scalar contraction and the $H^1$ chain rule give
$\|g_0\|_2\le\|v_0\|_2$ and $\|g_1\|_{H^1}\le\|v_1\|_{H^1}$.
Thus $K_{1-\Delta}(t,v_+)\le K_{1-\Delta}(t,v)$, and
\eqref{quadratic-interpolation} gives the asserted $H^s$ bound.
The proof for negative parts is identical.

Finally, the heat semigroup of $1-\Delta_{\Cyl}$ preserves
nonnegativity. Indeed, for $g\ge0$ the resolvent minimizer of
\[
 \int_{\Cyl}\bigl(|\nabla w|^2+(1+\lambda)|w|^2-2gw\bigr)\dd X,
 \qquad \lambda>0,
\]
is unchanged by taking its positive part: this operation decreases
the functional by the $H^1$ chain rule and $g\ge0$, and the minimizer
is unique. Resolvent approximation of the exponential gives the
stated positivity of the heat semigroup. Formula~\eqref{prelim-negative-power}
then proves \eqref{bessel-positive-pairing}. In particular the
$H^{-\rho}$ norm squared is increasing on nonnegative $L^2$ functions
under pointwise order.
\end{proof}

When $Lu=g$ and $\eta^2u$ is an admissible test, polarization gives
\begin{equation}\label{cutoff-energy-identity}
 [\eta u]_{H^s(\R^n)}^2=\int g\,u\eta^2
 +\frac{c_{n,s}}2\iint
  \frac{u(x)u(y)(\eta(x)-\eta(y))^2}{|x-y|^{n+2s}}\dd x\dd y.
\end{equation}
The identity extends by energy approximation. If $u$ has only a
weighted $L^1$ tail, one point in each nonlocal pairing lies in the
cutoff support, and the exterior value occurs to the first power.

\begin{lemma}[A local source mass bound]\label{source-mass}
If $Lu=g\ge0$ in $B_2$, with $g\in L^1_{\mathrm{loc}}(B_2)$ and
$u\in L_s^1(\R^n)$, then, for each fixed $R<2$,
\[
 \int_{B_R}g\le C_R \|u\|_{L_s^1(\R^n)}.
\]
\end{lemma}
\begin{proof}
Choose $\zeta\in C_c^\infty(B_2)$ nonnegative and equal to one on
$B_R$. Testing the distributional equation gives
\[
 0\le\int_{B_R}g\le\int g\zeta=\int_{\R^n}uL\zeta
 \le C_R\|u\|_{L_s^1(\R^n)},
\]
because $|L\zeta(x)|\le C_R(1+|x|^{n+2s})^{-1}$.
\end{proof}

For $n>2s$, write the Riesz potential of a finite measure as
\begin{equation}\label{prelim-riesz-potential}
 I_{2s}\mu(x)=b_{n,s}\int_{\R^n}|x-y|^{2s-n}\dd \mu(y),\qquad
 b_{n,s}=\frac{\Gamma(n/2-s)}{4^s\pi^{n/2}\Gamma(s)}.
\end{equation}
Its Fourier multiplier is $|\xi|^{-2s}$, so $L(I_{2s}\mu)=\mu$
in distributions. Morrey growth gives the following H\"older estimate.

\begin{lemma}[Riesz potentials of Morrey measures]
\label{prelim-morrey-potential}
Let $n>2s$, $0<\alpha<\min\{1,2s\}$, and let $\mu\ge0$ be supported
in $B_1$. Suppose, for some $M<\infty$,
\[
 \mu(\R^n)\le M,\qquad
 \mu(B_\rho(a))\le M\rho^{n-2s+\alpha}
 \quad(a\in B_{1/2},\ 0<\rho\le1/4).
\]
Then $v=I_{2s}\mu$ satisfies
\[
 \|v\|_{C^\alpha(B_{1/4})}\le C_{n,s,\alpha}M,\qquad
 \|v\|_{L_s^1(\R^n)}\le C_{n,s}\mu(\R^n).
\]
\end{lemma}
\begin{proof}
For $x\in B_{1/2}$ and $r\le1/4$, dyadic shells give
\[
 \int_{|x-y|<r}|x-y|^{2s-n}\dd \mu(y)
 \le C\sum_{k\ge0}(2^{-k}r)^{2s-n}
                \mu(B_{2^{-k}r}(x))\le CMr^\alpha.
\]
The part outside that radius is bounded by the total mass. Thus
$\|v\|_{L^\infty(B_{1/2})}\le CM$.

For $x,x'\in B_{1/4}$ with
$h=|x-x'|<1/16$, the parts within $2h$ of either point cost
$CMh^\alpha$. On the complement the kernel difference is at most
$Ch|x-y|^{2s-n-1}$. Summing shells up to radius $1/4$ costs
\[
 CMh\sum_{2^kh\le1/4}(2^kh)^{\alpha-1}\le CMh^\alpha;
\]
the remaining part costs $Ch\mu(\R^n)$. The convergence uses
$\alpha<1$; larger $h$ follow from the supremum bound.

Finally, Fubini gives the weighted norm estimate because
\[
 \sup_{y\in B_1}\int_{\R^n}
 \frac{|x-y|^{2s-n}}{1+|x|^{n+2s}}\dd x<\infty.
\]
\end{proof}

We use the maximum and comparison principles for the fractional
Laplacian: a zero-exterior supersolution is nonnegative, and ordered
sources and exterior data give ordered solutions. In the $L^1$-weak
Dirichlet formulation these are the uniqueness and comparison results
of \cite[Propositions~C.1--C.2]{SPapprox}. For an $s$-harmonic
replacement in $B_R$, the exterior Poisson kernel is
\begin{equation}\label{prelim-fractional-poisson}
 \Pi_R(x,y)=p_{n,s}\frac{(R^2-|x|^2)^s}
 {(|y|^2-R^2)^s|x-y|^n}\mathbf1_{\{|y|>R\}},\qquad
 p_{n,s}=\frac{2\sin(\pi s)}{\pi|\mathbb S^{n-1}|},
\end{equation}
for $|x|<R$. Its mass in $y$ is one. The formula is
\cite[Theorem~2.10]{Bucur}; it is distinct from the classical
spherical Poisson operator in \eqref{spherical-poisson}.

\begin{lemma}[Interior regularity of distributionally harmonic functions]
\label{harmonic-weyl}
If $w\in L_s^1(\R^n)$ and $(-\Delta)^sw=0$ distributionally in $B_1$,
then $w$ has a smooth representative in $B_1$. For every nonnegative
integer $k$ and $0<a<1$,
\[
 \|w\|_{C^k(B_a)}\le C_{n,s,k,a}\|w\|_{L_s^1(\R^n)}.
\]
\end{lemma}
\begin{proof}
First suppose $w$ is smooth on $\R^n$, and fix $a<b<1$.
The ball Poisson formula \eqref{prelim-fractional-poisson} reads
\[
 w(x)=\int_{|y|>R}\Pi_R(x,y)w(y)\dd y,\qquad |x|<R<b.
\]
For the smooth $w$ now under consideration, first cut off its exterior
data at large radii. For each $K\Subset B_R$ and multi-index $\gamma$,
differentiating the explicit kernel gives
\[
 \sup_{x\in K}|\partial_x^\gamma\Pi_R(x,y)|
 \le\frac{C_{K,\gamma,R}}{1+|y|^{n+2s}}
 \qquad(|y|>2R+1).
\]
Thus the $C^k(K)$ norm of the removed tail tends to zero with
$\int_{|y|>T}|w(y)|(1+|y|^{n+2s})^{-1}\dd y$ as $T\to\infty$.
Without derivatives, the same kernel gives a bound on all of $B_R$
by that tail integral times $C_R(R^2-|x|^2)^s$. Hence the
bounded-data Poisson solutions converge with all derivatives on
compact subsets and the removed contribution tends uniformly to zero.

Since $w$ is smooth near $\partial B_R$, the limit
has its continuous boundary trace, equals $w$ outside $B_R$, and is
bounded inside. The difference from the smooth $w$ is therefore a
bounded continuous $s$-harmonic function, zero outside, and vanishes
by the maximum principle. The general distributional case is handled
by mollification below.

Average the formula against a smooth probability density supported in
an interval $(b_0,b_1)\Subset(a,b)$. Every derivative in $x\in B_a$
of the averaged kernel is bounded in absolute value by
$C_k(1+|y|^{n+2s})^{-1}$. In the bounded $y$ region this follows from
the integrability of $(|y|-R)^{-s}$ in $R$; for large $y$ it follows
from the decay $|y|^{-n-2s}$. Hence the asserted estimates hold in the
smooth case.

For a distributional solution, set $w_\epsilon=\rho_\epsilon*w$ with
a compactly supported nonnegative mollifier. Convolution commutes
with $(-\Delta)^s$ in distributions, so $w_\epsilon$ is smooth and
harmonic in $B_{1-\epsilon}$, while
\[
 \|w_\epsilon\|_{L_s^1(\R^n)}\le C \|w\|_{L_s^1(\R^n)}
 \qquad(\epsilon<1/2).
\]
The preceding averaged formula gives uniform bounds of every order
on each smaller ball. Compactness and distributional convergence to
$w$ give its smooth representative and the asserted estimates.
\end{proof}

\begin{lemma}[Local linear H\"older gain]\label{prelim-local-schauder}
Let $n>2s$, $u\in L_s^1(\R^n)$, and $Lu=g$ distributionally in
$B_1$. For $0<a<1$ the following estimates hold.
If $g\in L^\infty(B_1)$, then, for every noninteger $0<q<2s$,
\[
 \|u\|_{C^q(B_a)}\le C\bigl(\|u\|_{L_s^1(\R^n)}
                            +\|g\|_{L^\infty(B_1)}\bigr).
\]
If $g\in C^t(B_1)$, $0<t<1$, and $q=t+2s$ is noninteger, then
\[
 \|u\|_{C^q(B_a)}\le C\bigl(\|u\|_{L_s^1(\R^n)}
                            +\|g\|_{C^t(B_1)}\bigr).
\]
Here $C^q=C^{m,q-m}$ when $m=\lfloor q\rfloor$, and $C$ depends
only on the displayed exponents, $n,s,a$. The estimates select the
corresponding regular representative of $u$.
\end{lemma}
\begin{proof}
Choose $a<b<1$ and $\chi\in C_c^\infty(B_1)$ equal to one on
$B_b$. Extend $G=\chi g$ by zero and set $v=I_{2s}G$.
The kernel in \eqref{prelim-riesz-potential} gives
$Lv=G$ and
$\|v\|_{L_s^1}\le C\|G\|_{L^1}$, by the Fubini estimate in
Lemma~\ref{prelim-morrey-potential}, applied to $|G|$.
Thus $u-v$ is smooth in $B_b$, with all interior derivatives
controlled by the stated right side, by Lemma~\ref{harmonic-weyl}.

The gain for $v$ follows from a dyadic decomposition.
Choose smooth Fourier cutoffs $\psi_0,\psi$ such that
\[
 1=\psi_0(\xi)+\sum_{k\ge0}\psi(2^{-k}\xi),
 \qquad \operatorname{supp}\psi\subset\{1/2<|\xi|<2\},
\]
with $\psi_0$ supported in a fixed ball. The convolution kernel $K$ of the Fourier multiplier
$\psi(\xi)|\xi|^{-2s}$ is a Schwartz function
with integral zero. The high-frequency pieces of $v$ are
\[
 v_k=K_k*G,\qquad K_k(x)=2^{(n-2s)k}K(2^kx).
\]
For every derivative order $\ell\ge0$,
\[
 \|D^\ell v_k\|_\infty
 \le C_\ell 2^{(\ell-2s)k}\|G\|_\infty.
\]
If $G\in C^t$, the zero integral of $D^\ell K_k$ allows us
to subtract $G(x)$ in the convolution. Hence
\[
 \|D^\ell v_k\|_\infty
 \le [G]_{C^t}\int |D^\ell K_k(y)|\,|y|^t\dd y
 \le C_{\ell,t}2^{(\ell-2s-t)k}[G]_{C^t}.
\]
The low-frequency piece is smooth with bounded derivatives:
$|\xi|^{-2s}\widehat G(\xi)$ is integrable on its Fourier
support because $n>2s$ and $|\widehat G|\le C\|G\|_1$.

In either case, for the asserted noninteger $q=m+\gamma$,
$0<\gamma<1$, the estimates give
$\|D^\ell v_k\|_\infty\le C2^{(\ell-q)k}$.
For a displacement $h$, split the sum at $2^k|h|=1$.
For the lower frequencies use
$|h|\|D^{m+1}v_k\|_\infty$, and for the higher frequencies
use $2\|D^m v_k\|_\infty$. The two geometric sums are bounded
by $C|h|^\gamma$. The derivatives through order $m$ converge
uniformly. This proves the claimed $C^q$ bound for $v$, and
the harmonic remainder completes the proof.
\end{proof}

Extending a one-dimensional solution constantly in a second variable
preserves its equation and stability, as follows.

\begin{lemma}\label{prelim-lifting}
Let $0<s<1$, $u\in L_s^1(\R)$, and $U(x_1,x_2)=u(x_1)$.
Then $\|U\|_{L_s^1(\R^2)}\le C_s\|u\|_{L_s^1(\R)}$.
If $(-\Delta_{\R})^su=g$ distributionally in an interval $I$, with
$g\in L^1_{\mathrm{loc}}(I)$, then
$(-\Delta_{\R^2})^sU=g(x_1)$ distributionally in $I\times\R$.
For $g=f(u)$, stability on $I$ implies stability of $U$ on this strip.
\end{lemma}
\begin{proof}
The beta integral and the normalization in \eqref{energy-definition} give
\begin{equation}\label{prelim-lifting-identities}
 \begin{split}
 c_{2,s}\int_{\R}(a^2+z^2)^{-1-s}\dd z
   &=c_{1,s}|a|^{-1-2s}\qquad(a\ne0),\\
 \int_{\R}\frac{\dd x_2}{1+(x_1^2+x_2^2)^{1+s}}
   &\le\frac{C_s}{1+|x_1|^{1+2s}}.
 \end{split}
\end{equation}
For the second bound, scale when $|x_1|\ge1$ and use a fixed integrable
majorant when $|x_1|\le1$. It proves the tail estimate by Tonelli.

For $\phi\in C_c^\infty(I\times\R)$ set
$\Phi(x_1)=\int_\R\phi(x_1,x_2)\dd x_2$. Then
\begin{equation}\label{lifting-test-identity}
 \int_\R(-\Delta_{\R^2})^s\phi(x_1,x_2)\dd x_2
       =(-\Delta_{\R})^s\Phi(x_1).
\end{equation}
Indeed, use symmetric second differences in the singular integral.
For increments of length less than one their $L^1(dx_2)$ norm is
$O(|z|^2)$, and for larger increments it is bounded. Fubini is therefore
valid since $0<s<1$; integrating first in $x_2$ and then in the second
increment gives the first identity in \eqref{prelim-lifting-identities}.
Also $|(-\Delta_{\R^2})^s\phi(x)|\le C_\phi(1+|x|^{2+2s})^{-1}$.
The tail bound thus justifies pairing with $U$ and gives
\[
 \int_{\R^2}U(-\Delta_{\R^2})^s\phi
 =\int_\R u(-\Delta_{\R})^s\Phi
 =\int_I g\Phi=\int_{I\times\R}g(x_1)\phi.
\]

Finally apply one-dimensional stability to each slice of $\phi$.
Plancherel and $|\xi_1|^{2s}\le(\xi_1^2+\xi_2^2)^s$ give
\[
 \int_{\R}[\phi(\cdot,x_2)]_{H^s(\R)}^2\dd x_2
       \le[\phi]_{H^s(\R^2)}^2,
\]
which proves the stability assertion.
\end{proof}

The following bootstrap justifies the differentiations and energy tests below;
its local norms do not enter the universal constants.

\begin{lemma}\label{bootstrap}
Let $u\in L^\infty_{\mathrm{loc}}(B_2)\cap L_s^1(\R^n)$ solve
$(-\Delta)^su=f(u)$ distributionally, where
$f\in C^1_{\mathrm{loc}}(\R)$. Let $\varepsilon>0$ satisfy
$1+2s-\varepsilon>0$, and suppose this regularity order is noninteger. Then
\[
 u\in C^{1+2s-\varepsilon}_{\mathrm{loc}}(B_2).
\]
\end{lemma}
\begin{proof}
Local boundedness of $u$ and $f\in C^1_{\mathrm{loc}}(\R)$ make
$f(u)$ locally bounded and $f$ Lipschitz on each compact value interval.

For $n>2s$, Lemma~\ref{prelim-local-schauder}, after
translation and dilation, supplies each local gain on nested balls.
A bounded source gives every noninteger order below $2s$;
a $C^t$ source, $0<t<1$, gives order $t+2s$ when noninteger.

For $0<t<1$, a locally Lipschitz composition preserves $C^t$.
Iterating this regularity gain on nested balls, with sufficiently small losses
at integer exponents, first gives $u\in C^{1+\eta}$ for some
$\eta>0$. To see that this takes finitely many steps, as long as the
available exponent is below one, the next exponent may be chosen
larger by any fixed number smaller than $2s$; the last source
exponent can be chosen below one and above $1-2s$.

Now $f(u)$ is $C^1$ by the ordinary chain rule, hence locally
$C^{1-\varepsilon}$ for every $0<\varepsilon<1$. One final application gives
the claimed regularity; lower orders follow by inclusion.

For the remaining case $n=1\le2s$, set $U(x_1,x_2)=u(x_1)$.
Lemma~\ref{prelim-lifting} gives the distributional equation and
finite tail in two dimensions.
Apply the preceding $n>2s$ argument to $U$ on interior balls
and restrict to $x_2=0$. This proves the stated regularity also
in this remaining case.
\end{proof}

\begin{lemma}[Iteration on nested radii]\label{prelim-iteration}
Let $\Phi$ be nonnegative and bounded on $[r_0,r_1]$. Suppose
\[
 \Phi(r)\le a\Phi(R)+B(R-r)^{-p}
 \quad(r_0\le r<R\le r_1),\qquad 0<a<2^{-p},\quad p>0.
\]
Then $\Phi(r_0)\le C_{a,p}B(r_1-r_0)^{-p}$.
\end{lemma}
\begin{proof}
Take $r_k=r_1-(r_1-r_0)2^{-k}$. Iteration gives a geometric series
with ratio $a2^p<1$; the remainder $a^k\Phi(r_k)$ tends to zero
by boundedness.
\end{proof}

\section{Harmonic replacement and fractional extensions}\label{part:pde}
The stability test will be built from the normal-derivative difference
under harmonic replacement. After positive averaging, control of this
derivative becomes a weighted estimate for $f(u)$, from which we will
deduce source-mass decay and H\"older continuity. We construct the
corresponding positive operator in polar coordinates and its fractional
extension; these provide the energy identities and cutoff estimates
needed for the stability argument. Throughout this section $n\ge2$.

\subsection{The corrected radial derivative}\label{calculus-section}\label{field-construction}
We use the ball Dirichlet-to-Neumann operator $\Lambda$ and its
Poisson semigroup from Section~\ref{prelim-sphere}. The radial derivative
measures the response to dilations, while
$\Lambda$ measures the same response for harmonic replacement.
Their difference defines the corrected radial derivative $D_{\mathrm{har}}$. Put
$r=|x|$, $t=\log r$, and
\[
 \partial_t=r\partial_r,\qquad D_{\mathrm{har}}=\Lambda-\partial_t.
\]
For one angular mode, a direct calculation yields
\[
 D_{\mathrm{har}}[\phi(r)Y_{j,\ell}]=(j\phi-r\phi')Y_{j,\ell},
 \qquad D_{\mathrm{har}}[r^jY_{j,\ell}]=0.
\]
Thus $D_{\mathrm{har}}$ removes every regular homogeneous harmonic mode; in
particular it annihilates constants and linear functions. On radial
functions $D_{\mathrm{har}}u=-ru_r$.

More generally, let $H_r$ be the classical
harmonic replacement of $u$ in $B_r$. Rescaling the unit-ball problem
gives $r\partial_\nu H_r(r\omega)=\Lambda[u(r\,\cdot)](\omega)$,
and hence
\begin{equation}\label{harmonic-replacement-field}
 D_{\mathrm{har}}u(r\omega)=r\bigl(\partial_\nu H_r(r\omega)
                         -\partial_\nu u(r\omega)\bigr).
\end{equation}
Thus $D_{\mathrm{har}}u$ is the normal-derivative difference between $u$ and its
harmonic replacement, multiplied by the radius. It is determined by
the values of $u$ on and inside the same ball: the derivative
$\partial_\nu u$ is the boundary limit taken from inside $B_r$.

For smooth $u$, Green's
formula expresses this difference as a weighted average of the
classical source:
\[
 D_{\mathrm{har}}u(x)=\frac1{|\Sph|}\int_{B_r}
       \frac{r^2-|y|^2}{|x-y|^n}(-\Delta u(y))\dd y,
 \qquad |x|=r.
\]
Indeed, the Dirichlet Green function of the ball satisfies
\[
 -\partial_{\nu_x}G_{B_r}(x,y)
 =\frac{r^2-|y|^2}{|\Sph|r|x-y|^n},\qquad |x|=r.
\]
Apply this identity to $u-H_r$, whose boundary values vanish, and
multiply its outward normal derivative by $-r$.
For the fractional equation, a positive average of this derivative
will give the corresponding average of the fractional source.

The Fourier transform in $t$ and the spherical harmonic decomposition
diagonalize $\partial_t$, $\Lambda$, and $D_{\mathrm{har}}$. With the
conventions of Section~\ref{prelim-cylinder}, for
$v\in C_c^\infty(\Cyl)$ we have
\[
 \begin{aligned}
 \widehat{\partial_t v}_{j,\ell}(\tau)
   &=\ii\tau\,\widehat v_{j,\ell}(\tau),\\
 \widehat{\Lambda v}_{j,\ell}(\tau)
   &=j\,\widehat v_{j,\ell}(\tau),\\
 \widehat{D_{\mathrm{har}}v}_{j,\ell}(\tau)
   &=(j-\ii\tau)\widehat v_{j,\ell}(\tau).
 \end{aligned}
\]
Since $|j-\ii\tau|^2\le1+j(j+n-2)+\tau^2$, these formulas give
$D_{\mathrm{har}}:H^\rho(\Cyl)\to H^{\rho-1}(\Cyl)$ for every $\rho\in\R$.

On a fixed annulus the Euclidean and cylinder Sobolev norms are
equivalent. The annular energy and weighted tail of $u$ therefore
control $D_{\mathrm{har}}u$ in the negative Sobolev space $H^{s-1}$,
since $s<1$.

\begin{lemma}\label{corrected-derivative-origin}
Let $v\in C_c^{1,\rho}(\R^n)$, where $s<\rho\le1$, and put
$q_v=D_{\mathrm{har}}v$. Then $q_v$ is continuous on $\Cyl$,
belongs to $H^s(\Cyl)$, and satisfies
\begin{equation}\label{corrected-derivative-decay}
 \|q_v\|_{H^s((\log r,\log(2r))\times\Sph)}
 \le C_{n,s,\rho}[\nabla v]_{C^\rho(B_1)}r^{1+\rho},
 \qquad 0<r\le\tfrac14.
\end{equation}
Smooth cutoffs at the inner cylinder end, with a fixed transition
profile, converge to $q_v$ in $H^s$.
For data with a common compact support, convergence in $C^{1,\rho}$
implies convergence of the corrected derivatives in $H^s(\Cyl)$.
\end{lemma}
\begin{proof}
The affine Taylor polynomial is annihilated by $D_{\mathrm{har}}$.
On $A=\{1/2<|z|<4\}$ set
\[
 v_r(z)=v(rz)-v(0)-r\nabla v(0)\cdot z.
\]
Taylor's theorem and the H\"older bound for $\nabla v$ give
\[
 \|v_r\|_{C^{1,\rho}(\overline A)}
 \le C[\nabla v]_{C^\rho(B_1)}r^{1+\rho}.
\]
For a first derivative the fractional difference integral on this
fixed annulus is bounded by a constant times
$\int_0^1 h^{2\rho-2s-1}\dd h<\infty$, also when $\rho=1$.
The annular norm equivalence in Lemma~\ref{prelim-localization}
therefore controls the cylinder $H^s$ norms of $v_r$ and its first
$t$ and rotation derivatives. These control its localized $H^{1+s}$ norm:
on the Fourier and spherical-harmonic side, taking the first $t$
and rotation derivatives adds the weight
$1+\tau^2+j(j+n-2)$. Using a fixed radial cutoff on the enlarged
annulus and the bound $D_{\mathrm{har}}:H^{1+s}\to H^s$ gives
\eqref{corrected-derivative-decay} on $1<|z|<2$ after dilation.

To sum these bounds, take a smooth partition $\chi_k$ in $t$ by
translates of a fixed compact cutoff, and wider cutoffs
$\widetilde\chi_k$ equal to one on $\operatorname{supp}\chi_k$.
Both families have finite overlap. The map
$(g_k)\mapsto\sum_k\widetilde\chi_k g_k$ is bounded from the square
sum of full-cylinder $L^2$, respectively $H^1$, spaces to the global
space, by finite overlap and the product rule.
Equation~\eqref{quadratic-interpolation}, applied to the direct sum
of $1-\Delta_{\Cyl}$, gives the same bound for $H^s$.

Apply it to $g_k=\chi_kq_v$. The square-summable powers
$r^{1+\rho}$ prove the
global $H^s$ assertion and convergence of the inner cutoff tails.
On the remaining finitely many slabs use the local $C^{1,\rho}$
norm; the outer end vanishes because $v$ is compactly supported.
This also proves continuity of the map from $C^{1,\rho}$ data with
common compact support to $H^s$ corrected derivatives.

For pointwise continuity, differentiation of the ball Poisson formula
gives, for $g\in C^{1,\rho}(\Sph)$,
\[
 \Lambda g(\omega)=\frac2{|\Sph|}\operatorname{PV}
  \int_{\Sph}\frac{g(\omega)-g(\eta)}{|\omega-\eta|^n}\dd\eta.
\]
Subtract the tangential linear Taylor term, whose integral vanishes
by reflection symmetry. The remaining integral over a neighborhood
of radius $\epsilon$ is bounded by
$C\|g\|_{C^{1,\rho}}\int_0^\epsilon h^{\rho-1}\dd h$ and tends
uniformly to zero. Away from that neighborhood the kernel is smooth.
This proves continuity of $\Lambda g$, locally uniformly for the
sphere restrictions of $v$. The radial derivative is continuous as
well, proving continuity of $q_v$ on compact cylinder annuli.
\end{proof}

The corrected radial derivative relates the Euclidean Laplacian to an
operator whose angular eigenvalues are those of a sphere in two more
dimensions. Put
\[
 \mathcal R=r^{-2}(\Lambda-r\partial_r),\qquad
 d\mu=r^{n+1}\dd r\dd\omega,
\]
and let $\mathcal L$ be the Friedrichs operator on $L^2(d\mu)$ associated
with the closure of
\begin{equation}\label{cone-form}
 \mathfrak a[w]=\int_0^\infty\!\int_{\Sph}
 \left(|\partial_r w|^2+r^{-2}|\nabla_{\Sph}w|^2
          +2r^{-2}\overline w\Lambda w\right)
                              r^{n+1}\dd\omega\dd r
\end{equation}
on smooth compact annular functions. On those functions its expression is
\[
 \mathcal L=-\partial_{rr}-\frac{n+1}{r}\partial_r
                   +r^{-2}\Lambda(\Lambda+n).
\]

\begin{lemma}\label{heat-intertwining-lemma}
For $v\in C_c^\infty(\R^n)$,
\begin{equation}\label{radial-derivative-isometry}
 \|\mathcal Rv\|_{L^2(d\mu)}=\|\nabla v\|_{L^2(\R^n)}.
\end{equation}
Moreover, $\mathcal Rv\in D(\mathcal L)$ and
\begin{align}
 \mathcal R e^{h\Delta}v
     &=e^{-h\mathcal L}\mathcal Rv,\qquad h\ge0,
       \label{heat-intertwining}\\
 \mathcal L^s\mathcal Rv
     &=\mathcal R(-\Delta)^sv.
       \label{heat-fractional-intertwining}
\end{align}
The last identity holds in $L^2(d\mu)$ and in distributions away
from the origin.
\end{lemma}
\begin{proof}
Expand the square in the left side of
\eqref{radial-derivative-isometry} and integrate its mixed term in $r$.
The identity $\Lambda(\Lambda+n-2)=-\Delta_{\Sph}$ gives
\[
 \begin{split}
 \|\mathcal Rv\|_{L^2(d\mu)}^2
 &=\int_0^\infty\!\int_{\Sph}|\partial_rv|^2r^{n-1}
                                      \dd\omega\dd r\\
 &\quad+\int_0^\infty r^{n-3}
       \langle v,\Lambda(\Lambda+n-2)v\rangle_{L^2(\Sph)}\dd r
 =\|\nabla v\|_2^2.
 \end{split}
\]
The boundary term at zero vanishes because $\Lambda$ annihilates
the constant Taylor term and the remaining angular modes are $O(r)$;
the boundary term at infinity vanishes by compact support.
Thus $\mathcal R:H^1(\R^n)\to L^2(d\mu)$ is continuous.

Direct differentiation in angular degree $j$ gives
\begin{equation}\label{local-heat-intertwining}
 \mathcal L\mathcal Rv=\mathcal R(-\Delta)v.
\end{equation}
We check the operator domain needed here. If $v$ is a Schwartz function,
its Taylor polynomial is a sum of terms $r^{j+2k}Y_j$, and
\[
 \mathcal R(r^{j+2k}Y_j)=-2k r^{j+2k-2}Y_j.
\]
The terms with $k=0$ disappear and all remaining terms are smooth at
zero. The bound $\Lambda:H^{m+1}(\Sph)\to H^m(\Sph)$, read directly
from its eigenvalues, controls the differentiated Taylor remainders.
Consequently both $\mathcal Rv$ and $\mathcal R(-\Delta)v$ are smooth
at zero and decrease rapidly at infinity. In particular,
$\mathfrak a[\mathcal Rv]<\infty$. Cutting off on $r<2\varepsilon$
costs $O(\varepsilon^n)$ in radial gradient energy; the remaining form
terms converge by integrability. An outer cutoff then proves membership
in the form domain. Equation~\eqref{local-heat-intertwining}, first
tested on compact annuli and then extended by form density, proves
$\mathcal Rv\in D(\mathcal L)$ with the asserted operator value.

For $v\in C_c^\infty$, put $Q(h)=\mathcal R e^{h\Delta}v$.
The isometry and the $H^1$ continuity of the Euclidean heat flow give
strong continuity in $L^2(d\mu)$ at $h=0$ and allow its differentiation.
The domain assertion just proved gives
$Q'=-\mathcal LQ$ and $Q(0)=\mathcal Rv$.
Uniqueness for this strong $L^2$ evolution proves
\eqref{heat-intertwining}.

Finally,
\[
 \|v-e^{h\Delta}v\|_{H^1}
 \le\min\{h\|\Delta v\|_{H^1},\,2\|v\|_{H^1}\}.
\]
The positive-power integral \eqref{prelim-positive-power} therefore
converges absolutely in $H^1(\R^n)$. Its Euclidean Fourier
transform identifies its value with $(-\Delta)^sv$ in that norm.
Apply the continuous map $\mathcal R$, use
\eqref{heat-intertwining}, and apply the same positive-power formula
to $\mathcal L$ to obtain \eqref{heat-fractional-intertwining}.
The continuous extension of $\mathcal R$ agrees with its differential
expression on every compact annulus, which also gives the
distributional identity there.
\end{proof}

\subsection{The weighted extension problem}

To express the fractional operators in cylinder coordinates, put
$\beta=(n-2s)/2>0$. An admissible radial weight means
\begin{equation}\label{weight-interval}
 2s-n<\vartheta<\min\{1,2s\}.
\end{equation}
Write $\gamma=\beta-1+\vartheta$. Then
\begin{equation}\label{extension-gamma-interval}
 -\beta-1<\gamma<\beta,\qquad |\gamma|<\beta+1.
\end{equation}

We define the cylinder operators through one-dimensional weighted
extension problems. Set
\[
 p(\varphi)=\sin^{1-2s}\varphi\cos^{n-1}\varphi,\qquad
 \lambda_m=m(m+n-2),\qquad 0<\varphi<\frac\pi2.
\]
Let $\mathcal V_m$ be the completion of smooth profiles, regular at the
axis $\varphi=\pi/2$, in the norm whose square is
\[
 \int_0^{\pi/2}p\left(|v'|^2+
           (1+\lambda_m\sec^2\varphi)|v|^2\right)\dd\varphi.
\]
Here a regular profile has the form
$(\cos\varphi)^m\psi(\cos^2\varphi)$ near the axis, with $\psi$ smooth.
The profiles are smooth up to $\varphi=0$, with no boundary condition
there. Write $\mathcal V_m^0=\{v\in\mathcal V_m:v(0)=0\}$ and
$\|v\|_p^2=\int_0^{\pi/2}p|v|^2$. For real $x$ define
\begin{equation}\label{dn-form-definition}
 E_{m,x}[v]=\int_0^{\pi/2}p
     \left(|v'|^2+(\lambda_m\sec^2\varphi+x)|v|^2\right)\dd\varphi,
\end{equation}
and denote the corresponding sesquilinear form by $E_{m,x}(v,w)$,
conjugate-linear in its first argument. The same convention is used
for $\langle v,w\rangle_p=\int p\overline v w$.

\begin{lemma}[The weighted angular form]\label{angular-groundstate}
Trace at zero is continuous on $\mathcal V_m$, and
$\mathcal V_{m+1}\subset\mathcal V_m$. More precisely, for $K\ge1$,
\begin{equation}\label{angular-scaled-trace}
 K^{2s}|v(0)|^2
 \le C_{n,s}\int_0^{1/(2K)}p\bigl(|v'|^2+K^2|v|^2\bigr)\dd\varphi.
\end{equation}
Put
$h_m(\varphi)=(\cos\varphi)^m$ and $q_m=m(m+n-2s)$.
For every real $x$ and every $v\in\mathcal V_m$,
\begin{equation}\label{angular-groundstate-identity}
 E_{m,x}[v]
 =\int_0^{\pi/2}p h_m^2
             \left|\left(\frac{v}{h_m}\right)'\right|^2\dd\varphi
      +(q_m+x)\|v\|_p^2.
\end{equation}
The quotient derivative is understood locally in $(0,\pi/2)$ and
then in the weighted energy on the right. In particular,
$E_{m,x}$ is coercive in the defining norm of $\mathcal V_m$ when
$q_m+x>0$.

For $m=j+1$, $x=\beta^2+\tau^2-\gamma^2$, and a weight in
\eqref{weight-interval}, one has the norm equivalence
\begin{equation}\label{dn-coercivity}
 E_{m,x}[v]\asymp\int_0^{\pi/2}p\left(|v'|^2+
       (1+\lambda_m\sec^2\varphi+\tau^2)|v|^2\right)\dd\varphi.
\end{equation}
The same bounds hold for $m=j$ and $x=\beta^2+\tau^2$.
The constants are independent of $j,\tau$ and locally uniform in
$s$ and in the admissible weight.
\end{lemma}
\begin{proof}
For $0<\varphi\le1/2$, weighted Cauchy--Schwarz gives
\[
 |v(\varphi)-v(0)|^2
 \le\left(\int_0^\varphi p^{-1}\right)
       \left(\int_0^\varphi p|v'|^2\right),\qquad
 \int_0^\varphi p^{-1}\le C\varphi^{2s}.
\]
Average $|v(0)|^2\le2|v(\varphi)|^2+
2|v(\varphi)-v(0)|^2$ over
$(4K)^{-1}<\varphi<(2K)^{-1}$ with weight $p$. Its mass is
comparable to $K^{2s-2}$, so multiplication by $K^{2s}$ gives
\eqref{angular-scaled-trace}. The case $K=1$ proves continuity
of the trace on $\mathcal V_m$; the same Cauchy--Schwarz bound
gives a continuous representative near zero.

We first justify approximation at the other endpoint. Write
$r=\pi/2-\varphi$. For a regular degree-$m$ profile, an ordinary cutoff
on $r<2\epsilon$ costs $O(\epsilon^{n+2m-2})$ in the derivative
part when $n+2m>2$; all other parts tend to zero by integrability.
When $n=2,m=0$, use a cutoff changing from zero to one on
$[\epsilon^2,\epsilon]$, with derivative bounded by
$C/(r|\log\epsilon|)$. Its derivative cost is at most
\[
 \frac{C}{|\log\epsilon|^2}
     \int_{\epsilon^2}^{\epsilon}\frac{\dd r}{r}
 =O(|\log\epsilon|^{-1}).
\]
Approximation first by regular profiles proves that profiles vanishing
near the axis are dense in every $\mathcal V_m$. Such a profile also
belongs to the preceding degree's domain, and
$\lambda_m\le\lambda_{m+1}$ proves the continuous inclusion
$\mathcal V_{m+1}\subset\mathcal V_m$.

Direct differentiation gives
\[
 -(p h_m')'+p\lambda_m\sec^2\varphi\,h_m=q_m p h_m.
\]
For a smooth profile vanishing near the axis, expand the derivative
of $v=h_m(v/h_m)$ and integrate the mixed term. The boundary term at
zero vanishes: $p h_m'/h_m=O(\varphi^{2-2s})$ there, and $s<1$.
There is no axis boundary term. This proves
\eqref{angular-groundstate-identity} on the dense class just described.
To pass to the full domain, first use $x=1$. The right-hand side is
then the square of a norm: $q_m+1>0$, and the displayed identity
makes a Cauchy sequence in $\mathcal V_m$ Cauchy also in that norm.
Its local limit is $v/h_m$, so its derivative is the weak derivative
appearing in the assertion. Passing to the limit proves the identity
for $x=1$, and adding $(x-1)\|v\|_p^2$ proves it for every real $x$.
In particular $\|v\|_p^2\le(q_m+x)^{-1}E_{m,x}[v]$ when $q_m+x>0$.
Since
\[
 E_{m,0}[v]=E_{m,x}[v]-x\|v\|_p^2
 \le\left(1+\frac{|x|}{q_m+x}\right)E_{m,x}[v],
\]
this also controls the full defining norm.

For the uniform assertion, write
$Q=E_{m,\beta^2+\tau^2}$. The identity gives
\[
 Q[v]\ge((m+\beta)^2+\tau^2)\|v\|_p^2.
\]
For $m=j+1\ge1$, \eqref{extension-gamma-interval} therefore yields
\[
 E_{m,\beta^2+\tau^2-\gamma^2}[v]
 =Q[v]-\gamma^2\|v\|_p^2
 \ge\left(1-\frac{\gamma^2}{(\beta+1)^2}\right)Q[v].
\]
The coefficient is strictly positive for every weight in the full
interval \eqref{weight-interval}. Since $\beta>0$, the integral
$Q$ controls the defining norm and the term $\tau^2\|v\|_p^2$;
the reverse bound is immediate. For $m=j$ and zero shift $\gamma=0$,
this last observation proves the assertion directly, including $j=0$.
All constants stay bounded on compact subsets of the parameter ranges.
\end{proof}

\begin{lemma}[Weighted extension bounds]\label{extension-symbol-bounds}
Put $q_m=m(m+n-2s)$. If $\Ree z>-q_m$, there is a unique
$\Phi_{m,z}\in\mathcal V_m$ with trace one satisfying
\begin{equation}\label{angular-DN-equation}
 -(p\Phi_{m,z}')'+p(\lambda_m\sec^2\varphi+z)\Phi_{m,z}=0.
\end{equation}
Its conormal derivative is well defined. With $d_s$ as in
\eqref{prelim-extension-normalization}, set
\begin{equation}\label{angular-DN-symbol}
 F_m(z)=-d_s\lim_{\varphi\downarrow0}p(\varphi)\Phi_{m,z}'(\varphi).
\end{equation}
Both $\Phi_{m,z}$ and $F_m(z)$ depend holomorphically on $z$ in this
half-plane, and $F_m(\overline z)=\overline{F_m(z)}$.
For every real $x>-q_m$ and $y\in\R$,
\begin{equation}\label{dn-positive-correction}
 \Ree F_m(x+\ii y)-F_m(x)
 =d_s E_{m,x}[\Phi_{m,x+\ii y}-\Phi_{m,x}]\ge0.
\end{equation}
For an admissible weight, define
\begin{equation}\label{shared-symbols}
 a_j(\tau)=F_j(\beta^2+\tau^2),\qquad
 t_{j,\vartheta}(\tau)
   =F_{j+1}(\beta^2+\tau^2-\gamma^2-2\ii\gamma\tau).
\end{equation}
Then, with $K=\langle(j,\tau)\rangle$,
\begin{equation}\label{uniform-real-symbol}
 0<a_j(\tau)\le CK^{2s},\qquad
 |t_{j,\vartheta}(\tau)|\le CK^{2s}.
\end{equation}

For either of the real parameters in \eqref{shared-symbols}, the
real trace-one minimizer $v=\Phi_{m,x}$ satisfies
\begin{equation}\label{dn-boundary-layer}
 E_{m,x}[v]\le CK^{2s},\qquad
 \|v\|_p^2\le CK^{2s-2},\qquad
 v(\varphi)\ge\tfrac12\quad(0\le\varphi\le\epsilon/K).
\end{equation}
For $m=j+1$, $x=\beta^2+\tau^2-\gamma^2$, and $|y|\le C_0K$,
the complex solutions satisfy
\begin{align}
 \|\Phi_{j+1,x+\ii y}\|_p^2&\le CK^{2s-2},
       \label{dn-extension-mass}\\
 E_{j+1,x}[\Phi_{j+1,x+\ii y}]&\le CK^{2s}.
       \label{dn-extension-energy}
\end{align}
Here $C_0$ is any fixed finite constant, and $C$ may depend on $C_0$.
In particular these estimates hold for $|y|\le2|\gamma\tau|$.
All constants are uniform in $j,\tau,y$ and locally uniform in
$s,\vartheta$ in their displayed ranges.
\end{lemma}
\begin{proof}
The preceding lemma shows that the real part of
\[
 E_{m,\Ree z}(\psi,w)+\ii\Imm z\,\langle\psi,w\rangle_p
\]
is coercive on $\mathcal V_m^0$. Prescribe a smooth trace-one lift
supported away from the axis, and solve for its zero-trace correction
by the complex Lax--Milgram theorem. This constructs the unique
solution. Its continuous representative is bounded near zero, where
the potential is bounded and $p$ is integrable. The equation therefore
gives a finite limit of $p\Phi'$ there. The axis approximation in
Lemma~\ref{angular-groundstate} justifies integration by parts without
an axis boundary term. For $z=x+\ii y$ this gives
\begin{equation}\label{dn-flux-energy}
 F_m(z)=d_s\left(E_{m,x}[\Phi_{m,z}]
                         +\ii y\|\Phi_{m,z}\|_p^2\right).
\end{equation}
When $z=x$ is real, the solution is the real trace-one energy minimizer.
For $v=\Phi_{m,x}$ and $w=\Phi_{m,x+\ii y}-v\in\mathcal V_m^0$,
its minimizing property gives $E_{m,x}(v,w)=0$. Hence
$E_{m,x}[v+w]=E_{m,x}[v]+E_{m,x}[w]$. Taking real parts in
\eqref{dn-flux-energy} proves \eqref{dn-positive-correction},
including its sign by coercivity.
Complex conjugation and uniqueness give the stated conjugation law.
For holomorphic dependence, write the zero-trace equation with its
fixed lift as a bounded operator equation on $\mathcal V_m^0$.
At each parameter its inverse is locally a convergent Neumann series
in the parameter increment. The solution is therefore holomorphic in
$\mathcal V_m$. Pairing its equation against the fixed trace-one lift
expresses $F_m$ as a holomorphic form pairing, proving the assertion
for $F_m$ as well.

For the uniform bounds, \eqref{dn-coercivity} implies
\[
 E_{m,x}[v]\ge c\int_0^{\pi/2}p(|v'|^2+K^2|v|^2).
\]
A fixed smooth cutoff $\eta$, equal to one near zero and supported
in $[0,\pi/3)$, gives the trace-one competitor
$\eta(\varphi)e^{-K\varphi}$. Its energy is at most $CK^{2s}$,
since $p\asymp\varphi^{1-2s}$ and $\sec\varphi$ is bounded on its
support. Minimization and coercivity give the first two bounds in
\eqref{dn-boundary-layer}. The trace estimate in the preceding proof
then gives $|v(\varphi)-1|^2\le C(K\varphi)^{2s}$, proving the last
bound after choosing $\epsilon$ small.

For the complex parameter let $v=\Phi_{m,x}$ and
$w=\Phi_{m,x+\ii y}-v\in\mathcal V_m^0$. Subtracting the weak equations
gives, for every $\psi\in\mathcal V_m^0$,
\begin{equation}\label{dn-correction-equation}
 E_{m,x}(\psi,w)+\ii y\langle\psi,w\rangle_p
                         =-\ii y\langle\psi,v\rangle_p.
\end{equation}
Test with $w$, take real parts, and use coercivity to obtain
\[
 E_{m,x}[w]\le |y|\|v\|_p\|w\|_p
       \le CK^{-1}|y|\|v\|_p E_{m,x}[w]^{1/2}.
\]
Consequently
\[
 E_{m,x}[w]\le CK^{-2}y^2\|v\|_p^2,\qquad
 \|w\|_p\le CK^{-2}|y|\|v\|_p.
\]
Since $|y|\le C_0K$ and $E_{m,x}(v,w)=0$, the bounds for $v$ give
\eqref{dn-extension-mass}--\eqref{dn-extension-energy}.
Formula~\eqref{dn-flux-energy} proves the asserted symbol bound.
For $a_j$ the same formula and real minimization give
$0<a_j(\tau)\le CK^{2s}$.
\end{proof}

The multipliers $\mathcal A_s$ and $\T_{s,\vartheta}$ with symbols
$a_j$ and $t_{j,\vartheta}$ act on the original degree-$j$ spaces
$\mathcal H_j$ of the cylinder and map $H^s(\Cyl)$ boundedly to
$H^{-s}(\Cyl)$. Their symbols satisfy
$a_j(-\tau)=\overline{a_j(\tau)}$ and
$t_{j,\vartheta}(-\tau)=\overline{t_{j,\vartheta}(\tau)}$, so both
operators preserve real functions. These definitions use only the
extension problem; their explicit Gamma expressions will be identified
when the optimal dimension and order conditions are considered.

\subsection{Cylinder operators and their energies}

The one-dimensional extensions assemble into a boundary value problem
on the extended cylinder. For $\gamma=\beta-1+\vartheta$ and real functions
$W,Z$, define the real bilinear forms
\[
 \begin{aligned}
 \mathcal E_\gamma(W,Z)
 &=\int p\bigl(W_tZ_t+W_\varphi Z_\varphi
       +\sec^2\varphi\,\nabla_\omega W\cdot\nabla_\omega Z\bigr)\\
 &\quad+\int p\sec^2\varphi\,[2W\Lambda Z+(n-1)WZ]
                +(\beta^2-\gamma^2)\int pWZ,\\
 \mathcal B_\gamma(W,Z)&=\mathcal E_\gamma(W,Z)-2\gamma\int pW_tZ.
 \end{aligned}
\]
All integrals are over $\R\times(0,\pi/2)\times\Sph$.
The energy domain is the completion, in $\mathcal E_\gamma$, of
real finite angular sums with compact $t$ support and degree-$j$
profiles in $\mathcal V_{j+1}$; positivity follows from
\eqref{dn-coercivity}.

\begin{lemma}\label{cylinder-extension}
For real $v\in H^s(\Cyl)$ there is a unique real finite-energy solution
$P_\gamma v$ of
\begin{equation}\label{cylinder-DN-equation}
 \begin{aligned}
 &\bigl[-\partial_{tt}-2\gamma\partial_t
       -p^{-1}\partial_\varphi(p\partial_\varphi)
       +\beta^2-\gamma^2\\
 &\qquad+\sec^2\varphi\,(-\Delta_{\Sph}+2\Lambda+n-1)\bigr]W=0,
       \qquad W|_{\varphi=0}=v.
 \end{aligned}
\end{equation}
It satisfies
\begin{equation}\label{cylinder-extension-mass}
 \|P_\gamma v\|_{L^2(p)}\le C\|v\|_{H^{s-1}(\Cyl)}.
\end{equation}
The trace of every real energy function $Z$ belongs to $H^s(\Cyl)$,
with
\begin{equation}\label{cylinder-trace-bound}
 \|Z|_{\varphi=0}\|_{H^s(\Cyl)}^2\le C\mathcal E_\gamma[Z].
\end{equation}
For every real energy test $Z$, its boundary Green identity is
\begin{equation}\label{cylinder-extension-green}
 d_s\mathcal B_\gamma(P_\gamma v,Z)
       =\langle Z|_{\varphi=0},\T_{s,\vartheta}v\rangle.
\end{equation}
\end{lemma}
\begin{proof}
The angular operator $-\Delta_{\Sph}+2\Lambda+n-1$ has eigenvalue
$\lambda_j+2j+n-1=\lambda_{j+1}$ in the original degree-$j$ space.
Thus \eqref{dn-coercivity}, after Fourier transformation in $t$,
shows that $\mathcal E_\gamma$ is equivalent to the integral of
\[
 p\left(|W_t|^2+|W_\varphi|^2+|W|^2
  +\sec^2\varphi\,[|\nabla_\omega W|^2+2W\Lambda W+(n-1)W^2]\right).
\]
In this expression the angular term is understood by integration on
$\Sph$; its value is nonnegative. In particular the energy controls
mass and both $t$ and $\varphi$ derivatives, for every weight in the
full interval \eqref{weight-interval}, including its negative part.
The drift form is bounded on this energy domain.

For each degree and frequency, \eqref{angular-scaled-trace} and
\eqref{dn-coercivity}, with $K=(1+j^2+\tau^2)^{1/2}$, give
\[
 K^{2s}|z(0)|^2\le C E_{j+1,\,\beta^2+\tau^2-\gamma^2}[z].
\]
Summing in the angular modes and integrating in $\tau$ proves
\eqref{cylinder-trace-bound}.
Density extends the trace map to the energy domain.

Integration in $t$, also justified by density, gives
$\mathcal B_\gamma(Z,Z)=\mathcal E_\gamma[Z]$.
Testing the difference of two solutions proves uniqueness.

For existence, multiply the degree-$j$, frequency-$\tau$
coefficient of $v$ by
$\Phi_{j+1,\,\beta^2+\tau^2-\gamma^2-2\ii\gamma\tau}(\varphi)$.
The conjugation law shows that this extension is real for real data.
The energy bounds in Lemma~\ref{extension-symbol-bounds} give a
finite-energy weak solution with trace $v$. Its mass bound and
Plancherel give \eqref{cylinder-extension-mass}. One can construct this
solution first for finite angular sums and compactly supported Fourier
data, approximate each profile in its $\mathcal V_{j+1}$ norm, and
then pass to the energy limit. This also places the solution in the
completion defining the energy domain.

For those finite sums, \eqref{angular-DN-equation} gives the Green
identity by integration by parts. The trace-one lift calculation
used in \eqref{dn-flux-energy} identifies its boundary coefficient
as $F_{j+1}$. The energy, trace, and symbol bounds make both sides
continuous, so density gives \eqref{cylinder-extension-green} for
all the stated traces and tests.
\end{proof}

\begin{proposition}[Fractional operators in cylinder coordinates]\label{cylinder-realizations}
For $v\in C_c^\infty(\Cyl)$ and an admissible weight $\vartheta$,
\begin{align}
 \T_{s,\vartheta}v
   &=r^{2s+2-\vartheta}\mathcal L^s(r^{\vartheta-2}v),
       \label{cone}\\
 (-\Delta)^s(r^{-\beta}v)
   &=r^{-\beta-2s}\mathcal A_sv.
       \label{cylinder-laplacian-realization}
\end{align}
The right side of \eqref{cone} extends as a bounded map
$H^s(\Cyl)\to H^{-s}(\Cyl)$. In particular, on smooth compact
annular functions,
\begin{equation}\label{field-intertwining}
 \T_{s,0}D_{\mathrm{har}}=r^{2s}D_{\mathrm{har}}L.
\end{equation}
\end{proposition}
\begin{proof}
By complex linearity it suffices to consider real $v$.
In degree $j$, the form of $\mathcal L$ is
\[
 \int_0^\infty
 \left(|w'|^2+j(j+n)r^{-2}|w|^2\right)r^{n+1}\dd r.
\]
Choose a normalized degree-$j$ harmonic $\widetilde Y_j$ on
$\mathbb S^{n+1}$. The map $w\mapsto w(r)\widetilde Y_j$ is unitary
onto that component of $L^2(\R^{n+2})$ and preserves the form.
Smooth coefficients there have the form $r^j\psi(r^2)$ at zero.
Cutting them off on $r<2\varepsilon$ costs
$O(\varepsilon^{n+2j})$ in radial gradient energy. An outer cutoff
and smooth approximation show that compact annular functions are
a form core. Thus the closed forms, their operators and their
fractional powers agree. We use this identification for each
original angular coefficient separately.

First let $v$ have finitely many angular components. For the cone
identity take the actual Poisson extension $Q$ of $r^{\vartheta-2}v$
in these dimension-$(n+2)$ components. Set
\[
 R=(r^2+y^2)^{1/2},\qquad t=\log R,\qquad
 r=R\cos\varphi,\quad y=R\sin\varphi,
\]
and define
\begin{equation}\label{actual-extension-transform}
 W(t,\varphi,\omega)
   =R^{2-\vartheta}\cos\varphi\,
                       Q(R\cos\varphi,\omega,R\sin\varphi).
\end{equation}
The trace of $W$ is $v$. We identify its equation, verify membership
in the energy domain, and then compute its boundary flux.

For a degree-$j$ Poisson extension in dimension $d$, write
$U=R^{-b}V(t,\varphi)Y_j$. Its extension equation becomes
\[
 \begin{aligned}
 &\bigl[-\partial_{tt}-(d-2s-2b)\partial_t
       -p_d^{-1}\partial_\varphi(p_d\partial_\varphi)\\
 &\qquad+j(j+d-2)\sec^2\varphi+b(d-2s-b)\bigr]V=0,\qquad
 p_d=\sin^{1-2s}\varphi\cos^{d-1}\varphi.
 \end{aligned}
\]
For $d=n+2$, $b=2-\vartheta$, one has $d-2s-2b=2\gamma$ and
$p_{n+2}=p\cos^2\varphi$. The change $V=\sec\varphi\,W$ uses
\[
 \cos\varphi\bigl[-p_{n+2}^{-1}
    (p_{n+2}(\sec\varphi\,W)_\varphi)_\varphi\bigr]
 =-p^{-1}(pW_\varphi)_\varphi
       +\bigl[(n-1)\sec^2\varphi-(n+1-2s)\bigr]W.
\]
Since $j(j+n)+n-1=\lambda_{j+1}$ and
$b(n+2-2s-b)-(n+1-2s)=\beta^2-\gamma^2$, this is precisely
\eqref{cylinder-DN-equation}.

The needed finite energy follows from the actual Euclidean Poisson
formula \cite[\S2.4]{CaSi},
\[
 P_y^{(d)}(x)=\kappa_{d,s}
       \frac{y^{2s}}{(|x|^2+y^2)^{(d+2s)/2}},\qquad
 U(x,y)=P_y^{(d)}*F(x),
\]
for the smooth compact annular datum $F$. We give the estimates at the
ends, where the change of variables introduces weights. If $R$ is
smaller than half the inner support radius, write $U=y^{2s}H$;
differentiation of the separated denominator gives
$|H|+R|\nabla_{x,y}H|\le C$. If $R$ is larger than twice the outer
support radius, it gives
\[
 |H|+R|\nabla_{x,y}H|\le CR^{-d-2s}.
\]
The same bounds hold after the finite number of angular derivatives
needed for a fixed finite harmonic sum.

The angular energy of $R^bU$, including its $t$ derivative, is therefore
bounded by $CR^{2(b+2s)}$ at zero and $CR^{2(b-d)}$ at infinity.
These bounds hold as well after multiplication by $\cos\varphi$ in
\eqref{actual-extension-transform}. To check angular integrability,
near $\varphi=0$ the possible singular derivative is
$O(\varphi^{2s-1})$, whose square times $p$ is integrable.
Near the axis the regular degree-$j$ profile is
$\cos^j\varphi\,\psi(\cos^2\varphi)$; multiplication by
$\cos\varphi$ gives the required degree-$j+1$ profile.
For the cone realization, $d=n+2$ and $b=2-\vartheta$; the full
weight interval implies
\[
 0<b<d-2s<d.
\]
For the Euclidean realization below, $d=n$ and $b=\beta\in(0,d)$.
Thus both end bounds are integrable in $\dd t=\dd R/R$.

On compact $R$ intervals the same energy is finite. Indeed, subtract
the affine Taylor polynomial of the smooth datum in a neighborhood
of the boundary point in the Poisson integral. Its remainder is
quadratic; differentiating it and its separated complement gives
$|U|+|\nabla_xU|\le C$ and $|U_y|\le Cy^{2s-1}$ for small $y$.
The resulting angular derivative is again square integrable against
$p$. At the axis the smooth degree-$j$ profile gives the integrability
just checked. This proves finite raw energy, and hence finite
$\mathcal E_\gamma$ energy by \eqref{dn-coercivity}.

Radial cutoffs and the axis approximation established for
$\mathcal V_{j+1}$ put $W$ in the energy domain of
$\mathcal B_\gamma$. Its equation and trace therefore imply
$W=P_\gamma v$ by uniqueness. At $\varphi=0$, differentiating at fixed
$x$ gives the same conormal derivative under
\eqref{actual-extension-transform}: terms differentiating the
radial weight or $\cos\varphi$ vanish as
$\varphi^{2-2s}$. The boundary Green formula and
\eqref{prelim-extension-normalization} consequently give
\[
 -d_s\lim_{\varphi\downarrow0}pW_\varphi
   =r^{2s+2-\vartheta}\mathcal L^s(r^{\vartheta-2}v)
   =\T_{s,\vartheta}v.
\]

For \eqref{cylinder-laplacian-realization}, take the actual
dimension-$n$ extension of $r^{-\beta}v$ and set $Z=R^\beta U$.
The displayed coordinate equation now has no $t$ drift and has
potential $\lambda_j\sec^2\varphi+\beta^2$. The same estimates,
with $b=\beta\in(0,n)$, put $Z$ in its energy domain. Fourier
transformation in $t$ and the uniqueness in
Lemma~\ref{extension-symbol-bounds} identify its angular components
with $\Phi_{j,\beta^2+\tau^2}$. Its boundary flux is therefore
$\mathcal A_sv$, proving the identity.

Finally, angular partial sums of a smooth annular function converge
with two derivatives on the fixed annulus. After multiplication by
$r^{\vartheta-2}$ or $r^{-\beta}$ they converge in the graph norm
of $\mathcal L$ or $-\Delta$. The spectral bound
$\lambda^{2s}\le1+\lambda^2$ gives convergence of their fractional
powers. The bounds in Lemma~\ref{extension-symbol-bounds} then
pass both identities to $v\in C_c^\infty(\Cyl)$ and give the
asserted continuous extension.

The identity $\mathcal R=r^{-2}D_{\mathrm{har}}$ and
\eqref{heat-fractional-intertwining} give
\[
 \T_{s,0}D_{\mathrm{har}}v
 =r^{2s+2}\mathcal L^s\mathcal Rv
 =r^{2s+2}\mathcal RLv
 =r^{2s}D_{\mathrm{har}}Lv,
\]
which proves \eqref{field-intertwining}.
\end{proof}

The cone realization also gives the weight conjugation directly:
\begin{equation}\label{cylinder-weight-conjugation}
 \T_{s,\vartheta}v
   =r^{-\vartheta}\T_{s,0}(r^\vartheta v),
 \qquad v\in C_c^\infty(\Cyl).
\end{equation}
For functions reaching
the inner cylinder end it passes to the limit whenever the weighted
and unweighted annular approximants converge in $H^s$; multiplication
by $r^{-\vartheta}$ is smooth on each compact cylinder set.

The realization of $\mathcal A_s$ also identifies the fractional
energy and the stability inequality in cylinder coordinates.

\begin{lemma}[Fractional energy in cylinder coordinates]
\label{cylinder-energy-forms}
Let $\beta=(n-2s)/2$, and let $\mathcal A_s$ be the multiplier
defined in \eqref{shared-symbols}. For $v\in C_c^\infty(\Cyl)$,
\begin{equation}\label{energy-identity}
 [r^{-\beta}v]_{H^s(\R^n)}^2=\langle v,\mathcal A_sv\rangle.
\end{equation}
The operator $\mathcal A_s$ is positive and selfadjoint.

Let $u$ have the solution and nonlinearity assumptions of
Theorem~\ref{full}. If $v$ is real and the physical test
$r^{-\beta}v$ is supported
in $B_2$, then
\begin{equation}\label{critical-stability}
 \int_{\R^n}f'(u)(r^{-\beta}v)^2\dd x
 =\int_{\Cyl}r^{2s}f'(u)v^2\dd t\dd\omega.
\end{equation}
By stability, the latter is at most
$\langle v,\mathcal A_sv\rangle$. These identities and this bound
extend to $v\in H^s(\Cyl)$ supported in $\{r\le R\}$, for each
fixed $R<2$.

Fix an admissible weight $\vartheta$. A pointwise lower bound
$\Ree t_{j,\vartheta}(\tau)-a_j(\tau)
\ge\kappa(1+j^2+\tau^2)^{s-1}$ implies
\begin{equation}\label{spectral-energy-gap}
 \Ree\langle v,\T_{s,\vartheta}v\rangle
 -\langle v,\mathcal A_sv\rangle
 \ge\kappa_0\|v\|_{H^{s-1}(\Cyl)}^2,
 \qquad v\in C_c^\infty(\Cyl),
\end{equation}
with $\kappa_0=\kappa$. Conversely, \eqref{spectral-energy-gap}
implies the pointwise bound with $\kappa=\kappa_0C_n^{s-1}$, where
$C_n$ is from \eqref{cylinder-weight-comparison}. The quadratic forms
and \eqref{spectral-energy-gap} extend continuously to $H^s(\Cyl)$.
\end{lemma}
\begin{proof}

By \eqref{cylinder-laplacian-realization},
\[
 L(r^{-\beta}v)=r^{-\beta-2s}\mathcal A_sv.
\]
Pair this identity with $r^{-\beta}v$ and use
$dx=r^n\dd t\dd\omega$ and $n-2\beta-2s=0$ to obtain
\eqref{energy-identity}. The real extension minimization in
Lemma~\ref{extension-symbol-bounds} gives $a_j(\tau)>0$.
Multiplication by this real symbol on the Fourier and spherical-harmonic
side defines the positive selfadjoint realization of $\mathcal A_s$.
The potential identity follows from $n-2\beta=2s$, and stability
applies to the smooth physical test.

For the density assertion, take $v\in H^s(\Cyl)$ supported in
$\{r\le R\}$, choose $R<R'<2$, and approximate it in $H^s$ by
$v_m\in C_c^\infty(\{r<R'\})$. Such an approximation is obtained
by cutting off the inner cylinder end and then smoothing.
Put $\phi_m=r^{-\beta}v_m$. The symbol bounds in
Lemma~\ref{extension-symbol-bounds} and \eqref{energy-identity}
control the energy of $\phi_m-\phi_k$, while
\[
 \|\phi_m-\phi_k\|_{L^2(\R^n)}^2
 =\int_{\Cyl}r^{2s}|v_m-v_k|^2\dd t\dd\omega
 \le (R')^{2s}\|v_m-v_k\|_2^2.
\]
Thus $\phi_m\to r^{-\beta}v$ in Euclidean $H^s$.
The compact local range assumption gives a bound for $f'(u)$ on
$B_{R'}$. Consequently $r^{2s}f'(u)$ is bounded on the supports,
so the potential integrals converge as well. This proves the
extended identities and the stability bound.

For the form comparison, write
\[
 g_j(\tau)=\Ree t_{j,\vartheta}(\tau)-a_j(\tau),\qquad
 W_j(\tau)=1+j(j+n-2)+\tau^2.
\]
Plancherel identifies the left-hand side of
\eqref{spectral-energy-gap} with
$\sum_{j,\ell}\int g_j(\tau)|\widehat v_{j,\ell}(\tau)|^2\dd\tau$.
Since $s-1<0$ and $1+j^2+\tau^2\le W_j(\tau)$, the pointwise
bound implies \eqref{spectral-energy-gap} with $\kappa_0=\kappa$.

Conversely fix $j,\ell,\tau_0$, choose
$\chi\in C_c^\infty(\R)$ with $\|\chi\|_2=1$, and test the form
inequality with
\[
 v_R(t,\omega)=R^{-1/2}\chi(t/R)e^{\ii\tau_0t}Y_{j,\ell}(\omega).
\]
Its Fourier mass is $R|\widehat\chi(R(\tau-\tau_0))|^2$.
After the change $z=R(\tau-\tau_0)$, continuity of the coefficients
gives $g_j(\tau_0)\ge\kappa_0W_j(\tau_0)^{s-1}$ as $R\to\infty$.
Continuity follows from the parameter dependence in
Lemma~\ref{extension-symbol-bounds}; polynomial symbol growth and
the rapid decay of $\widehat\chi$ justify dominated convergence.
Since $W_j\le C_n(1+j^2+\tau^2)$ and $s-1<0$, this is the stated
pointwise bound. If the form inequality is initially written for
real tests, apply it separately to the real and imaginary parts:
both multipliers preserve real functions, so their real quadratic
forms add. Thus the same complex test is permissible.

The same symbol bounds give
$\T_{s,\vartheta},\mathcal A_s:H^s\to H^{-s}$, so density extends
these quadratic forms and their inequality to $H^s(\Cyl)$.
\end{proof}

\section{Poisson averaging and the localized equation}
\label{pde-contract-section}\label{poisson-source-section}
\label{linear-poisson-identity}

The corrected derivative need not have a sign. We obtain a positive
average that recovers the fractional source, then apply the derivative
to the localized semilinear equation. Convexity makes the resulting
defect nonnegative, while localization contributes a controlled tail error.

\subsection{A positive average of the source}

For the corrected derivative, values on larger spheres can be
combined with the positive ball Poisson operators. Define
\begin{equation}\label{intro-poisson-averaging}
 \mathcal S_sg(r\omega)=\frac{2^s}{\Gamma(1-s)}
 \int_0^\infty e^{-2sh}(1-e^{-2h})^{-s}
       P_{e^{-h}}[g(e^hr\,\cdot)](\omega)\dd h.
\end{equation}

For an admissible $\vartheta$, let $\mathcal S_{s,\vartheta}$ be given
by the same formula with $e^{-2sh}$ replaced by
$e^{-(2s-\vartheta)h}$. On compact annular data,
\[
 \mathcal S_{s,\vartheta}=r^{-\vartheta}\mathcal S_sr^\vartheta,
 \qquad \mathcal S_{s,0}=\mathcal S_s.
\]
Their Fourier and spherical-harmonic multipliers are expressed by
\begin{equation}\label{averaging-symbol-definition}
 m_s(z)=\frac{2^s}{\Gamma(1-s)}
       \int_0^\infty e^{-(z+2s)r}(1-e^{-2r})^{-s}\dd r,
 \qquad \Ree z>-2s.
\end{equation}
The symbol of $\mathcal S_{s,\vartheta}$ is
$m_s(j-\vartheta-\ii\tau)$.

\begin{lemma}[Averaging bound]\label{averaging-symbol-bound}
For every fixed real $c>-2s$,
\begin{equation}\label{uniform-smoothing-symbol}
 |m_s(j+c+\ii\tau)|\le C_{s,c}(1+j^2+\tau^2)^{(s-1)/2}.
\end{equation}
For every admissible weight, $\mathcal S_{s,\vartheta}$ preserves
nonnegativity on $L^2(\Cyl)$ and is bounded from $H^{s-1}(\Cyl)$
to $L^2(\Cyl)$. In particular,
\begin{equation}\label{intro-poisson-smoothing}
 \mathcal S_s:H^{s-1}(\Cyl)\longrightarrow L^2(\Cyl).
\end{equation}
\end{lemma}
\begin{proof}
The scalar weight in the averaging integral is $O(h^{-s})$ at
zero and decays exponentially at infinity because $\vartheta<2s$.
The spherical Poisson operators and radial translations are positive
$L^2(\Cyl)$ contractions. The integral therefore defines a positive
bounded operator on $L^2$, and transforming it gives the stated symbol.

The amplitude $k_j(r)=e^{-(j+c+2s)r}(1-e^{-2r})^{-s}$ is positive
and decreasing. Comparison with $e^{-(j+c+2s)r}r^{-s}$ near zero,
and exponential decay away from zero, give
$\int_0^\infty k_j\le C_{s,c}(1+j)^{s-1}$.
For $|\tau|\ge1$, put $\rho=|\tau|^{-1}$.
The integral on $(0,\rho)$ is bounded by $C\rho^{1-s}$.
Integration by parts on $(\rho,\infty)$ gives
\[
 \left|\int_\rho^\infty e^{\ii\tau r}k_j(r)\dd r\right|
 \le\frac{k_j(\rho)+\int_\rho^\infty|k_j'(r)|\dd r}{|\tau|}
 =\frac{2k_j(\rho)}{|\tau|}\le C|\tau|^{s-1}.
\]
Combining the two estimates proves \eqref{uniform-smoothing-symbol}.
Take $c=-\vartheta$ and apply Plancherel for the mapping assertion.
\end{proof}

The corresponding average of a locally bounded source $g$ in $B_2$ is
\begin{equation}\label{source-definition}
 F_s[g](x)=\frac{2^{1-s}}{\Gamma(s)|\Sph|}
 \int_{B_{|x|}}\frac{(|x|^2-|y|^2)^s}{|x-y|^n}g(y)\dd y,
 \qquad 0<|x|<2.
\end{equation}
This linear operator allows either sign of $g$. Scaling $y=rz$ gives
\begin{equation}\label{source-average-bound}
 |F_s[g](r\omega)|\le C_{n,s}r^{2s}\|g\|_{L^\infty(B_r)},
 \qquad 0<r<2.
\end{equation}
Indeed, $1-|z|^2\le2|\omega-z|$ for $z\in B_1$, so the scaled
kernel is bounded by $C|\omega-z|^{s-n}$, which is integrable
because $s>0$. The same formula applies at any radius for which
$g$ is locally bounded on a neighborhood of the closed ball.
The source of the semilinear equation is $g=f(u)$, so its average is
$F_s[f(u)]$. Positivity of the kernel gives $F_s[g]\ge0$ whenever
$g\ge0$.

\begin{theorem}[Poisson averaging and the fractional source identity]
\label{intro-poisson-source}
Let $n\ge2$, $0<s<1$, and let $v\in C_c^{1+\rho}(\R^n)$ be
real-valued, where $s<\rho<1$.
If $H_r^v$ is its classical harmonic replacement in $B_r$, put
\[
 q_v(r\omega)
   =r\bigl(\partial_\nu H_r^v(r\omega)-\partial_\nu v(r\omega)\bigr).
\]
Then $q_v\in H^s(\Cyl)$ and, for every $x=r\omega\ne0$,
\begin{equation}\label{intro-fractional-source}
 \mathcal S_sq_v(x)=F_s[(-\Delta)^sv](x).
\end{equation}
In particular, if $(-\Delta)^sv\ge0$ in $B_r$, then
$\mathcal S_sq_v\ge0$ on $\partial B_r$.
\end{theorem}

\begin{proof}
By \eqref{harmonic-replacement-field}, $q_v=D_{\mathrm{har}}v$.
We first prove the source identity for $v\in C_c^\infty(\R^n)$ at
radius one. Let $H^{(s)}$ be the fractional harmonic replacement of $v$
in $B_1$. We compute the boundary limit of
$(v-H^{(s)})(a\omega)/(1-a^2)^s$ in two ways: the Green
formula gives the interior source average, and the exterior Poisson
formula gives the average of the corrected derivative.
The exterior Poisson kernel is given in
\eqref{prelim-fractional-poisson}. Then $w=v-H^{(s)}$ is zero outside
$B_1$ and satisfies $(-\Delta)^sw=(-\Delta)^sv$ in $B_1$.
The ball Green formula \cite[Theorems~3.1--3.3]{Bucur} gives
\[
 w(x)=\int_{B_1}G_s(x,y)(-\Delta)^sv(y)\dd y,
\]
where
\begin{equation}\label{ball-green-source}
 \begin{split}
 G_s(x,y)&=\gamma_{n,s}|x-y|^{2s-n}
           \int_0^{R(x,y)}\frac{t^{s-1}}{(1+t)^{n/2}}\dd t,\\
 R(x,y)&=\frac{(1-|x|^2)(1-|y|^2)}{|x-y|^2},\qquad
 \gamma_{n,s}=\frac{2^{1-2s}}{|\Sph|\Gamma(s)^2}.
 \end{split}
\end{equation}

We compute the boundary limit of this representation. For
$\omega\in\Sph$ and fixed $y\in B_1$,
\[
 \lim_{a\uparrow1}\frac{G_s(a\omega,y)}{(1-a^2)^s}
 =\frac{\gamma_{n,s}}s
       \frac{(1-|y|^2)^s}{|\omega-y|^n}.
\]
The limit can be taken inside the integral. Indeed, the explicit
formula implies both
\[
 G_s(x,y)\le C|x-y|^{2s-n},\qquad
 G_s(x,y)\le C(1-|x|^2)^s(1-|y|^2)^s|x-y|^{-n}.
\]
Put $\delta=1-a$. After division by $(1-a^2)^s$, the first bound
integrates to $C\delta^s$ on $|y-\omega|<2\delta$.
On $2\delta<|y-\omega|<\epsilon$, use
$|a\omega-y|\ge|\omega-y|/2$ and
$1-|y|^2\le2|\omega-y|$ in the second bound. Its integral is at most
\[
 C\int_{|y-\omega|<\epsilon}|y-\omega|^{s-n}\dd y
 \le C\epsilon^s.
\]
These bounds prove uniform integrability at the boundary point; away
from it, dominated convergence applies. Since $(-\Delta)^sv$ is bounded,
we obtain
\begin{equation}\label{green-boundary-source}
 \lim_{a\uparrow1}\frac{v(a\omega)-H^{(s)}(a\omega)}{(1-a^2)^s}
 =\frac{\gamma_{n,s}}s
       \int_{B_1}\frac{(1-|y|^2)^s}{|\omega-y|^n}
                         (-\Delta)^sv(y)\dd y.
\end{equation}

The exterior Poisson formula gives another expression for the same
limit. Subtract $v(\omega)$ using the mass-one property of its kernel.
The term $v(a\omega)-v(\omega)$, divided by $(1-a^2)^s$, tends to zero.
For $|y|\ge1$, the identity
\[
 |a\omega-y|^2=a|\omega-y|^2+(1-a)(|y|^2-a)
                 \ge a|\omega-y|^2
\]
and the local Lipschitz bound for $v$ give the integrable majorant
$C(|y|^2-1)^{-s}|\omega-y|^{1-n}$ near $\omega$.
Its local integrability follows in boundary polar coordinates from
$s<1$. Away from that point the boundary singularity is
$(|y|-1)^{-s}$, and the far tail is $O(|y|^{-n-2s})$.
Dominated convergence therefore gives
\begin{equation}\label{poisson-boundary-source}
 \lim_{a\uparrow1}\frac{v(a\omega)-H^{(s)}(a\omega)}{(1-a^2)^s}
 =p_{n,s}\int_{|y|>1}
       \frac{v(\omega)-v(y)}{(|y|^2-1)^s|\omega-y|^n}\dd y.
\end{equation}

It remains to express this exterior integral by classical Poisson
averaging. Write
\[
 A_hv(\omega)=P_{e^{-h}}[v(e^h\,\cdot)](\omega),\qquad
 k(h)=(e^{2h}-1)^{-s}.
\]
The changes of variables $y=e^h\eta$ and
\eqref{spherical-poisson} give
\begin{equation}\label{exterior-source-average}
 \int_{|y|>1}\frac{v(\omega)-v(y)}{(|y|^2-1)^s|\omega-y|^n}\dd y
 =|\Sph|\int_0^\infty
   e^{-2sh}(1-e^{-2h})^{-s-1}[v(\omega)-A_hv(\omega)]\dd h.
\end{equation}
Differentiating the Poisson semigroup and the dilation yields
$\partial_hA_hv=-A_hD_{\mathrm{har}}v$. Moreover,
\[
 -k'(h)=2s e^{-2sh}(1-e^{-2h})^{-s-1}.
\]
Integration by parts in \eqref{intro-poisson-averaging}, with
$A_hv-v$ in place of $A_hv$, therefore gives
\[
 \mathcal S_sD_{\mathrm{har}}v(\omega)
 =\frac{2^{s+1}s}{\Gamma(1-s)}
  \int_0^\infty e^{-2sh}(1-e^{-2h})^{-s-1}
                  [v(\omega)-A_hv(\omega)]\dd h.
\]
There is no boundary term: $A_hv-v=O(h)$ at zero, while $k(h)$
decays exponentially at infinity. Combining this identity with
\eqref{green-boundary-source}--\eqref{exterior-source-average} gives
the source identity at radius one. Its constant is
\[
 \frac{2^{s+1}s}{\Gamma(1-s)|\Sph|p_{n,s}}
                  \frac{\gamma_{n,s}}s
 =\frac{2^{1-s}}{\Gamma(s)|\Sph|},
\]
using $\Gamma(s)\Gamma(1-s)=\pi/\sin(\pi s)$.
Applying the unit-radius identity to $v(r\,\cdot)$ proves
\eqref{intro-fractional-source} for every $r>0$.

For the stated regularity, Lemma~\ref{corrected-derivative-origin}
gives $q_v\in H^s(\Cyl)$ and continuity. Choose $s<\rho'<\rho$
and mollify $v$. Then $v_\epsilon\to v$ in $C^{1,\rho'}$ with a
common compact support, so the same lemma gives
$q_{v_\epsilon}\to q_v$ in $H^s(\Cyl)$.
Also $(-\Delta)^sv_\epsilon\to(-\Delta)^sv$ uniformly: the
singular-integral remainder is integrable since $1+\rho'>2s$.
The boundedness in \eqref{intro-poisson-smoothing} passes the left
side of \eqref{intro-fractional-source}, and the integrable ball kernel passes
the right side locally uniformly. Both sides are continuous for
$r>0$, by the integral formulas, so the identity holds at every
nonzero point. Positivity is immediate from its right-hand kernel.
\end{proof}

Thus the same fractional harmonic replacement links the positive
average of the corrected derivative to the interior fractional source.

\subsection{The equation for the corrected radial derivative}\label{regularity-appendix}\label{operator-inputs}

We now apply $D_{\mathrm{har}}$ to the semilinear equation. First we localize $u$
so that $D_{\mathrm{har}}$ acts on compactly supported data. The
nonlocal error created by this localization is smooth in the smaller
ball and is controlled by the tail of $u$. The qualitative bootstrap
in Lemma~\ref{bootstrap} justifies these differentiations.

Choose a smooth radial cutoff $0\le\chi_0\le1$,
equal to one on $B_{3/2}$ and supported in $B_{7/4}$, and put
$\bar u=\chi_0u$. In $B_1$,
\begin{equation}\label{tail-error}
 L\bar u=f(u)+\mathfrak h,\qquad
 \mathfrak h=L((\chi_0-1)u),\qquad
 \|\mathfrak h\|_{C^k(B_1)}\le C_k\|u\|_{L_s^1(\R^n)}.
\end{equation}
Indeed, the support of $(\chi_0-1)u$ is separated from $B_1$;
differentiating the kernel of $L$ gives an integrable bound
$C_k(1+|y|^{n+2s})^{-1}$ for each derivative.

Define
\[
 q=D_{\mathrm{har}}\bar u,
 \qquad q=D_{\mathrm{har}}u\quad\hbox{in }B_1.
\]
Applying
$\T_{s,0} D_{\mathrm{har}}=r^{2s}D_{\mathrm{har}}L$ to the localized equation produces a radial
derivative of $f(u)$ and an angular action of $\Lambda$ on $f(u)$.
The radial chain rule supplies the term $f'(u)\partial_tu$, while the failure
of an angular chain rule is kept as a separate defect. Thus
\begin{equation}\label{q-equation}
 \T_{s,0} q=Vq-B_0+E_0,\qquad
 V=r^{2s}f'(u),\quad E_0=r^{2s}D_{\mathrm{har}}\mathfrak h,\quad
 B_0=r^{2s}\bigl(f'(u)\Lambda u-\Lambda f(u)\bigr).
\end{equation}
The following lemma gives $B_0\ge0$.
\begin{lemma}[The angular convexity defect]\label{angular-defect}
Let $f\in C^1_{\mathrm{loc}}(\R)$ be convex and let
$v\in H^1(\Sph)\cap L^\infty(\Sph)$. Choose a compact interval
$J\subset\R$ containing the essential range of $v$. Then
\[
 B[v]=f'(v)\Lambda v-\Lambda f(v)\in L^2(\Sph),
 \qquad B[v]\ge0\quad\hbox{almost everywhere},
\]
and
\begin{equation}\label{defect-L2}
 \|B[v]\|_{L^2(\Sph)}
 \le 2\|f'\|_{L^\infty(J)}
           \|\nabla_{\Sph}v\|_{L^2(\Sph)}.
\end{equation}
\end{lemma}
\begin{proof}
The Sobolev chain rule gives $f(v)\in H^1$ and
$\nabla_{\Sph}f(v)=f'(v)\nabla_{\Sph}v$.
Since $j^2\le j(j+n-2)$ for $n\ge2$, the eigenvalue formulas and
Parseval's identity give $\|\Lambda g\|_2\le\|\nabla_{\Sph}g\|_2$.
This proves membership in $L^2$ and \eqref{defect-L2}.

For the sign, use the positive probability kernel of $\mathsf P_h$.
The tangent-line inequality
$f(b)-f(a)-f'(a)(b-a)\ge0$, integrated with $a=v(\omega)$ and
$b=v(\eta)$, gives
\[
 B_h[v]=
 \frac{\mathsf P_h f(v)-f(v)+f'(v)(v-\mathsf P_h v)}h\ge0.
\]
For $g\in H^1$, the spectral expansion yields
$h^{-1}(I-\mathsf P_h)g\to\Lambda g$ in $L^2$:
use $(1-e^{-hj})/h\to j$, bounded by $j$, and dominated convergence
in the square-summable coefficients. Apply this to $v$ and $f(v)$.
The boundedness of $f'(v)$ gives $B_h[v]\to B[v]$ in $L^2$.
The cone of nonnegative functions is closed in $L^2$, proving the sign.
\end{proof}

Applying Lemma~\ref{angular-defect} on each sphere gives the nonnegative
$L^2_{\rm loc}$ term in \eqref{q-equation}.

The source estimate needed for H\"older continuity includes a
positive radial weight. We therefore conjugate the equation for $q$
by $r^{-\vartheta}$. Fix $\vartheta$ satisfying \eqref{weight-interval}
and set $q_\vartheta=r^{-\vartheta}q$. On smooth compact annular
functions, \eqref{cone} gives
\[
 \T_{s,\vartheta}=r^{-\vartheta}\T_{s,0}r^\vartheta.
\]
Using the same conjugation for the averaging operator defined above,
we obtain
\begin{align}
 \T_{s,\vartheta} q_\vartheta
 &=Vq_\vartheta-B_\vartheta+E_\vartheta,\qquad
 B_\vartheta=r^{-\vartheta}B_0\ge0,\quad
 E_\vartheta=r^{-\vartheta}E_0,\label{shifted-q-equation}\\
 r^{-\vartheta}\mathcal S_sq
 &=\mathcal S_{s,\vartheta}q_\vartheta.\label{shifted-smoothing}
\end{align}
We next control $q_\vartheta$ near $r=0$ to justify cutoff approximation
in $H^s(\Cyl)$, and estimate localized copies of $E_\vartheta$.
\begin{lemma}\label{field-admissibility}
Under the solution assumptions of Theorem~\ref{full}, with $n\ge2$
and without its dimension restriction, fix a weight in
\eqref{weight-interval}. Every smooth outer localization of
$q_\vartheta$ supported in $B_1$ belongs to
$H^s(\Cyl)\cap C(\Cyl)$ and is approximated in $H^s$ by cutoffs
away from $r=0$. On this half-cylinder $B_\vartheta\in L^2$.
For every smooth cutoff $\chi$ supported in $B_1$,
\begin{equation}\label{E-bound}
 \|\chi E_\vartheta\|_{H^{1-s}(\Cyl)}
 \le C_{n,s,\vartheta,\chi}\|u\|_{L_s^1(\R^n)}.
\end{equation}
\end{lemma}
\begin{proof}
Choose $0<\varepsilon<s$ with $a_0=2s-\varepsilon$ noninteger.
Then $a_0>s$, $a_0\ne1$, and Lemma~\ref{bootstrap} gives
$u\in C^{1+a_0}_{\mathrm{loc}}$. Fix $\vartheta$ in
\eqref{weight-interval}.

On $B_1$, let $M_u$ bound $|\nabla u|$ and let $M_f$ bound $|f'|$
on the compact range of $u$. Since
$|\nabla_{\Sph}u(r\omega)|\le rM_u$, \eqref{defect-L2} gives
\begin{equation}\label{defect-origin}
 \|B[u(r\,\cdot)]\|_{L^2(\Sph)}\le C M_fM_u r.
\end{equation}
Consequently, for every admissible weight $\vartheta<2s$,
\begin{equation}\label{weighted-defect-L2}
 \int_{-\infty}^{0}\!\|B_\vartheta(t,\cdot)\|_2^2\dd t
 \le C_u\int_0^1 r^{2(2s+1-\vartheta)}\,\frac{dr}{r}<\infty,
 \qquad B_\vartheta=r^{2s-\vartheta}B[u(r\,\cdot)].
\end{equation}

Put $\rho=\min\{a_0,1\}>s$. The localized function $\bar u$ is
$C_c^{1,\rho}$, so Lemma~\ref{corrected-derivative-origin} gives a
local cylinder $H^s$ bound $C_ur^{1+\rho}$ for
$q=D_{\mathrm{har}}\bar u$. On each rescaled annulus the factor
$|z|^{-\vartheta}$ is a fixed smooth multiplier. Consequently
\begin{equation}\label{field-origin-Hs}
 \|q_\vartheta\|_{H^s((\log r,\log(2r))\times\Sph)}
       \le C_u r^{1+\rho-\vartheta}.
\end{equation}
Since $1+\rho-\vartheta>0$, the slab summation used in that lemma
proves $H^s$ membership of every outer localization of
$q_\vartheta$ and strong convergence of its inner cutoffs.
Continuity follows from the same lemma and the smooth radial weight.

We finally bound the localization error using \eqref{tail-error}.
Subtract the affine Taylor polynomial of $\mathfrak h$ at zero.
Since $D_{\mathrm{har}}$ kills that polynomial, its remainder and the smooth bounds
in \eqref{tail-error} imply
\[
 \|D_{\mathrm{har}}\mathfrak h(r\,\cdot)\|_{H^1(\Sph)}
 +\|\partial_t D_{\mathrm{har}}\mathfrak h(r\,\cdot)\|_{L^2(\Sph)}
 \le Cr^2\|u\|_{L_s^1(\R^n)}.
\]
More explicitly, if $R_{\mathfrak h}$ is the Taylor remainder, the
function $\omega\mapsto R_{\mathfrak h}(r\omega)$ and its angular
derivatives through order two are
$O(r^2\|u\|_{L_s^1(\R^n)})$.
The same estimate holds after applying $r\partial_r$.
The multiplier bound
$\Lambda:H^2(\Sph)\to H^1(\Sph)$ gives the displayed estimate.

Thus $\chi E_\vartheta=\chi r^{2s-\vartheta}D_{\mathrm{har}}\mathfrak h$ has
$H^1(\Cyl)$ norm at most $C_\chi \|u\|_{L_s^1(\R^n)}$, because
$2s+2-\vartheta>0$. The embedding $H^1\subset H^{1-s}$ proves
\eqref{E-bound}.
\end{proof}

The bounds for the defect and for $q_\vartheta$ may depend on the
fixed solution; only their finiteness is used. The error bound
\eqref{E-bound} is universal.
\begin{proposition}[Localized derivative equation and source average]\label{mellin-domain}
For $\bar u=\chi_0u$ and $q=D_{\mathrm{har}}\bar u$ constructed above,
the equations \eqref{q-equation} and \eqref{shifted-q-equation} hold
in distributions on $0<r<1$. The source identity
\eqref{intro-fractional-source} holds with $v=\bar u$ and $q_v=q$. In particular,
\[
 |\mathcal S_sq-F_s[f(u)]|
 \le Cr^{2s}\|u\|_{L_s^1(\R^n)},\qquad 0<r<1.
\]
\end{proposition}
\begin{proof}
Choose the noninteger $a_0>s$ from the proof of
Lemma~\ref{field-admissibility}. The localized function
$\bar u$ belongs to $C_c^{1+a_0}(B_{7/4})$.
Choose $s<\rho<\min\{a_0,1\}$. Theorem~\ref{intro-poisson-source}
then gives $q\in H^s(\Cyl)$ and the source identity for $\bar u$.
The function $\bar u$ itself need not be in $L^2(\Cyl)$ because
it tends to $u(0)$ as $t\to-\infty$.

Choose a smooth $\eta_L$ equal to zero for $t\le-2L$ and
one for $t\ge-L$, with $|\partial_t^k\eta_L|\le C_kL^{-k}$,
and put $\bar u_L=\eta_L\bar u$. Then
\[
 D_{\mathrm{har}}\bar u_L=\eta_Lq-\eta_L'\bar u
       \longrightarrow q\quad\hbox{in }H^s(\Cyl).
\]
The first term converges by the origin estimate. In the second,
the constant mode has $H^s$ norm $O(L^{-1/2})$; after subtracting
$u(0)$, the remaining terms decay exponentially as $t\to-\infty$.
For fixed $L$, smooth approximation on compact annuli in $H^{1+s}$
justifies the multiplier identities for $\bar u_L$.

For the unweighted intertwining identity, the order-$2s$ bound gives
\[
 \T_{s,0}D_{\mathrm{har}}\bar u_L
       \longrightarrow\T_{s,0}q\quad\hbox{in }H^{-s}(\Cyl).
\]
On each fixed compact annulus, $\bar u_L-\bar u$ is supported
near the origin for all sufficiently large $L$, and its Euclidean
$L^1$ norm tends to zero. The ordinary kernel of $(-\Delta)^s$
therefore gives $L\bar u_L-L\bar u\to0$ there with every fixed
derivative. Applying $D_{\mathrm{har}}$ and multiplying by $r^{2s}$
proves that \eqref{field-intertwining} passes to $\bar u$ in
distributions. This linear argument applies to every smooth compact
physical function, including functions that do not vanish at the origin.
To justify its nonlinear right-hand side,
mollify $\bar u$. On every smaller ball, convolution commutes with
$L$ and
\[
 L(\rho_\epsilon*\bar u)
       =\rho_\epsilon*(L\bar u)
       \longrightarrow f(u)+\mathfrak h\quad\hbox{in }C^1.
\]
The radial chain rule is therefore valid in the limit. On each sphere
the convergence also holds in $H^1$, so application of
$\Lambda:H^1\to L^2$ and Lemma~\ref{angular-defect} gives
\[
 D_{\mathrm{har}}(L\bar u)
 =f'(u)D_{\mathrm{har}}u-B[u]+D_{\mathrm{har}}\mathfrak h.
\]
Multiplication by $r^{2s}$ proves \eqref{q-equation}.

For the weighted identity, use the compact annular approximants
$\eta_Lq$ and $\eta_Lq_\vartheta$.
The origin estimate \eqref{field-origin-Hs} gives convergence to
$q$ and $q_\vartheta$, respectively, in $H^s(\Cyl)$; on the fixed
outer annulus, $C_c^{1+a_0}$ regularity gives the same convergence.

For these annular approximants, \eqref{cone} gives directly
\[
 \T_{s,\vartheta}(r^{-\vartheta}v)
          =r^{-\vartheta}\T_{s,0}v.
\]

The bounds
$\T_{s,0},\T_{s,\vartheta}:H^s\to H^{-s}$ allow passage to the limit,
and multiplication by $r^{-\vartheta}$ is smooth on every compact
cylinder set. This proves \eqref{shifted-q-equation}.

Finally, $L\bar u=f(u)+\mathfrak h$ on $B_1$, with
$\|\mathfrak h\|_{L^\infty(B_1)}\le C\|u\|_{L_s^1(\R^n)}$.
The bound \eqref{source-average-bound} and linearity of $F_s$ in
\eqref{intro-fractional-source} give the asserted source error estimate.
\end{proof}

Theorem~\ref{intro-poisson-source} and Proposition~\ref{mellin-domain}
express $\mathcal S_sq$ as the source average up to the controlled
tail error. The map
$\mathcal S_s:H^{s-1}(\Cyl)\to L^2(\Cyl)$ converts a negative
Sobolev bound for $q$ into an $L^2$ bound for this average.

% energy-identity, weight-interval, and field-intertwining.
\section{Kernel and localization estimates}\label{kernel-appendix}
We use the positive heat semigroup of the cone operator to control
interactions between opposite signs. Radially separated actions are
estimated by a differentiated Riesz kernel, and radial localization
is controlled by the extension energy.

Throughout this section, $n\ge2$, $0<s<1$, and
\[
 2s-n<\vartheta<\min\{1,2s\}.
\]
Constants may depend on $n,s,\vartheta$ and the fixed cutoff profiles.
We use $d\mu=r^{n+1}\dd r\dd\omega$ and the cone operator
$\mathcal L$ from \eqref{cone-form}.

\subsection{Positive heat interactions}\label{positive-kernel-proof}
\begin{lemma}[Positivity and radial restriction]\label{cone-positivity}

The heat semigroup of $\mathcal L$ preserves positivity and is
contractive on $L^\infty$.

For an interval $I\Subset(0,\infty)$, let $\mathcal L_I$ be the
realization of the same form with zero radial boundary values on $I$.
Its heat semigroup, extended by zero, satisfies
\begin{equation}\label{killed-heat-domination}
 0\le e^{-h\mathcal L_I}g\le e^{-h\mathcal L}g,
 \qquad g\ge0\text{ supported in }I\times\Sph,
 \quad h>0.
\end{equation}
\end{lemma}
\begin{proof}

If $C:\R\to\R$ is a contraction with
$C(0)=0$, the probability kernel of the spherical Poisson semigroup
gives
\[
 \langle C(h),(I-\mathsf P_a)C(h)\rangle
 \le \langle h,(I-\mathsf P_a)h\rangle.
\]
Indeed, each side is half the integral of the squared difference
of its two function values against that symmetric probability
kernel. Divide by $a$ and let $a\downarrow0$ to obtain the same
contraction inequality for the form of $\Lambda$. The Sobolev chain
rule gives it for both gradient terms in \eqref{cone-form}.
It passes to the closed form by approximation and lower
semicontinuity.

In particular, replacing a function by its positive part decreases
the functional
$\mathfrak a[h]+\lambda\|h-g\|_{L^2(d\mu)}^2$ when $g\ge0$.
The resolvent minimizer is therefore nonnegative. Truncation to
$[0,1]$ gives the analogous assertion for data in $[0,1]$.
The resolvent approximation of the exponential shows that the heat
semigroups of $\mathcal L$ and $\mathcal L_I$ preserve positivity
and do not increase the $L^\infty$ norm.

For the domination, let $u$ and $u_I$ solve the full and radial
Dirichlet resolvent equations with the same nonnegative data,
extended by zero outside $I$. Their difference has positive part
$w=(u_I-u)_+$ with zero radial trace on $\partial I$, since
$u\ge0$. Thus $w$ is an admissible test in both equations. The
gradient terms and the positive Poisson form give
\[
 \mathfrak a[u_I-u,w]\ge\mathfrak a[w],
\]
because the form pairing of two nonnegative functions with
disjoint supports is nonpositive. Subtracting the equations yields
$\mathfrak a[w]+\lambda\|w\|_2^2\le0$, hence $u_I\le u$.
The same comparison holds if the full equation has additional
nonnegative data outside $I$. Iterate this resolvent inequality
and pass to the exponential to obtain
\eqref{killed-heat-domination}.

\end{proof}

For nonnegative functions $a,b$ on the cylinder, define
\begin{equation}\label{positive-heat-interaction}
 \mathcal I_\vartheta(a,b)
  =c_s\int_0^\infty
   \left\langle r^{2s-n-\vartheta}a,
       e^{-h\mathcal L}(r^{\vartheta-2}b)\right\rangle_{L^2(d\mu)}
                           \frac{\dd h}{h^{1+s}}.
\end{equation}
Initially $a,b$ are bounded and supported in compact annuli.
For arbitrary nonnegative functions, the right side means the
increasing limit obtained by replacing each function $a$ by
$\min\{a,k\}\mathbf1_{\{e^{-k}<r<e^k\}}$, and similarly for $b$.
Positivity and the $L^2$ continuity of the heat operators show that
this definition is independent of the increasing bounded annular
approximation: take its pointwise minimum with each fixed bounded
annular truncation, pass to the limit in $L^2$ at each heat
time, and then use monotone convergence in the nonnegative time
integral. It is additive and monotone in each argument, with
the value $+\infty$ allowed. In particular,
\begin{equation}\label{interaction-order}
 0\le a\le a',\quad 0\le b\le b'
 \quad\Longrightarrow\quad
 \mathcal I_\vartheta(a,b)\le\mathcal I_\vartheta(a',b').
\end{equation}
The order of the two arguments is kept fixed throughout.

\begin{lemma}[Cross interactions and admissible tests]\label{cross-testing}
Let $v\in H^s(\Cyl)$ be real, vanish for all sufficiently
large $t$, and put $w=v_+$, $z=v_-$. For every smooth
$0\le\chi\le1$ which vanishes for large $t$ and has bounded derivatives,
\begin{equation}\label{cross-finite}
 -\langle\chi^2w,\T_{s,\vartheta}z\rangle
       =\mathcal I_\vartheta(\chi^2w,z)<\infty.
\end{equation}
This quantity is nonnegative. If the physical support of $\chi v$
is compactly contained in $B_2$, the stability test
$r^{-\beta}\chi w$ is admissible by approximation.
\end{lemma}
\phantomsection\label{cross-testing-details}
\begin{proof}
First suppose $v$ has compact cylinder support. Every $H^s$ function
$g$ supported in a fixed compact annulus belongs to
$D(\mathcal L^{s/2})$. Indeed, for smooth such $g$, \eqref{cone}
with $\vartheta=0$ gives
\[
 \|\mathcal L^{s/2}g\|_{L^2(d\mu)}^2
 =\langle r^{n-2s}g,\T_{s,0}(r^2g)\rangle
 \le C\|g\|_{H^s(\Cyl)}^2.
\]
Smooth annular approximation and closedness of the fractional form
prove the assertion. Since positive and negative parts preserve
$H^s$ for $0<s<1$, the functions
\[
 a=r^{2s-n-\vartheta}\chi^2w,
 \qquad b=r^{\vartheta-2}z
\]
belong to that form domain and satisfy $a,b\ge0$, $ab=0$.
The spectral form version of \eqref{prelim-positive-power} gives
\[
 -\langle a,\mathcal L^sb\rangle_{d\mu}
 =c_s\int_0^\infty
    \langle a,e^{-h\mathcal L}b\rangle_{d\mu}\frac{\dd h}{h^{1+s}}.
\]
Indeed, Cauchy--Schwarz for $I-e^{-h\mathcal L}$ and then for its
time integral bounds the absolute integral of its form pairing by
the product of the finite fractional energies of $a,b$.
The left side agrees with the cylinder pairing in
\eqref{cross-finite}, by smooth approximation and \eqref{cone}.
The right side is \eqref{positive-heat-interaction}: truncation
converges in $L^2$ for each heat time and then monotonically in the
time integral. This proves the identity and finiteness.

For general $v$, take smooth $0\le\rho_m\uparrow1$ with
$\rho_m=0$ on $t\le-m-1$ and $\rho_m=1$ on $t\ge-m$.
Then $\rho_mv\to v$ in $H^s$, and
$(\rho_mv)_\pm=\rho_mv_\pm\to v_\pm$ in $H^s$.
Apply the compact argument to $\rho_mv$. The interactions
$\mathcal I_\vartheta(\chi^2\rho_mw,\rho_mz)$ increase to
$\mathcal I_\vartheta(\chi^2w,z)$, whereas the cylinder pairings
converge by the boundedness of $\T_{s,\vartheta}$. This proves
\eqref{cross-finite} at the infinite inner end.

Finally let $a=\chi w$. The functions $r^{-\beta}\rho_ma$ have
compact physical support away from the origin and admit smooth
approximations in energy. By \eqref{energy-identity}, their energies
converge to $\langle a,\mathcal A_sa\rangle$; their physical $L^2$
norms converge because $r^{2s}$ is bounded on their common outer
support. The compact local range hypothesis bounds $f'(u)$ there,
so the potential terms converge as well. Stability therefore passes
to $r^{-\beta}a$.
\end{proof}

\subsection{Separated radial actions}\label{separated-kernel-section}

For separated radial supports, the differentiated Riesz kernel gives
both exponential decay and Sobolev smoothing. We use this formula
only between separated radii.

\begin{lemma}[Estimates between separated radial supports]\label{radial-kernel-details}
Put $D=D_{\mathrm{har}}$ and let
\[
 I_{2-2s}(x)=c_{n,2-2s}|x|^{-n+2-2s}
\]
be normalized so that convolution with $I_{2-2s}$ is
$(-\Delta)^{s-1}$. For radially separated arguments, the kernel of
$\T_{s,\vartheta}$ relative to cylinder measure is
\begin{equation}\label{radial-kernel-series}
 K_\vartheta(X,Y)
 =r_X^{2s-\vartheta}r_Y^{n+\vartheta-2}
       D_XD_Y I_{2-2s}(r_X\omega_X-r_Y\omega_Y).
\end{equation}
This formula defines the separated action for every
$-n<\vartheta<2s+2$. For each fixed number of cylinder derivatives
and every $0<h\le1$, there are $C,M<\infty$ such that, whenever
$|t_X-t_Y|\ge h$,
\begin{equation}\label{separated-kernel-bound}
 |\nabla_X^k\nabla_Y^\ell K_\vartheta(X,Y)|
 \le Ch^{-M}
 \begin{cases}
 e^{-(2s+2-\vartheta)(t_Y-t_X)},&t_X<t_Y,\\
 e^{-(n+\vartheta)(t_X-t_Y)},&t_Y<t_X.
 \end{cases}
\end{equation}
Here derivatives may be taken in a fixed finite collection of
angular charts. The transpose is the kernel with weight
$\vartheta^*=2s+2-n-\vartheta$. In particular both the operator and
its adjoint map $H^{-1}$ to $H^1$ between radially separated sets,
with polynomial dependence on the inverse separation. The same
assertions hold when the inner set reaches $t=-\infty$.
\end{lemma}
\begin{proof}
For $F\in C_c^\infty(\Cyl)$ define the right inverse
\[
 D^{-1}F(r\omega)
   =\int_r^\infty P_{r/\rho}[F(\rho\,\cdot)](\omega)
                         \frac{d\rho}{\rho}.
\]
Differentiating the Poisson semigroup gives $DD^{-1}F=F$.
Below the radial support of $F$ this function is harmonic, by
the semigroup law for the ball Poisson operators; above that
support it is zero. Its harmonic expansion in the inner ball
and differentiation on the intervening annulus show that
$V=D^{-1}F$ belongs to $C_c^\infty(\R^n)$.

The factorization of the ordinary polar Laplacian gives
\[
 -\Delta V
 =r^{-2}(\partial_t+\Lambda+n-2)DV
 =r^{-2}(\partial_t+\Lambda+n-2)F.
\]
Since $n\ge2>2-2s$, the Riesz potential above is locally
integrable and its usual Fourier identity yields
\[
 (-\Delta)^sV=I_{2-2s}*(-\Delta V).
\]
The intertwining identity for the corrected derivative, with
$F=r^\vartheta v$, therefore gives
\[
 \T_{s,\vartheta}v
 =r^{2s-\vartheta}D(-\Delta)^sD^{-1}(r^\vartheta v).
\]
The inner-cutoff argument in Proposition~\ref{mellin-domain}
extends this identity to smooth compact physical functions, including
the function $V$ above, which need not vanish at the origin. We first use the
identity at the admissible weights. The wider interval in the
statement refers to the separated kernel and its transpose.
For the separated action every kernel in this expression is
smooth on the entire pair of spheres. Write
$dy=r_Y^n\,dt_Y\,d\omega_Y$. Integration by parts in $t_Y$
and selfadjointness of $\Lambda_Y$ give
\begin{align*}
 &\int I_{2-2s}(x-y)r_Y^{-2}
       (\partial_{t_Y}+\Lambda_Y+n-2)F(Y)\,dy\\
 &\qquad=\int_\Cyl r_Y^{n-2}
       (\Lambda_Y-\partial_{t_Y})I_{2-2s}(x-y)F(Y)\,dY.
\end{align*}
The derivative of $r_Y^{n-2}$ cancels the term $n-2$.
There is no radial boundary term because $F$ has compact radial
support. Applying $D_X$ proves \eqref{radial-kernel-series}.

We give the derivative and decay estimates explicitly. For a
smooth function on the sphere,
\[
 \|\Lambda g\|_{C^k(\Sph)}\le C_k\|g\|_{C^{k+2}(\Sph)}.
\]
Indeed, differentiate the ball Poisson formula at the boundary
after subtracting the value and tangential linear Taylor term
at the observation point. The remainder is quadratic, hence
integrable against the resulting kernel of order
$|\omega-\eta|^{-n}$; the linear term integrates to zero by
reflection symmetry. Applying the same argument to rotation
derivatives proves the displayed estimate for every $k$.
Consequently, on two radially separated spheres, every fixed
number of applications of $D_X,D_Y$ and cylinder derivatives
is bounded by finitely many ordinary derivatives of the Riesz
kernel. On $h\le|t_X-t_Y|\le\log2$ these derivatives cost at
most a fixed power of $h^{-1}$, since
$|r_X\omega_X-r_Y\omega_Y|\ge|r_X-r_Y|$.

For $r_X\le r_Y/2$, expand the Riesz kernel in $x$ at zero.
The constant and linear Taylor terms are harmonic polynomials,
and $D_X$ annihilates them. Its remainder, after any fixed
number of normalized radial or angular derivatives, is bounded by
\[
 C r_X^2r_Y^{-n-2s}.
\]
The same estimate holds after applying $D_Y$: differentiation
in the outer variable preserves the homogeneity, and the
preceding spherical estimate controls its angular action.
Multiplication by the powers in \eqref{radial-kernel-series}
gives $(r_X/r_Y)^{2s+2-\vartheta}$.
For $r_Y\le r_X/2$, expand instead in $y$. The factor $D_Y$
removes the constant and linear terms, leaving
$Cr_Y^2r_X^{-n-2s}$; the resulting bound is
$(r_Y/r_X)^{n+\vartheta}$. These Taylor arguments also apply
after all prescribed cylinder derivatives. Homogeneity removes
the common radial scale and proves \eqref{separated-kernel-bound}.
Both exponents are positive precisely in the stated weight
interval. Symmetry of $I_{2-2s}(x-y)$ and commutation of $D_X,D_Y$
give $K_\vartheta(Y,X)=K_{\vartheta^*}(X,Y)$ directly.

Let $K$ include separated smooth cutoff factors, whose derivatives
cost powers of a width $h_*$. The decay just proved implies, for
$a,b\in\{0,1\}$,
\begin{equation}\label{separated-kernel-Schur}
 \sup_X\int|\nabla_X^a\nabla_Y^bK(X,Y)|\,dY
 +\sup_Y\int|\nabla_X^a\nabla_Y^bK(X,Y)|\,dX
 \le Ch_*^{-M}.
\end{equation}
Angular integration is over the compact sphere and radial
integration is bounded by the two exponential decays. This
also proves the assertion at the infinite inner end.
Write $v=f-\operatorname{div}F$, where
$f=(1-\Delta_\Cyl)^{-1}v$ and $F=\nabla f$. Then
\[
 \|f\|_2+\|F\|_2\le C\|v\|_{H^{-1}}.
\]
Integrating by parts in $Y$ and using the $L^2$ Schur estimate
from \eqref{separated-kernel-Schur}, also after one $X$
derivative, gives
\[
 \|Kv\|_{H^1}\le Ch_*^{-M}\|v\|_{H^{-1}},\qquad
 \|Kv\|_{H^{1-s}}\le Ch_*^{-M}\|v\|_{H^{s-1}}.
\]
The second bound uses $H^{s-1}\subset H^{-1}$ and
$H^1\subset H^{1-s}$. Density completes the proof.
\end{proof}

\subsection{Cutoffs and separated tails}\label{radial-cutoff-proof}\label{separated-tails-proof}

Radial localization is controlled directly by the extension energy.
For a real radial cutoff $\chi$ and a real $v\in H^s(\Cyl)$, write
\[
 \mathfrak C_\chi[v]
 =\Ree\langle\chi v,\T_{s,\vartheta}(\chi v)\rangle
       -\Ree\langle\chi^2v,\T_{s,\vartheta}v\rangle.
\]

\begin{lemma}[Radial localization and commutator estimates]\label{cutoff}
Let $1/2\le R<R'\le3/4$ and let $v\in H^s(\Cyl)$ be real-valued
and vanish for $r\ge R'$. Choose smooth radial cutoffs
$0\le\chi,\psi\le1$, with $\chi=1$ for $r\le R$,
$\operatorname{supp}\chi\subset\{r<R'\}$, $\psi=1$ for $r\le1/4$,
and $\operatorname{supp}\psi\subset\{r<3/8\}$.
Choose $\chi$ from a fixed smooth transition profile, so that
$|\partial_t^k\chi|\le C_k(R'-R)^{-k}$ for every required derivative,
with the profile constants $C_k$ fixed; choose $\psi$ once and for all.
Then, for every $0<\varepsilon\le1$,
\begin{equation}\label{C-bound}
 |\mathfrak C_\chi[v]|
 \le\varepsilon\|\chi v\|_{H^{-\sigma}}^2
   +C_\varepsilon(R'-R)^{-p}
                 \|(1-\psi)v\|_{H^{-\sigma}}^2.
\end{equation}
The finite exponent $p$ is independent of $\varepsilon$. All constants
grow at most polynomially in $(R'-R)^{-1}$; $C_\varepsilon$ may also
depend on the fixed weight.
\end{lemma}
\begin{proof}
We first establish the auxiliary estimate
\begin{equation}\label{quantitative-commutator}
 |\mathfrak C_\chi[v]|
 \le C_{n,s,\vartheta}\|\chi\|_{W^{2,\infty}(\R)}^2
                         \|v\|_{H^{s-1}(\Cyl)}^2
\end{equation}
for $v\in H^s(\Cyl)$ and every bounded smooth real cutoff whose
first two derivatives are bounded and have compact support.
Use the extension $P_\gamma$ and the forms $\mathcal E_\gamma,\mathcal B_\gamma$
from Lemma~\ref{cylinder-extension}, with $\gamma=\beta-1+\vartheta$.
Multiplication by the present cutoffs preserves their energy domain
and its zero-trace subspace.

Set $W=P_\gamma v$, $U=P_\gamma(\chi v)$, and $Z=U-\chi W$. The function $Z$
has zero trace. For every zero-trace test $\psi$,
$\mathcal B_\gamma(W,\chi\psi)=\mathcal B_\gamma(U,\psi)=0$.
Expanding these identities and integrating the term containing $W_t$
by parts in $t$ gives
\[
 \mathcal B_\gamma(Z,\psi)
 =-\int pW\,[2\chi'\psi_t+(\chi''-2\gamma \chi')\psi].
\]
Taking $\psi=Z$ and using coercivity proves
\[
 \mathcal E_\gamma[Z]
 \le C(\|\chi'\|_\infty+\|\chi''\|_\infty)^2
                         \|W\|_{L^2(p)}^2.
\]

The identity $\mathcal B_\gamma(U,Z)=0$ also gives
\[
 \mathcal E_\gamma[U]-\mathcal E_\gamma[\chi W]
   =-\mathcal E_\gamma[Z]+4\gamma \int p(\chi W)_tZ
   =-\mathcal E_\gamma[Z]-4\gamma \int p\chi WZ_t.
\]
On the other hand, expansion of $\mathcal B_\gamma(W,\chi^2W)$ yields
\[
 \mathcal E_\gamma[\chi W]-\mathcal B_\gamma(W,\chi^2W)
       =\int p(\chi'^2-2\gamma \chi\chi')W^2.
\]
Combining these two equalities with the boundary Green identity gives
\[
 d_s^{-1}\mathfrak C_\chi[v]
   =\int p(\chi'^2-2\gamma \chi\chi')W^2
      -\mathcal E_\gamma[Z]-4\gamma \int p\chi WZ_t.
\]
Finally, the extension mass estimate \eqref{cylinder-extension-mass} gives
$\|W\|_{L^2(p)}\le C\|v\|_{H^{s-1}}$.
Together with coercivity and the bound for $Z$, this proves
\eqref{quantitative-commutator}.

The argument starts with smooth compact cylinder traces. Bounded
cutoff multiplication and approximation in the energy domain justify
the zero-trace tests and the integrations by parts. Approximation of
$v$ in $H^s$ then proves the estimate as stated. These arguments also
allow $\chi$ to tend to one at $t=-\infty$. Finally, the first two
derivatives of a fixed transition profile have polynomial bounds in
$(R'-R)^{-1}$, giving the asserted polynomial cutoff dependence.

To obtain \eqref{C-bound}, put $a=\psi v$ and $b=(1-\psi)v$. Since $\chi a=a$ and
$\chi^2a=a$, one has $\mathfrak C_\chi[a]=0$. Estimate~\eqref{quantitative-commutator}
bounds $\mathfrak C_\chi[b]$ by
$C(R'-R)^{-M}\|b\|_{H^{-\sigma}}^2$.
The remaining mixed expression is exactly
\[
 \Ree\langle a,\T_{s,\vartheta}((\chi-1)b)\rangle
       +\Ree\langle(\chi-\chi^2)b,\T_{s,\vartheta}a\rangle.
\]
Each nonzero term has one radius at most $3/8$ and the other
greater than $R\ge1/2$, hence a radial separation of at least $1/8$.
The separated estimates \eqref{separated-kernel-Schur} for the operator
and its adjoint, together with smooth cutoff multiplication, bound it by
\[
 C(R'-R)^{-M}\|\chi v\|_{H^{-\sigma}}
                         \|(1-\psi)v\|_{H^{-\sigma}}.
\]
Here $a=\psi\chi v$, and the fixed multiplication by $\psi$ is bounded
on $H^{-\sigma}$. More generally, each cutoff multiplication used
above has norm on $H^{\pm\sigma}$ bounded by a fixed power of its
$W^{1,\infty}$ norm: use the product rule on $H^1$, the $L^2$ bound,
interpolation for $0<\sigma<1$, and duality for the negative order.
Young's inequality proves \eqref{C-bound}, with
an exponent independent of $\varepsilon$.
\end{proof}

\begin{lemma}[Tail estimates for separated radial supports]
Let $\bar u=\chi_0u$ be the localization in
Section~\ref{regularity-appendix}, and fix $\vartheta$ in
\eqref{weight-interval}.\label{separated-source-tails} Let $1/2\le R<R'\le3/4$ and set
$h_*=R'-R$. Choose smooth radial cutoffs $0\le\chi,\eta\le1$ such that
\[
 \begin{split}
 &\chi=1\text{ on }r\le R,\qquad
   \operatorname{supp}\chi\subset\{r<R+h_*/3\},\\
 &\eta=1\text{ on }r\le R+2h_*/3,\qquad
   \operatorname{supp}\eta\subset\{r<R'\},
 \end{split}
\]
with $k$th cylinder derivatives bounded by $C_kh_*^{-k}$.
Then, for some finite $p>0$,
\begin{equation}\label{separated-field-tail}
 \bigl\|\chi\T_{s,\vartheta}
       [ (\eta-1)r^{-\vartheta}D_{\mathrm{har}}\bar u]\bigr\|_{H^{1-s}(\Cyl)}
 \le Ch_*^{-p}\|u\|_{L_s^1(\R^n)}.
\end{equation}
If $\chi_{\rm src}$ is a fixed smooth radial cutoff equal to one on
$r\le1/4$ and supported in $r<1/2$, then
\begin{equation}\label{separated-smoothing-tail}
 \left|r^{-\vartheta}\mathcal S_s
       [(1-\chi_{\rm src})D_{\mathrm{har}}\bar u](r\omega)\right|
 \le Cr^{2s-\vartheta}\|u\|_{L_s^1(\R^n)},\qquad 0<r<1/8.
\end{equation}
The constants depend only on $n,s,\vartheta$ and the cutoff profiles.
\end{lemma}
\begin{proof}
Write $D=D_{\mathrm{har}}$, $a=\eta-1$, and $a'=\partial_ta$.
The product rule $aD\bar u=D(a\bar u)+a'\bar u$ and the
intertwining identity give
\begin{equation}\label{tail-kernel-transfer}
 \chi\T_{s,\vartheta}(r^{-\vartheta}aD\bar u)
 =\chi r^{2s-\vartheta}D(-\Delta)^s(a\bar u)
       +\chi\T_{s,\vartheta}(r^{-\vartheta}a'\bar u).
\end{equation}
Thus the estimates require no derivative of $u$ on the outer
annulus. Both source factors are supported in the fixed annulus
$A_0=\{1/4<r<2\}$, at radial distance at least $h_*/3$ from
$\operatorname{supp}\chi$.

For the first term, $a\bar u$ vanishes near every observation
sphere. Hence the ordinary fractional Laplacian is given there by
\[
 (-\Delta)^s(a\bar u)(x)
 =-c_{n,s}\int_{A_0}\frac{a(y)\bar u(y)}{|x-y|^{n+2s}}\,dy.
\]
The derivative bounds for this separated Euclidean kernel cost
only powers of $h_*^{-1}$. Applying $D_X$ uses the spherical
$C^k$ estimate in the preceding proof. Near $x=0$ it removes
the constant and linear Taylor terms, so the resulting kernel
and each required cylinder derivative are bounded by
$Ch_*^{-M}r_X^2$, uniformly for $y\in A_0$.
After multiplication by $r_X^{2s-\vartheta}$ and $\chi$, their
$L^2(dX)$ norms are at most $Ch_*^{-M}$, because
$2s+2-\vartheta>0$.

For the second term of \eqref{tail-kernel-transfer}, use
\eqref{radial-kernel-series}. Estimate
\eqref{separated-kernel-bound} gives, uniformly for $Y\in A_0$,
an $H^1(dX)$ bound $Ch_*^{-M}$ for
$\chi(X)K_\vartheta(X,Y)$. In the source factor,
$|r_Y^{-\vartheta}a'(Y)|\le Ch_*^{-1}$.
Minkowski's inequality in these two kernel representations yields
\begin{equation}\label{tail-kernel-H1}
 \|\chi\T_{s,\vartheta}(r^{-\vartheta}aD\bar u)\|_{H^1(\Cyl)}
 \le Ch_*^{-p}\int_{A_0}|\bar u(y)|\,dy
 \le Ch_*^{-p}\|u\|_{L_s^1(\R^n)}.
\end{equation}
Cylinder and Euclidean measures, as well as the tail weight,
are comparable on $A_0$. The inclusion $H^1\subset H^{1-s}$ proves
\eqref{separated-field-tail}.

For the second assertion put $b=1-\chi_{\rm src}$ and
$b'=\partial_tb$. The same product rule and the positive-source
identity \eqref{intro-fractional-source} give
\[
 \mathcal S_s(bD\bar u)
   =F_s[(-\Delta)^s(b\bar u)]+\mathcal S_s(b'\bar u).
\]
Since $b\bar u$ vanishes on $B_{1/4}$, the preceding separated
Euclidean kernel formula gives
\[
 \|(-\Delta)^s(b\bar u)\|_{L^\infty(B_{1/8})}
 \le C\|u\|_{L_s^1}.
\]
By \eqref{source-average-bound}, the first term is bounded by
$Cr^{2s}\|u\|_{L_s^1}$.

In the second term, $b'$ is supported in
$1/4\le r_Y\le1/2$. For $r<1/8$ the ratio
$\rho=r/r_Y$ is at most $1/2$. The Poisson kernel is uniformly
bounded for these ratios, and the scalar weight in
\eqref{intro-poisson-averaging} is bounded by $C\rho^{2s}$.
Changing variables from $h$ to $t_Y$ therefore gives
\[
 |\mathcal S_s(b'\bar u)(r\omega)|
 \le Cr^{2s}\int_{A_0}|\bar u(Y)|\,dY
 \le Cr^{2s}\|u\|_{L_s^1}.
\]
Multiplication by $r^{-\vartheta}$ proves
\eqref{separated-smoothing-tail}. All identities were proved first
for smooth outer data. Approximation on the fixed annulus,
the qualitative regularity of $\bar u$, and the displayed estimates
justify passage to the localized solution.
\end{proof}

\section{Stability and interior regularity}\label{localization-section}
Throughout this section $n\ge2$. A lower bound for the difference of the
quadratic forms of $\T_{s,\vartheta}$ and $\mathcal A_s$, combined with
stability, yields a weighted bound for $F_s[f(u)]$.

For the source estimate and the reduction below, set
\begin{equation}\label{annular-energy-definition}
 A_*=B_{15/8}\setminus\overline B_{1/8},\qquad
 M_s(u)^2=[u]_{H^s(A_*)}^2+\|u\|_{L_s^1(\R^n)}^2.
\end{equation}
\begin{proposition}[Weighted source estimate]\label{weighted-partial}
Let $n\ge2$, $0<s<1$, and let $u$ have the solution and nonlinearity
assumptions of Theorem~\ref{full}, without a dimension restriction.
Suppose $\vartheta$ satisfies \eqref{weight-interval} and
\begin{equation}\label{weighted-gap}
 \Ree t_{j,\vartheta}(\tau)-a_j(\tau)
 \ge\kappa_\vartheta(1+j^2+\tau^2)^{s-1}
 \qquad(j\ge0,\ \tau\in\R),\qquad\kappa_\vartheta>0.
\end{equation}
With $M_s(u)$ defined in \eqref{annular-energy-definition}, one has
\begin{equation}\label{general-partial}
 \int_{B_{1/8}}|x|^{-n-2\vartheta}|F_s[f(u)](x)|^2\dd x
 \le C M_s(u)^2,
\end{equation}
where $C$ depends only on $n,s,\vartheta,\kappa_\vartheta$.
\end{proposition}

\subsection{The weighted source estimate}

Positivity of $\mathcal S_{s,\vartheta}$ allows us to bound $F_s[f(u)]$ using
the positive part of $q_\vartheta$. The annular energy, however, controls
$q_\vartheta=r^{-\vartheta}q$ only in $H^{s-1}$. Passing to its
positive part can destroy cancellation, and this operation is not
bounded in a negative Sobolev norm.

For example, let $0\le\eta\in C_c^\infty(\R)$ be nonzero and set
\[
 v_m(t,\omega)=\eta(t)\sin(mt).
\]
Recall $\sigma=1-s\in(0,1)$.
The Fourier transform in $t$ is a difference of translates of
$\widehat\eta$ centered at $\pm m$. Splitting into neighborhoods of
these two frequencies and their complement, where $\widehat\eta$
decays rapidly, gives $\|v_m\|_{H^{-\sigma}(\Cyl)}\to0$.
On the other hand, the $2\pi$-periodic function $(\sin t)_+$ has
mean $1/\pi$. Its nonconstant Fourier modes tend to zero in
distributions after replacing $t$ by $mt$, so
$(v_m)_+\rightharpoonup\eta/\pi$. Pairing with $\eta$ shows that
their $H^{-\sigma}$ norms do not tend to zero.
The positive part in this example has acquired a nonzero
low-frequency component. In the proof, its cost is controlled by
the interaction between the positive and negative parts. That
interaction has a favorable sign in the tested equation for $q_\vartheta$,
and we retain it when applying stability.

Fix henceforth a weight $\vartheta$ in \eqref{weight-interval}.
Quadratic forms use $\Ree\T_{s,\vartheta}$; the tested equation uses
the full operator $\T_{s,\vartheta}$.

\begin{lemma}[Positive parts in a negative Sobolev space]\label{rectification}
Let $v\in H^s(\Cyl)$ be real-valued and supported in a fixed
compact cylinder annulus $Q$. For every $0<\varepsilon\le1$,
\begin{equation}\label{rectification-bound}
 \|v_+\|_{H^{-\sigma}}^2
 \le C\varepsilon^{-\sigma/s}\|v\|_{H^{-\sigma}}^2
       +\varepsilon\mathcal I_\vartheta(v_+,v_-).
\end{equation}
Here $\sigma=1-s$, and the constants may depend on $Q$ and the
fixed weight. The interaction is finite by Lemma~\ref{cross-testing}.
\end{lemma}
\begin{proof}
Choose a radial interval $I\Subset(0,\infty)$ whose interior
contains $Q$, and let $A=\mathcal L_I$ on $L^2(I\times\Sph,d\mu)$.
Its form domain has zero radial trace. On this fixed interval the
weights in \eqref{cone-form} are bounded above and below. Also
\[
 0\le\langle h,\Lambda h\rangle
   =\sum_{j,\ell}j|h_{j,\ell}|^2
   \le C\bigl(\|\nabla_{\Sph}h\|_2^2+\|h\|_2^2\bigr).
\]
The radial Poincare inequality consequently gives
\begin{equation}\label{annular-form-equivalence}
 c\|h\|_{H^1(I\times\Sph)}^2
  \le\langle h,Ah\rangle
  \le C\|h\|_{H^1(I\times\Sph)}^2,
 \qquad h\in H^1_0(I;L^2(\Sph))\cap L^2(I;H^1(\Sph)).
\end{equation}
In particular $A\ge c>0$.

This form comparison also controls the negative norm. For $0<\rho<1$,
by
\eqref{quadratic-interpolation}, a map bounded on $L^2$ and on the
form domains is bounded on the fractional domains. Apply this to
zero extension from $I\times\Sph$ to
$\Cyl$, and to restriction after multiplication by a fixed smooth
cutoff supported in $I$ and equal to one on $Q$. Both maps are
bounded on $L^2$ and on the form domains by
\eqref{annular-form-equivalence}. Thus zero extension is bounded
from $D(A^{\rho/2})$ to $H^\rho(\Cyl)$, and cutoff restriction
is bounded in the reverse direction. These assertions also account
for the smooth bounded change between $d\mu$ and cylinder measure.
Duality, with the cutoff equal to one on the support, now gives
\begin{equation}\label{annular-negative-equivalence}
 \|a\|_{H^{-\rho}(\Cyl)}^2
  \asymp\langle a,A^{-\rho}a\rangle_{L^2(d\mu)},
 \qquad \operatorname{supp}a\subset Q.
\end{equation}
More explicitly, the dual pairing on the right integrates against
$r^{n+2}\dd t\dd\omega$. Multiplication by this weight and its
inverse, localized to $I$, is bounded on $H^\rho$ and its dual by
the same endpoint argument. It therefore does not change the
equivalence. The preceding argument with $\rho=s$ also shows that
the positive and negative parts of $v$ belong to $D(A^{s/2})$.

Put $w=v_+$, $z=v_-$ and $G_M=(M^2+A)^{-\sigma}$ for $M\ge1$.
All pairings until the last step use $d\mu$. Since the heat
semigroup of $A$ is positive and $wz=0$, the fractional-power
formulas give
\begin{align*}
 \langle w,G_Mz\rangle
  &=\frac1{\Gamma(\sigma)}\int_0^\infty
       h^{-s}e^{-M^2h}\langle w,e^{-hA}z\rangle\dd h,\\
 I_A(w,z):=-\langle w,A^sz\rangle
  &=c_s\int_0^\infty
       h^{-1-s}\langle w,e^{-hA}z\rangle\dd h\ge0.
\end{align*}
The second pairing is a form pairing and is finite. The identities
follow first by restricting the heat integrals to a compact time
interval and then by monotone convergence and the spectral form
formula. The elementary bound $he^{-M^2h}\le(eM^2)^{-1}$ yields
\begin{equation}\label{mass-domination}
 \langle w,G_Mz\rangle\le CM^{-2}I_A(w,z).
\end{equation}

Positivity and $v=w-z$ imply
\[
 \langle w,G_Mw\rangle
 \le\langle v,G_Mv\rangle+2\langle w,G_Mz\rangle.
\]
Since $A\ge c>0$, comparison at each spectral value gives
$G_M\le A^{-\sigma}\le CM^{2\sigma}G_M$. Therefore
\[
 \langle w,A^{-\sigma}w\rangle
 \le CM^{2\sigma}\langle v,A^{-\sigma}v\rangle
                  +CM^{-2s}I_A(w,z).
\]

Finally, the two weights $r^{2s-n-\vartheta}$ and
$r^{\vartheta-2}$ have positive lower bounds on $Q$. Positivity
and \eqref{killed-heat-domination} therefore give, for every $h>0$,
\[
 \langle w,e^{-hA}z\rangle_{d\mu}
 \le \langle w,e^{-h\mathcal L}z\rangle_{d\mu}
 \le C_Q\left\langle r^{2s-n-\vartheta}w,
       e^{-h\mathcal L}(r^{\vartheta-2}z)\right\rangle_{d\mu}.
\]
Integrating in $h$ proves
\[
 I_A(w,z)\le C_Q\mathcal I_\vartheta(w,z).
\]
Use \eqref{annular-negative-equivalence} with $\rho=\sigma$ and
choose $M$ to be a sufficiently large fixed multiple of
$\varepsilon^{-1/(2s)}$. The interaction coefficient is then at
most $\varepsilon$, while $M^{2\sigma}$ has the required size
$C\varepsilon^{-\sigma/s}$. This proves
\eqref{rectification-bound}.
\end{proof}

We now test the equation for $q_\vartheta$ against $(q_\vartheta)_+$ and
use stability to remove the potential term. The coercivity estimate controls the negative Sobolev norm,
while the retained interaction controls the cost of taking positive
parts on the outer annulus. The cutoff and tail estimates then close
an iteration between concentric annuli.

\begin{proof}[Proof of Proposition~\ref{weighted-partial}]
Recall $q_\vartheta=r^{-\vartheta}q$ and set
\[
 w=(q_\vartheta)_+,\qquad z=(q_\vartheta)_-.
\]
Equation~\eqref{shifted-q-equation} reads
\[
 \T_{s,\vartheta}q_\vartheta
   =Vq_\vartheta-B_\vartheta+E_\vartheta,
\]
with
$V=r^{2s}f'(u)$ and $B_\vartheta\ge0$.

We first bound $q_\vartheta$ on the annulus where the cutoffs
vary. Fix a compact annulus $K\Subset A_*$ containing
$\{1/4\le r\le3/4\}$, and let the smooth radial cutoff $\varphi$
be supported in $K$. Put $b=r^{-\vartheta}\varphi$. The product rule gives
\[
 \varphi q_\vartheta=D_{\mathrm{har}}(b\bar u)
                              +(\partial_t b)\bar u.
\]
Here $D_{\mathrm{har}}:H^s\to H^{s-1}$ by
$|j-\ii\tau|\le C\langle(j,\tau)\rangle$, and $L^2\subset H^{s-1}$.
The quantitative localization in Lemma~\ref{prelim-localization},
including the fixed cutoff defining $\bar u$, therefore proves
\begin{equation}\label{signed-q-energy}
 \begin{aligned}
 \|\varphi q_\vartheta\|_{H^{s-1}(\Cyl)}^2
 &\le C\bigl(\|\varphi\|_\infty+\|\partial_t\varphi\|_\infty\bigr)^2
         \|u\|_{H^s(A_*)}^2\\
 &\le C\bigl(\|\varphi\|_\infty+\|\partial_t\varphi\|_\infty\bigr)^2M_s(u)^2.
 \end{aligned}
\end{equation}
The constants depend on $n,s,\vartheta,K$.
In particular a cutoff with transition width $h_*$ and first derivative
bounded by $C/h_*$ costs at most $Ch_*^{-2}$ in this squared estimate.

For $1/2\le R<R'\le3/4$, put $h_*=R'-R$.
Take $\chi=1$ on $r\le R$, supported in $r<R+h_*/3$, and
$\eta=1$ on $r\le R+2h_*/3$, supported in $r<R'$, with values in
$[0,1]$. Localize $q_\vartheta$ before taking its positive and negative parts:
\[
 q_\eta=\eta q_\vartheta,\qquad
 w_\eta=\eta w,\qquad z_\eta=\eta z.
\]
On $\operatorname{supp}\chi$,
\begin{equation}\label{truncated-q-equation}
 \T_{s,\vartheta} q_\eta =V q_\eta -B_\vartheta+E_\eta,
 \qquad E_\eta=E_\vartheta+\T_{s,\vartheta}((\eta-1) q_\vartheta).
\end{equation}
Apply \eqref{separated-field-tail} with
$q_\vartheta =r^{-\vartheta}D_{\mathrm{har}}\bar u$. Together with \eqref{E-bound}, it gives
\begin{equation}\label{new-tail-error}
 \|\chi E_\eta\|_{H^{1-s}}\le C h_*^{-p}\|u\|_{L_s^1(\R^n)}.
\end{equation}

Local $C^1$ regularity of $f$ suffices for the following tests.
Lemma~\ref{field-admissibility} gives $q_\eta\in H^s(\Cyl)\cap C(\Cyl)$,
so its positive and negative parts belong to $H^s$.
Lemma~\ref{cross-testing} justifies the finite interaction and the
physical stability test, including the cutoff limit at the origin.
The potential $V=r^{2s}f'(u)$ is bounded on the support of the outer
cutoff. By \eqref{weighted-defect-L2}, $B_\vartheta$ is in $L^2$ there;
consequently
\[
 \langle\chi^2w_\eta,B_\vartheta\rangle
 =\int_{\Cyl}\chi^2w_\eta B_\vartheta
 \quad\hbox{is finite and nonnegative}.
\]
To pair \eqref{truncated-q-equation} with $\chi^2w_\eta$, approximate
$w_\eta$ by origin cutoffs followed by nonnegative smooth functions
on compact cylinders, and multiply the tests by $\chi^2$.
Convergence in $H^s$ handles $\T_{s,\vartheta}$, and convergence
in $L^2$ handles $Vq_\eta$ and $B_\vartheta$.
For the error, pair $\chi E_\eta\in H^{1-s}\subset L^2$, by
\eqref{new-tail-error}, with the convergent factor $\chi w_\eta$.
These solution-dependent bounds justify the testing limits.

Apply stability to $r^{-\beta}\chi w_\eta$, using
\eqref{energy-identity} and \eqref{critical-stability}, and pair
\eqref{truncated-q-equation} with
$\chi^2w_\eta$. Since $q_\eta =w_\eta-z_\eta$,
the pairing with $\T_{s,\vartheta}$ is exactly
\[
 \Ree\langle\chi^2w_\eta,\T_{s,\vartheta} q_\eta \rangle
 =\langle\chi w_\eta,\Ree\T_{s,\vartheta}\chi w_\eta\rangle
   -\mathfrak C_\chi[w_\eta]+I_{\chi,\eta}.
\]
The stability test bounds the potential pairing above by
$\langle\chi w_\eta,\mathcal A_s\chi w_\eta\rangle$. Hence
\begin{align}
 &\langle\chi w_\eta,(\Ree\T_{s,\vartheta}-\mathcal A_s)
                        \chi w_\eta\rangle
 +\langle\chi^2w_\eta,B_\vartheta\rangle+I_{\chi,\eta}\notag\\
 &\hspace{25mm}\le\mathfrak C_\chi[w_\eta]
                      +\langle\chi^2w_\eta,E_\eta\rangle,\label{retained-identity}\\
 I_{\chi,\eta}
 &:=-\langle\chi^2w_\eta,\T_{s,\vartheta}z_\eta\rangle
   =\mathcal I_\vartheta(\chi^2\eta w,\eta z)\ge0.
 \notag
\end{align}
We retain $I_{\chi,\eta}$ for Lemma~\ref{rectification} and discard
the nonnegative defect term.

Since $\eta=1$ on $\operatorname{supp}\chi$, we have
$\chi w_\eta=\chi w$. Lemma~\ref{cutoff} gives, for $0<\varepsilon\le1$,
\[
 |\mathfrak C_\chi[w_\eta]|
 \le\varepsilon\|\chi w\|_{H^{s-1}}^2
 +C_\varepsilon h_*^{-p_1}
                 \|(1-\psi)\eta w\|_{H^{s-1}}^2.
\]
For the other error, Sobolev duality and \eqref{new-tail-error} give
\begin{align*}
 |\langle\chi^2w_\eta,E_\eta\rangle|
 &=|\langle\chi w,\chi E_\eta\rangle|\\
 &\le\|\chi w\|_{H^{s-1}}\|\chi E_\eta\|_{H^{1-s}}\\
 &\le\delta\|\chi w\|_{H^{s-1}}^2
       +C_\delta h_*^{-2p_2}\|u\|_{L_s^1(\R^n)}^2,
       \qquad 0<\delta\le1.
\end{align*}
The exponents $p_1,p_2$ are fixed independently of
$\varepsilon,\delta$ and of the solution.

The coercivity estimate \eqref{weighted-gap} bounds the first term
in \eqref{retained-identity} below by
$\kappa_\vartheta\|\chi w\|_{H^{s-1}}^2$.
Choose $\varepsilon+\delta\le\kappa_\vartheta/2$ and absorb both
small terms into this bound. Increasing the constant and taking
$p=\max\{p_1,2p_2\}$ gives
\begin{equation}\label{pre-fill}
 \|\chi w\|_{H^{s-1}}^2+I_{\chi,\eta}
 \le C h_*^{-p}\|(1-\psi)\eta w\|_{H^{s-1}}^2
                         +C h_*^{-p} M_s(u)^2.
\end{equation}

Define the monotone function
\begin{equation}\label{monotone-F}
 \Phi(R)=\|\mathbf1_{\{r<R\}}w\|_{H^{s-1}}^2
       +\mathcal I_\vartheta(\mathbf1_{\{r<R\}}w,
                             \mathbf1_{\{r<R\}}z).
\end{equation}
For $0\le a\le b$, the positivity in \eqref{bessel-positive-pairing} gives
\[
 \|b\|_{H^{-\sigma}}^2-\|a\|_{H^{-\sigma}}^2
 =\langle b-a,(1-\Delta_{\Cyl})^{-\sigma}(b+a)\rangle\ge0.
\]
Thus the first term is monotone; the second is monotone by
\eqref{interaction-order}.

To check finiteness, choose a fixed smooth $0\le\zeta\le1$, equal to
one on $r\le3/4$ and supported in $r<1$. Lemma~\ref{field-admissibility}
gives $\zeta q_\vartheta\in H^s\cap C$; apply
Lemma~\ref{cross-testing} to $\zeta q_\vartheta$ with a test cutoff equal to
one on its support. Positivity bounds the interaction in $\Phi(3/4)$
by that finite pairing. The norm term is finite because
$\mathbf1_{\{r<3/4\}}w\in L^2\subset H^{s-1}$.
Only finiteness for the individual solution is used at this point.

Since $\chi=\eta=1$ on $r<R$, the left side of \eqref{pre-fill}
dominates $\Phi(R)$. Apply Lemma~\ref{rectification} to
$v=(1-\psi)\eta q_\vartheta$, which is supported in a fixed compact cylinder
annulus. Its interaction is at most the second term of $\Phi(R')$.
Its signed norm is bounded by \eqref{signed-q-energy}, applied to
$\varphi=(1-\psi)\eta$. Thus its square is at most
$Ch_*^{-p_0}M_s(u)^2$ with the fixed choice $p_0=2$.

For any prescribed $0<a<1$, choose the parameter in
Lemma~\ref{rectification} to be
\[
 \varepsilon=c a h_*^p,
\]
where $c$ is small compared with the constant in \eqref{pre-fill}.
The sign-interaction term then has coefficient at most $a$, while the
signed-norm term costs at most
\[
 C_a h_*^{-p-p\sigma/s-p_0}M_s(u)^2.
\]
Thus
\begin{equation}\label{hole-fill}
 \Phi(R)\le a\Phi(R')+C_a(R'-R)^{-P} M_s(u)^2.
\end{equation}
For the fixed weight $\vartheta$, the exponent $P=P(n,s,\vartheta)$
is finite and independent of $a$: the cutoff constants have polynomial
growth, and \eqref{rectification-bound} costs only a fixed power of
its epsilon.

Choose $a<2^{-P}$ and apply Lemma~\ref{prelim-iteration} to
\eqref{hole-fill} on $[1/2,3/4]$. Monotonicity and $\Phi(3/4)<\infty$
verify its boundedness hypothesis for the fixed solution. Hence
\begin{equation}\label{native-q}
 \Phi(1/2)\le C M_s(u)^2.
\end{equation}
Thus $(q_\vartheta)_+$ is controlled near the origin
by the energy and weighted tail of $u$.

We finally apply the positive averaging operator to recover the source.
Choose $0\le\chi_{\rm src}\le1$ equal to one for
$r\le1/4$ and supported in $r<1/2$. Positivity of the averaging operator gives
\[
 \mathcal S_{s,\vartheta} q_\vartheta 
 \le \mathcal S_{s,\vartheta}(\chi_{\rm src}w)
       +\mathcal S_{s,\vartheta}((1-\chi_{\rm src}) q_\vartheta).
\]
The symbol bound \eqref{uniform-smoothing-symbol} and Plancherel
give $\mathcal S_{s,\vartheta}:H^{s-1}(\Cyl)\to L^2(\Cyl)$.
Multiplication by the fixed smooth cutoff $\chi_{\rm src}$ is bounded
on $H^{s-1}(\Cyl)$, by smooth localization and duality. Since
$\chi_{\rm src}w=\chi_{\rm src}\mathbf1_{\{r<1/2\}}w$, we obtain
\[
 \begin{aligned}
 \|\mathcal S_{s,\vartheta}(\chi_{\rm src}w)\|_{L^2}
 &\le C\|\chi_{\rm src}w\|_{H^{s-1}}
 \le C\|\mathbf1_{\{r<1/2\}}w\|_{H^{s-1}}\\
 &\le C \Phi(1/2)^{1/2}\le C M_s(u),
 \end{aligned}
\]
where the last inequality is \eqref{native-q}.
For $r<1/8$ the second term is exactly
$r^{-\vartheta}\mathcal S_s((1-\chi_{\rm src})q)$ and is bounded in absolute
value by $Cr^{2s-\vartheta}\|u\|_{L_s^1(\R^n)}$, by
\eqref{separated-smoothing-tail}.

The localization error $\mathfrak h$ in
\eqref{tail-error} has the same bound after applying the source-averaging formula. Hence
\[
 0\le r^{-\vartheta}F_s[f(u)]
 \le \mathcal S_{s,\vartheta}(\chi_{\rm src}w)+Cr^{2s-\vartheta}\|u\|_{L_s^1(\R^n)}.
\]
Since $\vartheta<2s$, the last term belongs to $L^2(dt\dd \omega)$ on
$r<1/8$. This proves \eqref{general-partial}.
\end{proof}

\subsection{Source decay and interior regularity}\label{full-section}\label{source-input-subsection}\label{source-reduction-subsection}

For $0<\alpha<\min\{1,2s\}$ satisfying
\begin{equation}\label{pde-input}
 \Ree t_{j,\alpha}(\tau)-a_j(\tau)
 \ge\kappa(1+j^2+\tau^2)^{s-1},\qquad j\ge0,\quad\tau\in\R,
\end{equation}
Proposition~\ref{weighted-partial} applies with $\vartheta=\alpha$.
Thus its estimate reads
\begin{equation}\label{main-bound}
 \int_{B_{1/8}}|x|^{-n-2\alpha}|F_s[f(u)](x)|^2\dd x
 \le C_0 M_s(u)^2,\qquad C_0=C_0(n,s,\alpha,\kappa).
\end{equation}
It is uniform under interior translations and dilations: if
$v(z)=u(a+Rz)$, then $Lv=R^{2s}f(v)$, with a nonnegative convex
reaction, and the change of variables preserves stability.

More explicitly, let $F_{s,a}[g]$ denote
\eqref{source-definition} centered at $a$: replace $B_{|x|}$ by
$B_{|x-a|}(a)$ and $|x|^2-|y|^2$ by
$|x-a|^2-|y-a|^2$. For the rescaled source
$g_{a,R}(z)=R^{2s}g(a+Rz)$, a change of variables gives
\[
 F_s[g_{a,R}](z)=F_{s,a}[g](a+Rz).
\]
Taking $g=f(u)$ makes the source on the left $R^{2s}f(v)$.
Thus the assertion is equivalently
\begin{equation}\label{scaled-source-input}
 R^{2\alpha}\int_{B_{R/8}(a)}
 \frac{|F_{s,a}[f(u)](x)|^2}{|x-a|^{n+2\alpha}}\dd x
 \le C_0\left(R^{2s-n}[u]_{H^s(a+RA_*)}^2+\|u\|_{L_s^1(a,R)}^2\right).
\end{equation}
For a nonnegative source, the weighted bound for its Poisson average
yields decay of its mass on balls.

\begin{lemma}\label{source-average-morrey}
Let $n\ge2$, $0<s<1$, $0<\alpha<\min\{1,2s\}$, and let
$g\ge0$ be locally bounded in $B_2$. If, for some $K\ge0$,
\begin{equation}\label{generic-source-bound}
 \int_{B_{1/8}}|x|^{-n-2\alpha}|F_s[g](x)|^2\dd x\le K^2,
\end{equation}
then
\[
 \int_{B_\rho}g\le C_{n,s,\alpha}K\rho^{n-2s+\alpha}
 \qquad(0<\rho\le1/16).
\]
\end{lemma}
\begin{proof}
The mass-one identity for the ordinary ball Poisson kernel
\eqref{spherical-poisson}, after scaling, gives
\[
 \fint_{\Sph}|r\omega-y|^{-n}\dd\omega
 =\frac{r^{2-n}}{r^2-|y|^2},\qquad |y|<r.
\]
Consequently, averaging \eqref{source-definition} yields
\[
 \fint_{\Sph}F_s[g](r\omega)\dd\omega
 =\frac{2^{1-s}r^{2-n}}{\Gamma(s)|\Sph|}
   \int_{B_r}(r^2-|y|^2)^{s-1}g(y)\dd y
 \ge c_{n,s}r^{2s-n}\int_{B_r}g.
\]
The inequality uses $s-1<0$ and $g\ge0$. Jensen's inequality and
\eqref{generic-source-bound} imply
\begin{equation}\label{power-square}
 \int_0^{1/8}\left(r^{2s-n}\int_{B_r}g\right)^2
       \frac{dr}{r^{1+2\alpha}}\le CK^2.
\end{equation}
Integrating over $[\rho,2\rho]$ and using monotonicity of the
nonnegative ball mass gives
\[
 \left(\int_{B_\rho}g\right)^2\rho^{4s-2n-2\alpha}\le CK^2.
\]
Taking square roots proves the claim.
\end{proof}

We now deduce the interior H\"older and energy estimates from source-mass
decay, combining Riesz potentials with an energy iteration.

The following energy estimate is valid also in dimension one.
\begin{lemma}[Energy bounded by amplitude times tail]\label{energy-amplitude}
If $u\in L_s^1(\R^n)\cap H^s_{\rm loc}(B_2)\cap L^\infty_{\rm loc}(B_2)$
and $Lu=g\ge0$ distributionally in $B_2$, with
$g\in L^\infty_{\rm loc}(B_2)$, then
\begin{equation}\label{energy-amplitude-bound}
 [u]_{H^s(B_{3/2})}^2\le C\|u\|_{L_s^1(\R^n)}\|u\|_{L^\infty(B_{15/8})}.
\end{equation}
\end{lemma}
\begin{proof}
Choose $\eta=1$ on $B_{3/2}$, supported in $B_{7/4}$, with values
in $[0,1]$. The test $\eta^2u$ is admissible under the stated local
energy assumption. Indeed, choose $\zeta\in C_c^\infty(B_2)$ equal
to one on $B_{15/8}$. Then $\zeta u\in H^s(\R^n)$ by Lemma~\ref{prelim-localization}, while
$L((1-\zeta)u)$ is smooth on $B_{7/4}$ and bounded there by
$C\|u\|_{L_s^1}$. Approximate $\eta^2u$ in $H^s$ by smooth tests
supported in $B_{7/4}$. The pairing with $\zeta u$ converges by energy
duality, the remaining pairing by the smooth tail bound, and the
source pairing by local boundedness of $g$ and $L^2$ convergence.

Thus testing with $\eta^2u$ gives the localized energy identity
\eqref{cutoff-energy-identity} in the form
\begin{align*}
 [\eta u]_{H^s(\R^n)}^2
 &=\int_{B_{7/4}}g\,u\eta^2\\
 &\quad+\frac{c_{n,s}}2\iint_{\R^n\times\R^n}
 \frac{u(x)u(y)(\eta(x)-\eta(y))^2}{|x-y|^{n+2s}}\dd x\dd y.
\end{align*}
Its exterior integrals contain one value of $u$ in the cutoff
support and the other only to the first power.

Put $M_\infty=\|u\|_{L^\infty(B_{15/8})}$. Lemma~\ref{source-mass} bounds
the first term by $CM_\infty\|u\|_{L_s^1(\R^n)}$.

In the double integral over
$B_{15/8}\times B_{15/8}$ bound one factor by $M_\infty$ and use
\[
 \sup_y\int\frac{(\eta(x)-\eta(y))^2}{|x-y|^{n+2s}}\dd x\le C.
\]
This follows from the Lipschitz bound at distance less than one
and from the boundedness of $\eta$ at larger distances. The other
factor has local $L^1$ norm at most $C\|u\|_{L_s^1(\R^n)}$.

If one point is
outside $B_{15/8}$, the other must lie in $B_{7/4}$; their separation
is positive, and the kernel is bounded by
$C(1+|y|^{n+2s})^{-1}$ at the exterior point $y$. This part also has
absolute value at most $CM_\infty\|u\|_{L_s^1(\R^n)}$. Finally,
\[
 [u]_{H^s(B_{3/2})}^2\le[\eta u]_{H^s(\R^n)}^2.
\]
\end{proof}

\begin{proposition}\label{source-to-full}
Let $n\ge2$, $0<s<1$, and $0<\alpha<\min\{1,2s\}$. Suppose
\[
 u\in L_s^1(\R^n)\cap H^s_{\rm loc}(B_2)\cap L^\infty_{\rm loc}(B_2),
 \qquad Lu=g\ge0\quad\hbox{distributionally in }B_2,
\]
where $g\in L^\infty_{\rm loc}(B_2)$. For every $R>0$ and $a$ with
$B_{2R}(a)\subseteq B_2$, write $u_{a,R}(z)=u(a+Rz)$ on $\R^n$.
Assume, with the same constant $C_0$ for all these centers and scales,
\begin{equation}\label{morrey-input}
 \int_{B_\rho(a)}g\le C_0 R^{-\alpha}M_s(u_{a,R})
              \rho^{n-2s+\alpha}
 \qquad(0<\rho\le R/16).
\end{equation}
Then $u$ has a representative satisfying
\[
 \|u\|_{C^\alpha(B_1)}+[u]_{H^s(B_1)}
 \le C\|u\|_{L_s^1(\R^n)},\qquad C=C(n,s,\alpha,C_0).
\]
\end{proposition}
\begin{proof}[Proof of Proposition~\ref{source-to-full}]
All constants below may depend on $n,s,\alpha,C_0$. Put
$M=[u]_{H^s(B_{3/2})}+\|u\|_{L_s^1(\R^n)}$. We first derive
\begin{equation}\label{morrey-mass}
 \int_{B_\rho(a)}g\le CM\rho^{n-2s+\alpha}
 \qquad(a\in B_{1/2},\ 0<\rho\le1/4).
\end{equation}
Apply \eqref{morrey-input} with $R=1/4$ and $a\in B_{1/2}$.
The annulus $a+A_*/4$ is contained in $B_{3/2}$, so its rescaled
energy is bounded by $C[u]_{H^s(B_{3/2})}^2$. Also
\[
 \|u(a+\tfrac14\,\cdot)\|_{L_s^1(\R^n)}
 \le C\|u\|_{L_s^1(\R^n)}
\]
uniformly in $a$: on bounded sets the two weights are comparable,
and at infinity they have the same decay. This proves
\eqref{morrey-mass} for $\rho\le1/64$. For the remaining radii, Lemma~\ref{source-mass} gives
$\int_{B_\rho(a)}g\le C\|u\|_{L_s^1(\R^n)}\le CM$; since
$\rho\ge1/64$, this also gives the required power of $\rho$.

Riesz potentials next give
\begin{equation}\label{holder-energy}
 \|u\|_{C^\alpha(B_{1/4})}\le CM.
\end{equation}
Choose $0\le\chi\le1$ supported in $B_1$ and equal to one on
$B_{3/4}$, and put $\mu=\chi g\dd x$. Lemma~\ref{source-mass} and \eqref{morrey-mass} give
\[
 \mu(\R^n)\le CM,\qquad
 \mu(B_\rho(a))\le CM\rho^{n-2s+\alpha}
 \quad(a\in B_{1/2},\ 0<\rho\le1/4).
\]

Apply Lemma~\ref{prelim-morrey-potential} to
$v=I_{2s}\mu$. It gives $Lv=\chi g$,
$\|v\|_{C^\alpha(B_{1/4})}\le CM$, and, using the sharper total-mass
bound from Lemma~\ref{source-mass},
$\|v\|_{L_s^1(\R^n)}\le C\|u\|_{L_s^1(\R^n)}$.

Thus $w=u-v$ is distributionally $s$-harmonic in $B_{3/4}$ and
$\|w\|_{L_s^1(\R^n)}\le C\|u\|_{L_s^1(\R^n)}$.
Lemma~\ref{harmonic-weyl}, after a fixed rescaling, yields
\begin{equation}\label{harmonic-tail}
 \|w\|_{C^1(B_{1/4})}\le C\|w\|_{L_s^1(\R^n)}.
\end{equation}
Combining the bounds for $v$ and $w$ proves \eqref{holder-energy}.

For the iteration below, we verify that the hypothesis has the same
constant after rescaling. Both factors in \eqref{morrey-input} scale explicitly:
\begin{align*}
 M_s(u_{a,R})^2
 &=R^{2s-n}[u]_{H^s(a+RA_*)}^2+\|u\|_{L_s^1(a,R)}^2,\\
 \int_{B_r}g_{a,R}(z)\dd z
 &=R^{2s-n}\int_{B_{Rr}(a)}g(y)\dd y,
 \qquad g_{a,R}(z)=R^{2s}g(a+Rz).
\end{align*}
Thus \eqref{morrey-input} is equivalent to
$\int_{B_r}g_{a,R}\le C_0M_s(u_{a,R})r^{n-2s+\alpha}$ for
$0<r\le1/16$. The factor $R^{-\alpha}$ follows from
$R^{n-2s}(\rho/R)^{n-2s+\alpha}$.
Since
\[
 (u_{a,R})_{b,T}=u_{a+Rb,RT},\qquad
 (g_{a,R})_{b,T}=g_{a+Rb,RT},
\]
the same Morrey bound, with the same $C_0$, holds for every interior
translation and dilation of $u$.

Lemma~\ref{energy-amplitude} bounds the energy in \eqref{holder-energy} by
\[
 [u]_{H^s(B_{3/2})}^2
 \le C\|u\|_{L_s^1(\R^n)}\|u\|_{L^\infty(B_{15/8})}.
\]
Combining the two bounds gives
\begin{equation}\label{one-step}
 \|u\|_{L^\infty(B_{1/4})}
 \le C\left(\|u\|_{L_s^1(\R^n)}^{1/2}\|u\|_{L^\infty(B_{15/8})}^{1/2}+\|u\|_{L_s^1(\R^n)}\right).
\end{equation}

We remove the outer amplitude by an iteration on balls. For
$1\le r<R\le3/2$, let $\rho=(R-r)/4$. For each $a\in B_r$, apply
\eqref{one-step} to $v(z)=u(a+\rho z)$. Its equation is
$Lv=g_{a,\rho}\ge0$. The covariance just established gives the
same hypothesis for $v$, with the same constant. Local boundedness
and local $H^s$ membership are also
preserved. The ball $a+\rho B_2$ lies in $B_R$.
Moreover,
\begin{equation}\label{rescaled-tail}
 \|v\|_{L_s^1(\R^n)}=\rho^{2s}\int_{\R^n}
  \frac{|u(y)|}{\rho^{n+2s}+|y-a|^{n+2s}}\dd y
 \le C\rho^{-n}\|u\|_{L_s^1(\R^n)},
 \quad a\in B_2,\quad 0<\rho\le1.
\end{equation}
For bounded $y$ use the denominator $\rho^{n+2s}$; for large $y$
use $|y-a|\asymp|y|$.

Taking the supremum over $a$ and applying
Young's inequality gives, with $M_\infty(r)=\|u\|_{L^\infty(B_r)}$,
\begin{equation}\label{sup-iteration}
 M_\infty(r)\le\varepsilon M_\infty(R)
       +C\varepsilon^{-1}(R-r)^{-n}\|u\|_{L_s^1(\R^n)},
 \qquad 0<\varepsilon<1.
\end{equation}
Fix $\varepsilon<2^{-n}$ and apply Lemma~\ref{prelim-iteration}
to \eqref{sup-iteration}. Its boundedness hypothesis holds since
$M_\infty(3/2)<\infty$ for the individual locally bounded solution.
Thus $M_\infty(1)\le C\|u\|_{L_s^1(\R^n)}$.

Applying this argument to a finite cover by fixed interior balls
gives $\|u\|_{L^\infty(B_R)}\le C_R\|u\|_{L_s^1(\R^n)}$ for every fixed $R<2$.
The tails of the translated and rescaled solutions are controlled
by \eqref{rescaled-tail}.

In particular, Lemma~\ref{energy-amplitude} now yields
\[
 [u]_{H^s(B_{3/2})}\le C\|u\|_{L_s^1(\R^n)}.
\]
Apply \eqref{holder-energy} on
finitely many fixed interior balls covering $\overline B_1$ to
obtain its H\"older seminorm. The corresponding local energies are
bounded by the same energy-amplitude argument. Pairs of points
outside a common small ball are controlled by the supremum bound.
All constants depend only on $n,s,\alpha,C_0$.
\end{proof}

\begin{theorem}[Interior estimate from a fixed positive weight]
\label{pde-interface}
Fix $n\ge2$, $0<s<1$, $0<\alpha<\min\{1,2s\}$ and $\kappa>0$.
Assume the comparison \eqref{pde-input}.
Let $u\in L_s^1(\R^n)\cap L^\infty_{\rm loc}(B_2)$ be a stable
distributional solution of $Lu=f(u)$ in $B_2$, where $f\ge0$ is
convex on $\R$ and belongs to $C^1_{\mathrm{loc}}(\R)$. Then
\[
 \|u\|_{C^\alpha(B_1)}+[u]_{H^s(B_1)}
 \le C\|u\|_{L_s^1(\R^n)},\qquad C=C(n,s,\alpha,\kappa).
\]
The same estimate holds after every interior translation and
dilation, with the rescaled tail norm \eqref{rescaled-tail-definition}.
\end{theorem}

\begin{proof}[Proof of Theorem~\ref{pde-interface}]
Proposition~\ref{weighted-partial} and \eqref{scaled-source-input} give the source estimate
for every $u_{a,R}$, with source $g_{a,R}=R^{2s}f(u_{a,R})$.
Apply Lemma~\ref{source-average-morrey} with
$K=C_0^{1/2}M_s(u_{a,R})$, then change variables as above. This
is precisely \eqref{morrey-input}, with a constant depending only
on $n,s,\alpha,\kappa$. Lemma~\ref{bootstrap} supplies local $H^s$
membership, and local boundedness is assumed. Proposition~\ref{source-to-full}
therefore applies to $g=f(u)\ge0$. Its constant depends only on
$n,s,\alpha,\kappa$. The same argument after translation and
dilation proves the stated scaled estimate.
\end{proof}

\section{Extension-energy comparison}\label{part:spectral}\label{gap-section}

We compare the two boundary operators using the one-dimensional
extension forms of Lemma~\ref{extension-symbol-bounds}. Put
\begin{equation}\label{delta-definition}
 \beta=\frac{n-2s}{2},\qquad d=\beta-1,\qquad
 \delta=n-1-d^2.
\end{equation}
For $n\ge2$, the inequality $\delta\ge0$ is equivalent to
$n\le2\bigl(s+2+\sqrt{2s+2}\bigr)$. Write
\[
 R_{j,s}(\tau)=\Ree t_{j,0}(\tau)-a_j(\tau).
\]
The unweighted comparison that will be needed below is
\begin{equation}\label{gap-hypothesis}
 R_{j,s}(\tau)\ge\kappa(1+j^2+\tau^2)^{s-1}
 \qquad(j\ge0,\ \tau\in\R),\qquad\kappa>0.
\end{equation}

This section proves the comparison for $\delta\ge0$ and yields the
explicit radical dimension bound. The remaining parameter pairs will
require a finer analysis of the radial mode.

\subsection{Comparison of extension energies}\label{extension-comparison}

Increasing the angular degree adds a positive potential to the
extension energy. The change of radial homogeneity subtracts a
constant. When $\delta>0$, this difference remains positive for
the following explicit weights.

\begin{proposition}[Positive potential comparison]
\label{analytic-gap}\label{explicit-weight-gap}
Let $n\ge2$, $0<s<1$, and let $\vartheta$ satisfy
\eqref{weight-interval}. Put
$\delta_\vartheta=n-1-(d+\vartheta)^2$. There is
$\kappa_{n,s,\vartheta}>0$ such that
\begin{equation}\label{analytic-extension-gap}
 \Ree t_{j,\vartheta}(\tau)-a_j(\tau)
 \ge\kappa_{n,s,\vartheta}(\delta_\vartheta+2j)
                         (1+j^2+\tau^2)^{s-1}
\end{equation}
for all $j\ge0$ with $\delta_\vartheta+2j>0$, and all $\tau\in\R$.
In particular, if $\delta>0$, every
\[
 0\le\alpha<\min\{1,2s,\sqrt{n-1}-d\}
\]
gives $\delta_\alpha>0$ and the bound holds with
$\vartheta=\alpha$ for all $j\ge0$.
\end{proposition}
\begin{proof}
Using the notation of \eqref{dn-form-definition}, set
\[
 Q_j=\lambda_j\sec^2\varphi+\beta^2+\tau^2,\qquad
 W_{j,\vartheta}=(2j+n-1)\sec^2\varphi-(d+\vartheta)^2.
\]
The real extension energy for $a_j(\tau)$ has potential $Q_j$,
whereas that for $t_{j,\vartheta}(\tau)$ has potential
$Q_j+W_{j,\vartheta}-2\ii\tau(d+\vartheta)$. In particular,
\begin{equation}\label{Wj-positive}
 W_{j,\vartheta}\ge\delta_\vartheta+2j>0.
\end{equation}

Let $v_A$ and $v_T$ minimize the real energies with potentials $Q_j$
and $Q_j+W_{j,\vartheta}$, respectively, among profiles with trace one.
Their domains are $\mathcal V_j$ and $\mathcal V_{j+1}$, with the latter
contained in the former. Let $\Phi_T$ solve the second, complex
boundary-value problem. The real part of its boundary energy is at
least the real minimum. Thus
\begin{equation}\label{exact-variational-gap}
 \begin{split}
 &\Ree t_{j,\vartheta}(\tau)-a_j(\tau)\\
 &\quad\ge d_s\bigl(E_T[v_T]-E_A[v_A]\bigr)
 \ge d_s\int_0^{\pi/2}pW_{j,\vartheta}|v_T|^2\dd\varphi.
 \end{split}
\end{equation}
Indeed, $E_A[v_A]\le E_A[v_T]$, and the imaginary part of a
trace-one profile has trace zero and contributes a nonnegative
energy.

By \eqref{dn-boundary-layer}, with $K=\langle(j,\tau)\rangle$,
$v_T\ge1/2$ on $[0,\epsilon/K]$. Consequently
\[
 \|v_T\|_p^2\ge c\int_0^{\epsilon/K}\varphi^{1-2s}\dd\varphi
                 \ge cK^{2s-2}.
\]
Together with \eqref{Wj-positive} and \eqref{exact-variational-gap},
this proves the bound. Finally, $\delta>0$ gives
$-\sqrt{n-1}<d<\sqrt{n-1}$, so the stated interval for $\alpha$
implies $-\sqrt{n-1}<d+\alpha<\sqrt{n-1}$ and hence
$\delta_\alpha>0$.
\end{proof}

At $\delta=0$, the radial mode has no positive constant in
\eqref{Wj-positive}. The imaginary part of the extension equation
supplies an additional positive energy. Keeping this energy gives
the endpoint estimate; a small positive weight will then follow from
continuity.

\begin{corollary}[Coercivity at the endpoint]\label{endpoint}
For $\delta\ge0$ there is $\kappa_{n,s}>0$ such that
\begin{equation}\label{general-gap}
 R_{j,s}(\tau)\ge\kappa_{n,s}(1+j^2+\tau^2)^{s-1}
 \qquad(j\ge0,\ \tau\in\R).
\end{equation}
\end{corollary}
\begin{proof}
Proposition~\ref{analytic-gap} treats $\delta>0$. Suppose $\delta=0$.
Since $d>-1$ and $d^2=n-1$, one has $d=\sqrt{n-1}>0$.
At weight zero,
\[
 W_{j,0}=2j\sec^2\varphi+(n-1)\tan^2\varphi.
\]
For $j\ge1$, this is at least $2j$, so the preceding comparison gives
\eqref{general-gap} for these modes.

It remains to consider $j=0$. Write $K=(1+\tau^2)^{1/2}$ and
\[
 E_0[v]=\int_0^{\pi/2}p\left(|v'|^2+
 [\tau^2+\beta^2+(n-1)\tan^2\varphi]|v|^2\right)\dd\varphi.
\]
Then $E_0[v]\ge c\int p(|v'|^2+K^2|v|^2)$.
Let $v_0$ be its real trace-one minimizer and $\Phi_0$ the
corresponding profile with the additional potential $-\ii b$,
where $b=2d\tau$. In \eqref{dn-positive-correction}, these quantities
are $m=1$, $x=\beta^2+\tau^2-d^2$, $y=-2d\tau$,
$E_{m,x}=E_0$, $v_0=\Phi_{1,x}$, and $\Phi_0=\Phi_{1,x+\ii y}$.
Here $x=\tau^2+2d+1>0$, so the identity applies.
Thus that identity gives
\begin{equation}\label{endpoint-drift-energy}
 \Ree t_{0,0}(\tau)
  =d_s\bigl(E_0[v_0]+E_0[w_0]\bigr),
 \qquad w_0=\Phi_0-v_0.
\end{equation}
By \eqref{dn-correction-equation}, for every zero-trace $\psi$,
\[
 E_0(\psi,w_0)-\ii b\langle\psi,w_0\rangle_p
                         =\ii b\langle\psi,v_0\rangle_p.
\]

We estimate the added energy from below. Choose a nonzero smooth
$\chi\ge0$, supported in $(\epsilon/2,\epsilon)$, where $\epsilon$
is supplied by \eqref{dn-boundary-layer}, and let
$h_K(\varphi)=\chi(K\varphi)$. The boundary-layer estimate gives
\[
 \int_0^{\pi/2}pv_0h_K\dd\varphi\ge cK^{2s-2},\qquad
 E_0[h_K]\le CK^{2s},\qquad \|h_K\|_p\le CK^{s-1}.
\]
Testing the correction equation with $h_K$ and using coercivity yields
\[
 \begin{split}
 |b|\int pv_0h_K
 &\le E_0[w_0]^{1/2}E_0[h_K]^{1/2}
             +|b|\|w_0\|_p\|h_K\|_p\\
 &\le CK^sE_0[w_0]^{1/2},
 \end{split}
\]
because $|b|\le CK$ and $\|w_0\|_p\le CK^{-1}E_0[w_0]^{1/2}$.
It follows that
\begin{equation}\label{endpoint-drift-lower}
 E_0[w_0]\ge cb^2K^{2s-4}.
\end{equation}

Let $E_A$ have potential $\tau^2+\beta^2$, with real trace-one
minimizer $v_A\in\mathcal V_0$. Since
$\mathcal V_1\subset\mathcal V_0$,
\[
 E_0[v_0]-E_A[v_A]
 \ge(n-1)\int_0^{\pi/2}p\tan^2\varphi\,v_0^2\dd\varphi.
\]
On $[\epsilon/(2K),\epsilon/K]$, one has $v_0\ge1/2$ and
$\tan\varphi\ge\varphi$. Hence
\[
 (n-1)\int_0^{\pi/2}p\tan^2\varphi\,v_0^2\dd\varphi
 \ge c\int_{\epsilon/(2K)}^{\epsilon/K}\varphi^{3-2s}\dd\varphi
 \ge cK^{2s-4}.
\]
Together with \eqref{endpoint-drift-energy},
\eqref{endpoint-drift-lower}, and $b^2=4d^2\tau^2$, this gives
\[
 \Ree t_{0,0}(\tau)-a_0(\tau)
 \ge c(1+\tau^2)K^{2s-4}=cK^{2s-2}.
\]
The angular potential controls bounded frequencies and the added
energy controls large frequencies. This proves \eqref{general-gap}.
\end{proof}

\subsection{Positive weights and the radical range}

The unweighted comparison remains coercive for a small positive
weight even when the real potential difference is not positive.
The following argument will also be used when $\delta<0$.

\begin{lemma}[Continuity with respect to the weight]\label{shift-gap}
Assume \eqref{gap-hypothesis}. There is
$\alpha\in(0,\min\{1,2s\})$ such that
\begin{equation}\label{shifted-gap}
 \Ree t_{j,\alpha}(\tau)-a_j(\tau)
 \ge\frac\kappa2(1+j^2+\tau^2)^{s-1}
 \qquad(j\ge0,\ \tau\in\R).
\end{equation}
\end{lemma}
\begin{proof}\label{weight-shift-proof}
We first estimate derivatives of the Dirichlet-to-Neumann function
$F_m$ from Lemma~\ref{extension-symbol-bounds}. Let $m=j+1$,
$K=(1+j^2+\tau^2)^{1/2}$, and
\[
 z=x+\ii y,\qquad x=\beta^2+\tau^2-\gamma^2,
 \qquad |\gamma-d|\le\eta,\quad |y|\le2|\gamma\tau|,
\]
where $\eta>0$ is fixed small enough that these weights lie in
\eqref{weight-interval}. The coercivity and mass bounds in that
lemma hold uniformly on this set, including each vertical segment
from $x$ to $x+\ii y$; in particular,
$\|\Phi_{m,x+\ii y}\|_p^2\le CK^{2s-2}$.

Let $\Phi_z$ be the trace-one profile. Holomorphic dependence in the
form norm follows by inverting the coercive zero-trace operator;
its derivative $\dot\Phi_z$ has trace zero and satisfies
\[
 E_{m,x}(\psi,\dot\Phi_z)+\ii y\langle\psi,\dot\Phi_z\rangle_p
   =-\langle\psi,\Phi_z\rangle_p
 \qquad(\psi\in\mathcal V_m^0).
\]
Testing with $\dot\Phi_z$ and taking real parts gives
\[
 \|\dot\Phi_z\|_p\le CK^{-2}\|\Phi_z\|_p.
\]
In the complex bilinear form of the boundary-energy identity,
the terms containing $\dot\Phi_z$ vanish because it has trace zero.
Differentiation therefore gives
\[
 F_m'(z)=d_s\int_0^{\pi/2}p\Phi_z^2\dd\varphi,
 \qquad
 F_m''(z)=2d_s\int_0^{\pi/2}p\Phi_z\dot\Phi_z\dd\varphi.
\]
These are bilinear identities; no complex conjugates occur in the
integrands. The mass bound for $\Phi_z$ now implies
\begin{equation}\label{dn-parameter-derivatives}
 |F_m'(z)|\le CK^{2s-2},\qquad
 |F_m''(z)|\le CK^{2s-4}.
\end{equation}

For real $x$ the derivative $F_m'(x)$ is real. Integrating the second
bound along a vertical segment gives
\[
 |\Imm F_m'(x+\ii y)|\le C|y|K^{2s-4}.
\]
Set $z_\gamma=\beta^2+\tau^2-\gamma^2-2\ii\gamma\tau$.
Using \eqref{dn-parameter-derivatives}, we obtain
\[
 \begin{split}
 \left|\frac{\dd}{\dd\gamma}\Ree F_m(z_\gamma)\right|
 &=\bigl|-2\gamma\Ree F_m'(z_\gamma)
                       +2\tau\Imm F_m'(z_\gamma)\bigr|\\
 &\le CK^{2s-2}.
 \end{split}
\]
The constants are uniform in $j,\tau$ and in the chosen compact
interval of $\gamma$. Since $t_{j,\alpha}(\tau)=F_{j+1}(z_{d+\alpha})$,
integration from $d$ to $d+\alpha$ yields
\[
 |\Ree t_{j,\alpha}(\tau)-\Ree t_{j,0}(\tau)|
 \le C\alpha(1+j^2+\tau^2)^{s-1}.
\]
Choose $0<\alpha<\min\{\eta,1,2s,\kappa/(2C)\}$ and use
\eqref{gap-hypothesis}.
\end{proof}

Combining these comparisons with Theorem~\ref{pde-interface} gives
the following estimate. The strict dimension condition appeared
in the radial decreasing, zero-exterior estimates of Sanz-Perela
\cite{SP}.

\begin{theorem}[Interior estimates in the radical range]\label{explicit-exponents}
Let $0<s<1$ and let $n\ge1$ satisfy
\begin{equation}\label{explicit-dimension-range}
 n\le2\bigl(s+2+\sqrt{2s+2}\bigr).
\end{equation}
Let $u$ have the solution, stability, and nonlinearity assumptions of
Theorem~\ref{full}, with \eqref{explicit-dimension-range} in place of
\eqref{dimension-range}. Then \eqref{full-bound} holds for some
$\alpha\in(0,\min\{1,2s\})$ and $C>0$ depending only on $n,s$.
If $n\ge2$ and the inequality in \eqref{explicit-dimension-range} is
strict, the estimate holds for every
\begin{equation}\label{explicit-alpha}
 0<\alpha<\min\left\{1,\,2s,\,
             s+1-\frac n2+\sqrt{n-1}\right\},
\end{equation}
with $C$ depending only on $n,s,\alpha$. In dimension one it holds
for every $0<\alpha<\min\{1,2s\}$.
The constants are independent of $f$ and the initial local bounds of $u$.
\end{theorem}
\begin{proof}
For $n\ge2$ and strict inequality in \eqref{explicit-dimension-range},
one has $\delta>0$. Given \eqref{explicit-alpha},
Proposition~\ref{explicit-weight-gap} supplies \eqref{pde-input}, and
Theorem~\ref{pde-interface} gives \eqref{full-bound} with that exponent.
At equality, $\delta=0$; Corollary~\ref{endpoint} and
Lemma~\ref{shift-gap} give \eqref{pde-input} for a small positive
$\alpha$. Applying
Theorem~\ref{pde-interface} gives the conclusion in this case as well.

For $n=1$, apply Lemma~\ref{prelim-lifting} to
$U(x_1,x_2)=u(x_1)$. The two-dimensional case just proved applies
to $U$ on interior balls of the strip $|x_1|<2$.
In dimension two, the last entry in the minimum in
\eqref{explicit-alpha} is $1+s$, so every
$0<\alpha<\min\{1,2s\}$ is available. Restriction to $x_2=0$
gives the H\"older estimate for $u$. Lemma~\ref{bootstrap} gives
local $H^s$ regularity, and Lemma~\ref{energy-amplitude}, after
fixed interior rescalings, gives the energy bound.
\end{proof}

\section{Optimal dimension and order conditions}\label{spectral-contract-section}\label{sp-appendix}
It remains to prove the unweighted comparison \eqref{gap-hypothesis}
under the optimal condition $A_{n,s}>H_{n,s}$. Lemma~\ref{shift-gap}
will then supply the positive weight required by
Theorem~\ref{pde-interface}. Throughout this section $n\ge2$ is an
integer and $0<s<1$. Write
\[
 N=\frac n2,\qquad \beta=N-s,\qquad d=\beta-1,\qquad
 \delta=n-1-d^2.
\]
Corollary~\ref{endpoint} already treats $\delta\ge0$. For the remaining
pairs with $A_{n,s}>H_{n,s}$ and $\delta<0$, the parameter bounds below
give $\delta>-1/2$. Then $\delta+2j>0$ for $j\ge1$, so the
positive-potential comparison treats the higher angular modes.
The remaining target is therefore the scalar bound
\[
 R_{0,s}(\tau)=\Ree t_{0,0}(\tau)-a_0(\tau)
 \ge\kappa_0(n,s)(1+\tau^2)^{s-1}.
\]
We first identify the exact coefficients and the possible dimensions
and orders. At low frequencies the bound follows directly from the
Gamma products. At larger frequencies we compare the loss in the real
extension energy with the gain from the imaginary drift, using the
positive spectral series of those same extension problems.

\subsection{Homogeneous coefficients and exact formulas}\label{shared-calculus}\label{prelim-special-functions}\label{operator-realization-details}\label{gamma-tools}
The homogeneous formula also identifies the logarithmic singular solution.

We use the Gamma recurrence $\Gamma(z+1)=z\Gamma(z)$ and Euler's
beta integral
\begin{equation}\label{prelim-beta-integral}
 \frac{\Gamma(a)\Gamma(b)}{\Gamma(a+b)}
 =\int_0^1t^{a-1}(1-t)^{b-1}\dd t,
 \qquad \Ree a,\Ree b>0;
\end{equation}
see \cite[5.5.1 and 5.12.1]{NIST}. We write
$\psi=(\log\Gamma)'$ and $\psi^{(k)}$ for its derivatives.

The fractional Laplacian acts diagonally on homogeneous angular modes.
In our Fourier normalization its coefficient is
\begin{equation}\label{power-multiplier}
 \begin{split}
 L(r^{-a}Y_{j,\ell})&=c_j(a)r^{-a-2s}Y_{j,\ell},\\
 c_j(a)&=2^{2s}
 \frac{\Gamma((j+n-a)/2)\Gamma((j+a+2s)/2)}
 {\Gamma((j+a)/2)\Gamma((j+n-a-2s)/2)}.
 \end{split}
\end{equation}
Initially $0<\Ree a<n-2s$. The relation between a harmonic
factor of degree $j$ and the effective radial dimension $n+2j$
is also given in \cite[Section~2.1, Proposition~3(ii), (44)]{DKK}.

\begin{lemma}[Homogeneous Fourier coefficients]\label{homogeneous-fourier}
Use the unitary Euclidean transform \eqref{prelim-euclidean-fourier}.
For $Y_j\in\mathcal H_j$, write
$\Phi_{j,a}(x)=|x|^{-a}Y_j(x/|x|)$ initially when
$\Ree a<n$. Then, for $0<\Ree a<n$,
\begin{equation}\label{homogeneous-transform}
 \widehat{\Phi_{j,a}}=\gamma_j(a)\Phi_{j,n-a},\qquad
 \gamma_j(a)=(-\ii)^j2^{n/2-a}
          \frac{\Gamma((n+j-a)/2)}{\Gamma((j+a)/2)}.
\end{equation}
The fractional-Laplacian formula \eqref{power-multiplier} holds for
$0<\Ree a<n-2s$. In degree zero it also holds at $a=-\ii\tau$,
$\tau\in\R$, for the ordinary distributional fractional Laplacian,
and it may be differentiated with respect to $a$ at $a=0$.
These degree-zero assertions for the radial powers $|x|^{-a}$ also
hold when $n=1$ and $n>2s$.
\end{lemma}
\begin{proof}
Let $p_j(x)=|x|^jY_j(x/|x|)$, a harmonic homogeneous polynomial.
For $t>0$, the Gaussian transform and differentiation give
\begin{equation}\label{harmonic-gaussian-transform}
 \widehat{p_j e^{-t|x|^2}}(\xi)
 =(-\ii)^j(2t)^{-n/2-j}p_j(\xi)e^{-|\xi|^2/(4t)}.
\end{equation}
For completeness, the differentiation identity used here is
$p_j(\partial_\xi)e^{-b|\xi|^2}
=(-2b)^jp_j(\xi)e^{-b|\xi|^2}$.
To check it, differentiate the generating formula
$e^{-b|\xi+z|^2}
=e^{-b|\xi|^2}e^{-2b\xi\cdot z}e^{-b|z|^2}$ by
$p_j(\partial_z)$ at $z=0$. Every term involving $|z|^2$
contains a contraction of two polynomial indices, hence a factor
$\Delta p_j$, and vanishes. The remaining term is
$p_j(-2b\xi)$.

The elementary Gamma integral gives, for
$\Ree(a+j)>0$ and $x\ne0$,
\[
 \Phi_{j,a}(x)=\frac{p_j(x)}{\Gamma((a+j)/2)}
          \int_0^\infty t^{(a+j)/2-1}e^{-t|x|^2}\dd t.
\]
Insert \eqref{harmonic-gaussian-transform} and then set
$z=|\xi|^2/(4t)$. Since
\[
 \int_0^\infty t^{(a-n-j)/2-1}e^{-|\xi|^2/(4t)}\dd t
 =\left(\frac{|\xi|^2}{4}\right)^{(a-n-j)/2}
             \Gamma\left(\frac{n+j-a}{2}\right),
\]
this yields \eqref{homogeneous-transform}, including its factor
$2^{n/2-a}$ and phase $(-\ii)^j$.

For $0<\Ree a<n$ all interchanges are justified after
pairing with a Schwartz function: the absolute integrals reduce
on the physical side to $|x|^{-\Ree a}$ times a Schwartz
function, and on the Fourier side to
$|\xi|^{\Ree a-n}$ times a Schwartz function.
Both are integrable at their respective origins.

When $0<\Ree a<n-2s$, multiplication of the Fourier
side by $|\xi|^{2s}$ and comparison with parameter $a+2s$ therefore
give
\[
 (-\Delta)^s\Phi_{j,a}
   =\frac{\gamma_j(a)}{\gamma_j(a+2s)}\Phi_{j,a+2s}
   =c_j(a)\Phi_{j,a+2s}.
\]

For degree zero at $a=-\ii\tau$, take $a=\epsilon-\ii\tau$
with $\epsilon\downarrow0$. The powers converge locally and in
$L_s^1(\R^n)$, and $c_0(a)\Phi_{0,a+2s}$ converges in distributions
since $2s<n$. Pairing with $(-\Delta)^s\phi$, for
$\phi\in C_c^\infty(\R^n)$, therefore proves the boundary
formula for the ordinary distributional fractional Laplacian.
At $a=0$ this gives $c_0(0)=0$ and $(-\Delta)^s1=0$.

Choose $0<\delta<\min\{2s,n-2s\}$. For $|\Ree a|<\delta$,
the functions $\Phi_{0,a}$ and $\partial_a\Phi_{0,a}$ are bounded
in absolute value by a constant times
$(|x|^{-\delta}+|x|^\delta)(1+|\log|x||)$.
This is integrable against $(-\Delta)^s\phi$: near zero this follows
from $\delta<n$, and at infinity from
$(-\Delta)^s\phi(x)=O(|x|^{-n-2s})$ and $\delta<2s$.
The right-hand side of the power formula and its parameter derivative
are locally integrable because $2s+\delta<n$. The dominated integrals define holomorphic distribution families
on this degree-zero neighborhood. Their equality for $\Ree a>0$
therefore extends across zero and may be differentiated there.
For $n=1$ and $j=0$,
$p_0$ is constant, so the Gaussian calculation and these domination
arguments apply without change to $|x|^{-a}$.
\end{proof}

\begin{lemma}[Exact cylinder coefficients]\label{uniform-gamma-symbols}
The symbols defined by the extension problem in \eqref{shared-symbols}
satisfy
\begin{equation}\label{gamma-cylinder-symbols}
 a_j(\tau)=c_j(\beta-\ii\tau),\qquad
 t_{j,\vartheta}(\tau)=c_{j+1}(1-\vartheta-\ii\tau).
\end{equation}
For $g_s(w)=\Gamma(w+s)/\Gamma(w)$ and $\zeta_m=(m+\beta)/2$,
the even meromorphic continuation of $a_m$ is
\begin{equation}\label{gamma-even-symbol}
 a_m(z)=2^{2s}g_s(\zeta_m+\ii z/2)g_s(\zeta_m-\ii z/2).
\end{equation}
Its poles are at $z=\pm\ii(m+\beta+2s+2k)$, $k=0,1,\ldots$.
Also $g_s(k+c)\asymp(1+k)^s$ for each fixed $c>0$.
The averaging symbol in \eqref{averaging-symbol-definition} is
\begin{equation}\label{gamma-averaging-symbol}
 m_s(z)=2^{s-1}\frac{\Gamma(z/2+s)}{\Gamma(z/2+1)}.
\end{equation}
\end{lemma}
\begin{proof}
Take a smooth compact annular trace $v$ with finitely many angular
components and invert its Fourier transform in $t$.
In \eqref{cylinder-laplacian-realization}, each resulting homogeneous
mode has exponent $a=\beta-\ii\tau$, with $0<\beta<n-2s$.
Lemma~\ref{homogeneous-fourier} therefore identifies its coefficient
with $c_j(\beta-\ii\tau)$, proving the first identity in
\eqref{gamma-cylinder-symbols}.

For the cone identity \eqref{cone}, use its degree-$j$ realization
as the fractional Laplacian in dimension $d=n+2$. The exponent of
each homogeneous mode of $r^{\vartheta-2}v$ is
$a=2-\vartheta-\ii\tau$. Admissibility gives
$0<2-\vartheta<d-2s$, so the same homogeneous formula applies in
its initial strip. Denote that dimension-$d$ coefficient by
$c_j^{(d)}$. Substituting $d=n+2$ and this exponent in its four
Gamma arguments gives
\[
 c_j^{(n+2)}(2-\vartheta-\ii\tau)
      =c_{j+1}^{(n)}(1-\vartheta-\ii\tau).
\]
This proves the second identity in \eqref{gamma-cylinder-symbols}.

These inverse Fourier calculations hold in distributions. To justify
the interchange, pair first with a smooth compact physical test.
For each of the two fixed real exponents just used, the homogeneous
powers are uniformly integrable against that test and its fractional
Laplacian, at zero and infinity. Their Gamma coefficients have at
most polynomial growth in $\tau$ by Stirling's ratio formula
\cite[5.11.12--5.11.13]{NIST}, whereas the Fourier coefficients of
the compact annular trace decrease faster than every power.
Dominated convergence justifies integration of the homogeneous
identity. The bounded cylinder realizations and angular density then
identify the symbols almost everywhere; their continuity gives equality
at every real frequency.

Substitution into \eqref{power-multiplier} gives
\eqref{gamma-even-symbol}. The poles of $g_s$ are
$-s,-s-1,\ldots$, and its zeros are $0,-1,\ldots$; the other factor
has positive real argument at each listed pole, so none cancels.
Stirling's ratio formula \cite[5.11.12--5.11.13]{NIST} gives the
stated real-axis growth. Finally, Euler's beta integral with
parameters $z/2+s$ and $1-s$, followed by $a=e^{-2h}$, evaluates
\eqref{averaging-symbol-definition} and proves
\eqref{gamma-averaging-symbol}.
\end{proof}

\subsection{Parameters and low-frequency comparison}\label{sp-sec:parameters}\label{sp-sec:low}
We now locate the pairs for which the scalar comparison is still
needed. If $A_{n,s}>H_{n,s}$ and $\delta<0$, then $(n,s)$ lies in one
of the following rows:
\begin{equation}\label{sp-eq:boxes}
 \begin{array}{c|c|c}
 n&s\text{ interval}&C\\ \hline
 7&0\le s\le51/1000&11\\
 8&7/25\le s\le71/200&14\\
 9&63/100\le s\le84/125&37/2.
 \end{array}
\end{equation}
Only $s>0$ is used; the endpoint $s=0$ in the first row is used for
limiting bounds. On these rows $\beta>1$ and $\delta>-1/2$.
The parameter coverage is proved below. The constant $C$ assigned to
each row will be used later to quantify the drift contribution.

For the low-frequency comparison use
$\xi=\tau/2$ and $X=\xi^2=\tau^2/4$. Thus the interval
$X\le5/16$ corresponds to $\tau^2\le5/4$.
With $g_s$ from Lemma~\ref{uniform-gamma-symbols}, set
\begin{equation}\label{sp-eq:PQ}
 P(s,\xi)=\Ree[g_s(\beta-\ii\xi)g_s(1+\ii\xi)],\qquad
 Q(s,\xi)=|g_s(\beta/2+\ii\xi)|^2.
\end{equation}
Thus
\begin{equation}\label{sp-eq:normalization}
 R_{0,s}(2\xi)=4^s[P(s,\xi)-Q(s,\xi)],\quad
 A_{n,s}=4^sP(s,0),\quad H_{n,s}=4^sQ(s,0).
\end{equation}
When a lemma varies $\beta$ independently, these same formulas define
$P$ and $Q$.

\begin{lemma}[Dimensions and orders with $A_{n,s}>H_{n,s}$]
\label{sp-lem:integer-hardy-range}
Let $n\geq2$ be an integer and $0<s<1$.
Then $A_{n,s}>H_{n,s}$ holds for every $0<s<1$ when $2\leq n\leq7$,
and it never holds when $n\geq10$.

For $n=8,9$ the equation $A_{n,s}=H_{n,s}$ has a unique root in
$(0,1)$, and $A_{n,s}>H_{n,s}$ holds exactly above that root.
Moreover,
\[
 A_{8,s}>H_{8,s}\ \Longrightarrow\ s>\frac7{25},\qquad
 A_{9,s}>H_{9,s}\ \Longrightarrow\ s>\frac{63}{100}.
\]
Consequently the pairs $(n,s)$ with $A_{n,s}>H_{n,s}$ and $\delta<0$
are contained in the rows of \eqref{sp-eq:boxes}.
\end{lemma}

\begin{proof}
Set
\begin{align*}
 L_n(s)&=\log\frac{A_{n,s}}{H_{n,s}}\\
 &=\log\Gamma(n/2)+\log\Gamma(1+s)-\log\Gamma(n/2-s)
   +2\log\Gamma(n/4-s/2)-2\log\Gamma(n/4+s/2),\\
 \delta_n(s)&=n-1-\bigl((n-2)/2-s\bigr)^2.
\end{align*}
Recall $\psi=(\log\Gamma)'$. We use primes for its first two
derivatives and $\psi^{(k)}$ for its higher derivatives.
The series \cite[5.15.1]{NIST}
\[
 \psi'(x)=\sum_{k=0}^{\infty}\frac1{(x+k)^2},\qquad x>0,
\]
shows that $\psi'$ is strictly decreasing. Allowing $n$ to be real
temporarily and putting $a=(n-2s)/4>0$, differentiation gives
\[
 \partial_n L_n(s)
 =\frac12\int_0^s
    \bigl[\psi'(2a+\eta)-\psi'(a+\eta)\bigr]\dd \eta<0.
\]
For $n\geq6$ one also has
\begin{align*}
 L_n''(s)
 &=\psi'(1+s)-\psi'(n/2-s)\\
 &\quad+\frac12\bigl[\psi'(n/4-s/2)-\psi'(n/4+s/2)\bigr]>0,
\end{align*}
since $1+s\leq n/2-s$. Thus $L_n$ is strictly convex in $s$.

Its endpoint values are
\[
 L_n(0)=0,\qquad
 L_n(1)=\log\frac8{n-2}\quad(n>2).
\]
In particular $L_{10}(0)=L_{10}(1)=0$, so $L_{10}(s)<0$ for
$0<s<1$. Monotonicity in $n$ excludes all $n\geq10$.

The digamma values at integers, half-integers, and quarter-integers
\cite[5.4.12--5.4.15 and 5.4.19]{NIST}, together with the recurrence
$\psi(x+1)=\psi(x)+1/x$, give
\[
 L_7'(0)=4\log2-\pi+\frac25>0,\qquad
 L_8'(0)=-\frac16<0,\qquad
 L_9'(0)=\pi+4\log2-\frac{656}{105}<0.
\]
For the first and last signs, the expansion
$\log2=2\sum_{k\geq0}3^{-(2k+1)}/(2k+1)$ gives
$56/81<\log2<25/36$; together with $\pi<22/7$ this yields
\[
 L_7'(0)>\frac{64}{2835}>0,\qquad
 L_9'(0)<-\frac{103}{315}<0.
\]
Convexity now shows $L_7(s)>0$ for $s>0$, and monotonicity in $n$
gives the assertion for all $2\leq n\leq7$.

For $n=8,9$, the negative derivative at zero and the positive value at one
give a root in $(0,1)$. Strict convexity and $L_n(0)=0$ imply that
$s\mapsto L_n(s)/s$ is strictly increasing, so that root is unique and
has the asserted sign characterization.

We next prove the two rational lower bounds on the roots. For $n=8$,
the recurrence for $\Gamma$ and the convergent expansion of
$\log\Gamma(1+z)$ for $|z|<1$ \cite[5.7.3]{NIST} give
\[
 L_8(s)=\log r(s)
 -\sum_{\substack{k\geq3\\ k\ \mathrm{odd}}}
   \frac{(2-2^{2-k})\zeta(k)}{k}s^k,
 \qquad
 r(s)=\frac{3(1-s/2)}{(1-s)(3-s)(1+s/2)^2}.
\]
Every summand following the minus sign is positive. Since
\[
 \zeta(3)>\sum_{k=1}^{6}k^{-3}
 =\frac{119}{100}+\frac7{24000},\qquad
 \zeta(5)>1+2^{-5}=\frac{33}{32},
\]
retaining the cubic and quintic terms and using $\log r\leq r-1$ yields
\begin{align*}
 L_8(7/25)
 &<\frac{9079}{662796}
    -\frac{119}{200}\left(\frac7{25}\right)^3
    -\frac{99}{256}\left(\frac7{25}\right)^5\\
 &=-\frac{11999272607}{414247500000000}<0.
\end{align*}
Hence the unique root for $n=8$ is greater than $7/25$.

For $n=9$, direct use of the same Gamma recurrences gives
\[
 L_9(1/2)=\log\frac{35}{36},\qquad
 L_9'(1/2)=\frac16,\qquad
 L_9''(1/2)=\frac{\pi^2}{6}-\frac{11}{12}<\frac34.
\]
Furthermore $L_9'''(s)<0$ for $0<s\leq5/6$.
Indeed, $g=-\psi''$ is positive and strictly decreasing, and with
$a=9/4-s/2$, $b=9/4+s/2$ one has
\[
 L_9'''(s)=-g(1+s)-g(9/2-s)+\frac{g(a)+g(b)}4<0,
\]
because $1+s\leq a<b$, so the positive part is at most $g(1+s)/2$.

Taylor's formula at $s=1/2$, with increment $13/100$, therefore gives
\begin{align*}
 L_9(63/100)
 &<\log\frac{35}{36}+\frac{13}{600}+\frac{507}{80000}\\
 &\leq-\frac1{36}-\frac1{2\cdot36^2}
       +\frac{13}{600}+\frac{507}{80000}
 =-\frac{1033}{6480000}<0.
\end{align*}
This proves that the unique root for $n=9$ is greater than $63/100$.

Finally $\delta_n(s)>0$ throughout $0<s<1$ for $2\leq n\leq6$.
For $n=7,8,9$, $\delta_n$ is increasing on $(0,1)$ and vanishes at
\[
 s_{\mathrm{rad}}(n)=\frac{n-2}{2}-\sqrt{n-1}.
\]
The upper endpoints of \eqref{sp-eq:boxes} are rational numbers just above
these three values. Indeed,
\[
 \delta_7(51/1000)=\frac{2399}{10^6}>0,\quad
 \delta_8(71/200)=\frac{159}{40000}>0,\quad
 \delta_9(84/125)=\frac{151}{62500}>0
\]
so each box reaches the region already covered by $\delta_n(s)\geq0$.
The lower endpoints for $n=8,9$ lie below the respective roots of
$A_{n,s}=H_{n,s}$, as just proved; for $n=7$ the lower endpoint is zero.
Thus every pair with $A_{n,s}>H_{n,s}$ and $\delta_n(s)<0$ lies in one
of the three boxes.
\end{proof}

To retain the lower bound $A_{n,s}-H_{n,s}$ at low frequencies,
we compare the two symbols directly. We keep $X=\xi^2=\tau^2/4$ throughout this comparison.

\begin{lemma}[Finite-product comparison]\label{sp-lem:rational-direct}
Let $0<s<1$, $\beta>1$, and put
\[
 p(X)=P(s,\sqrt X),\qquad q(X)=Q(s,\sqrt X),\qquad
 p_0=p(0),\quad q_0=q(0),\quad \lambda=(1+s)^2.
\]
Then $p_0,q_0>0$, both functions are differentiable at zero, and
\begin{equation}\label{sp-eq:direct-rational}
 p(X)\ge p_0+\frac{\lambda X}{\lambda+X}p'(0),\qquad
 q(X)\le q_0+Xq'(0),\qquad X\ge0.
\end{equation}
Moreover $p(X)\ge p_0>0$ for every $X\ge0$.
\end{lemma}
\begin{proof}
Set $d=\beta-1$ and
\[
 c_k=(k+\beta)(k+1),\quad \rho_k=(k+\beta+s)(k+1+s),\quad
 G_K(w)=(K+1)^{2s}\prod_{k=0}^K\frac{w+c_k}{w+\rho_k}.
\]
Since $c_k<\rho_k<c_{k+1}$, the residues of $G_K$ are negative:
in the residue at $-\rho_k$, the factor $c_k-\rho_k$ is negative and
all ratios $(c_j-\rho_k)/(\rho_j-\rho_k)$ for $j\ne k$ are positive.
Consequently
\[
 G_K(w)=G_K(0)+\sum_{k=0}^K m_{k,K}\frac{w}{w+\rho_k},
 \qquad m_{k,K}>0.
\]
For $r>0$, set $u_r=(r+d^2/4)^{1/2}$,
$\alpha_r=(u_r-d/2)^2$, and $\zeta_r=(u_r+d/2)^2$. Direct calculation gives
\[
 \Ree\frac{X+\ii d\sqrt X}{X+r+\ii d\sqrt X}
 =\left(\frac12+\frac d{4u_r}\right)\frac{X}{X+\alpha_r}
 +\left(\frac12-\frac d{4u_r}\right)\frac{X}{X+\zeta_r}.
\]
Both coefficients are positive. For $r=\rho_k$ the pole parameters are
$(k+1+s)^2$ and $(k+\beta+s)^2$, both at least $\lambda$.
Thus $p_K(X)=\Ree G_K(X+\ii d\sqrt X)$ is a positive constant plus
positive multiples of $X/(X+\mu)$ with $\mu\ge\lambda$. Applying
\[
 \frac{X}{X+\mu}\ge\frac{\lambda X}{\lambda+X}\frac1\mu
\]
to each term proves the first inequality in
\eqref{sp-eq:direct-rational} for $p_K$.

The Euler product
\cite[5.8.1]{NIST} gives locally uniform convergence
$G_K(\xi ^2+\ii d\xi )\to g_s(\beta-\ii\xi )g_s(1+\ii\xi )$ near the real axis.
The even analytic functions
$\tfrac12[G_K(\xi ^2+\ii d\xi )+G_K(\xi ^2-\ii d\xi )]$ converge locally uniformly near
$\xi =0$, so their constant and quadratic coefficients converge as well.
This justifies passage to $p$ and $p'(0)$.

For $a=\beta/2$, the finite products
\[
 q_K(X)=(K+1)^{2s}\prod_{k=0}^K
            \frac{X+(a+k)^2}{X+(a+s+k)^2}
\]
have interlacing zeros and poles. The same residue argument expresses
each as a positive constant plus positive multiples of $X/(X+\mu)$.
They are concave, and hence so is their limit $q$. Its tangent at zero
proves the second inequality in \eqref{sp-eq:direct-rational}.
\end{proof}

\begin{proposition}[Comparison at low frequencies]\label{sp-prop:low-direct}
On the boxes \eqref{sp-eq:boxes}, for $s>0$ and $A_{n,s}\ge H_{n,s}$,
\begin{equation}\label{sp-eq:low-direct}
 R_{0,s}(\tau)\ge A_{n,s}-H_{n,s},
 \qquad \tau^2\le5/4.
\end{equation}
\end{proposition}
\begin{proof}
Define
\[
 u_s(y)=\frac1{y(y+s)},\qquad
 v_s(y)=\frac{2y+s}{y^2(y+s)^2}
       =\int_0^1\frac{2\dd v}{(y+sv)^3}.
\]
The absolutely convergent sums
\begin{align*}
 U(s)&=\sum_{k\ge0}\{u_s(k+1)-u_s(k+N-s)\},\\
 V(s)&=\sum_{k\ge0}\{v_s(k+1)+v_s(k+N-s)\},\\
 W(s)&=\sum_{k\ge0}v_s(k+N/2-s/2)
\end{align*}
give, by differentiating the Gamma product or the polygamma series
\cite[5.15.1]{NIST},
\begin{equation}\label{sp-eq:direct-slopes}
 p'(0)=sp_0\eta(s),\qquad q'(0)=sq_0W(s),\qquad
 \eta(s)=\tfrac12\{V(s)-sU(s)^2\}.
\end{equation}
Indeed, for $L=\log g_s$, one has
$L'(z)=s\sum_{k\ge0}u_s(z+k)$ and
$L''(z)=-s\sum_{k\ge0}v_s(z+k)$. The first two logarithmic
$\xi $ derivatives of $g_s(\beta-\ii\xi )g_s(1+\ii\xi )$ at zero are
$\ii sU(s)$ and $sV(s)$, respectively, which proves the formula for $p'(0)$.

With $p_0=p(0)$, $q_0=q(0)$ and $\lambda=(1+s)^2$, the finite-product
bounds of Lemma~\ref{sp-lem:rational-direct} give
\[
 p(X)-q(X)\ge p_0-q_0
 +sp_0X\left(\frac{\lambda}{\lambda+X}\eta(s)-W(s)\right).
\]
Since $p_0\ge q_0$, it is therefore enough to show
\[
 \frac{\lambda}{\lambda+X}\eta(s)>W(s),\qquad 0\le X\le5/16.
\]
We will bound $U,W$ from above and $V,\lambda$ from below, recalling
$\eta=(V-sU^2)/2$. Fix a row with $s$ interval $[\ell,h]$, and put
$\beta_{\max}=N-\ell$ and $a_-=(N-h)/2$. On all three boxes $1<\beta\le4$,
so a telescoping comparison gives
\[
 0<U(s)\le\sum_{k\ge0}\{u_s(k+1)-u_s(k+4)\}
          =\sum_{j=1}^3u_s(j)\le\overline U_n:=\sum_{j=1}^3u_\ell(j).
\]

The integral formula for $v_s$ shows that $v_s(k+1)$ decreases in $s$
and $v_s(k+N-s)$ increases in $s$. Also each summand in $W$ equals
\[
 \frac{2(k+N/2)}{((k+N/2)^2-s^2/4)^2},
\]
which increases in $s$. Since $v_s$ is positive, decreasing, and convex,
comparison with an integral gives a lower tail bound, while the
midpoint inequality gives an upper tail bound. Using
$\int_c^\infty v_s(y)\dd y=u_s(c)$, we obtain
\begin{align}
 V(s)&\ge\underline V_n:=v_h(1)+v_h(2)+u_h(3)+u_\ell(\beta_{\max}),
       \label{sp-eq:short-V}\\
 W(s)&\le\overline W_n:=v_h(a_-)+u_h(a_-+1/2).
       \label{sp-eq:short-W}
\end{align}
For the first inequality retain two terms of
$\sum_{k\ge0}v_h(k+1)$ and bound its remaining tail by
$\int_3^\infty v_h$; bound the whole other sum below by
$\int_{\beta_{\max}}^\infty v_\ell$. For the second, retain its first term and use
$\sum_{k\ge1}v_h(a_-+k)\le\int_{1/2}^\infty v_h(a_-+y)\dd y$.

All nine bounds above are finite rational expressions. Substitution
of the endpoints in \eqref{sp-eq:boxes} gives
\begin{equation}\label{sp-eq:short-rational-table}
\begin{array}{c|ccc}
 n&\overline U_n<u_n&\underline V_n>v_n&\overline W_n<w_n\\\hline
 7&3/2&2&3/5\\
 8&9/8&3/2&5/12\\
 9&9/10&63/50&8/25
\end{array}
\end{equation}
Multiplication by positive denominators verifies these bounds.
For example, in the closest row $n=9$,
the summands of $10^4\overline U_n$ are bounded above by
$6135,1902,919$; those of $10^4\underline V_n$ are bounded below by
$9557,1635,907,574$; and those of $10^4\overline W_n$ are bounded
above by $1837,1343$. Their sums already imply the stated bounds.

Set $\eta_n=(v_n-hu_n^2)/2$ and $\lambda_n=(1+\ell)^2$.
The final rational calculation is
\[
\begin{array}{c|cc}
 n&\eta_n&\lambda_n\eta_n/(\lambda_n+5/16)-w_n\\\hline
 7&7541/8000&1241/10500\\
 8&13449/25600&47861/1950900\\
 9&4473/12500&3197/17675000
\end{array}
\]
Every entry in the last column is positive. Hence, for $0\le X\le5/16$,
\[
 \frac{\lambda}{\lambda+X}\eta(s)-W(s)>0.
\]

The sufficient slope inequality is now established, so the finite-product
bound above gives $p(X)-q(X)\ge p_0-q_0$.
Since $X=\tau^2/4$ and
$R_{0,s}(\tau)=4^s[p(X)-q(X)]$, this is the claimed estimate.
\end{proof}

\subsection{Extension forms and positive spectral series}\label{ap:sec:real-forms}\label{ap:sec:positive-series}
The low-frequency bound is now proved on $\tau^2\le5/4$.
For the complementary interval we return to the extension problems.
Throughout this subsection fix $n\in\{7,8,9\}$ and $0<s<1$.
The spectral identities hold throughout these ranges; their remaining
application is the case $-1/2<\delta<0$.
We now use the squared cylinder frequency and the imaginary parameter
\[
 x=\tau^2=4X,\qquad b=2d\tau,
\]
and put $B=N+s$. The parameter $b$ is distinct from the homogeneity
shift in the cylinder equation.
On $J=(0,\pi/2)$, recall the weight
$p(\varphi)=\sin^{1-2s}\varphi\cos^{n-1}\varphi$ and set
$\mathcal H=L^2(J,p\dd\varphi)$, with inner product linear in the
second slot. The two real extension forms are
\[
 \begin{split}
 E_{A,x}[v]&=\int_Jp\bigl(|v'|^2+(\beta^2+x)|v|^2\bigr)\dd\varphi,\\
 E_{T,x}[v]&=\int_Jp\bigl(|v'|^2+
 [(n-1)\sec^2\varphi+\beta^2-d^2+x]|v|^2\bigr)\dd\varphi.
 \end{split}
\]
These are $E_{0,\beta^2+x}$ on $\mathcal V_0$ and
$E_{1,\beta^2-d^2+x}$ on $\mathcal V_1$, respectively; the spaces
include regularity at $\varphi=\pi/2$ and a free trace at zero, as in
the extension problem of Section~\ref{part:pde}. Thus, on their common domain,
\[
 E_{T,x}[v]-E_{A,x}[v]
 =\int_Jp\bigl((n-1)\sec^2\varphi-d^2\bigr)|v|^2\dd\varphi.
\]
With $d_s$ from \eqref{prelim-extension-normalization}, define
\[
 M_A(x)=d_s\min_{\operatorname{tr}v=1}E_{A,x}[v],\qquad
 M_T(x)=d_s\min_{\operatorname{tr}v=1}E_{T,x}[v],\qquad x\ge0.
\]

Define the positive numbers
\begin{align}
 \lambda_k&=(B+2k)^2,\nonumber\\
 r_k&=4^{s+1}\frac{\sin\pi s}{\pi}(B+2k)
             g_s(k+1)g_s(N+k),\label{ap:eq:Adata}\\
 \nu_k&=(B+1+2k)^2-d^2
       =4(k+1+s)(k+N),\nonumber\\
 t_k&=4^{s+1}\frac{\sin\pi s}{\pi}(B+1+2k)
             g_s(k+1)g_s(N+1+k).\label{ap:eq:Tdata}
\end{align}

\begin{proposition}[Positive spectral representations]\label{ap:prop:positive-series}
The minima defining $M_A$ and $M_T$ are attained uniquely.
These functions have the meromorphic continuations
\begin{equation}\label{ap:eq:real-DN-shifts}
 M_A(z)=a_0(\sqrt z),\qquad
 M_T(z)=a_1(\sqrt{z-d^2}).
\end{equation}
Here the even meromorphic functions $a_m$ are given by
\eqref{gamma-even-symbol}. In particular,
\[
 M_A(x)=c_0(\beta+i\sqrt x),\qquad
 M_T(0)=A_{n,s},\qquad M_A(0)=H_{n,s},
\]
and $M_T(x+ib)=c_1(1+i\tau)$ when $x=\tau^2$.
The spectral representations are
\begin{align}
 M_A(z)&=H_{n,s}+\sum_{k\ge0}\frac{r_kz}{\lambda_k(\lambda_k+z)},
 \label{ap:eq:Arep}\\
 M_T(z)&=A_{n,s}+\sum_{k\ge0}\frac{t_kz}{\nu_k(\nu_k+z)}.
 \label{ap:eq:Trep}
\end{align}
The series and their differentiated series converge locally uniformly
off the poles. The zero-trace A and T operators have eigenvalues
$\lambda_k$ and $\nu_k$, respectively, and the residues of $M_A$ and
$M_T$ at the corresponding poles are $-r_k$ and $-t_k$.

For $x\ge0$ and $b\in\R$, one has
\begin{equation}\label{ap:eq:positive-gain-series}
 \Ree M_T(x+ib)-M_T(x)
 =b^2\sum_{k\ge0}
 \frac{t_k}{(\nu_k+x)((\nu_k+x)^2+b^2)}.
\end{equation}
For $x=\tau^2$ and $b=2d\tau$, this gives
\begin{equation}\label{ap:eq:exact-R}
 R_{0,s}(\tau)=M_T(x)-M_A(x)
 +b^2\sum_{k\ge0}
 \frac{t_k}{(\nu_k+x)((\nu_k+x)^2+b^2)}.
\end{equation}
The spectral data satisfy the exact relations
\begin{align}
 \nu_k-\lambda_k&=4k+2(1+s)-\delta,\nonumber\\
 \frac{t_k}{r_k}&=
 \frac{B+1+2k}{B+2k}\frac{N+s+k}{N+k}>1
 \label{ap:eq:relations}
\end{align}

\end{proposition}

\begin{proof}
Existence, uniqueness and trace continuity follow from the common
form construction in Lemma~\ref{extension-symbol-bounds}.
The relevant constant potentials are positive here because
$\beta>1$ and $\beta^2-d^2=2\beta-1>0$.

The common boundary problem gives
$M_A(z)=F_0(\beta^2+z)$ and $M_T(z)=F_1(\beta^2-d^2+z)$.
For real $z>d^2$, the exact coefficients in
\eqref{gamma-cylinder-symbols} give \eqref{ap:eq:real-DN-shifts}.
Both sides are analytic for $\Ree z>0$, so the identity theorem
extends these formulas to that half-plane, and continuity includes
$z=0$. The T potential differs from the
degree-one potential by the constant $-d^2$. Since the functions
$a_m$ are even, the displayed expressions do not depend on a choice
of square-root branch; they extend the real minima meromorphically.
The evaluations at zero and the identity at $x+ib$ follow from
$\beta-d=1$ and $x+ib-d^2=(\tau+id)^2$.

For angular degree $m\in\{0,1\}$, consider the functions
\[
 (\sin\varphi)^{2s}(\cos\varphi)^m
 P_k^{(s,N+m-1)}(\cos 2\varphi).
\]
Set $y=\cos2\varphi$. For any polynomial $P$, the change of variables gives
\begin{align*}
 &\| (\sin\varphi)^{2s}(\cos\varphi)^mP(\cos2\varphi)\|_{\mathcal H}^2\\
 &\qquad=2^{-(N+m+s+1)}\int_{-1}^1
       |P(y)|^2(1-y)^s(1+y)^{N+m-1}\dd y.
\end{align*}
The Jacobi orthogonality weight is given in
\cite[Table~18.3.1]{NIST}; its parameters here satisfy
$s>-1$ and $N+m-1>-1$. Polynomial density gives completeness in
this weighted $L^2$ space. Thus the displayed eigenfunctions are
complete in $\mathcal H$ after normalization. Substituting the
Jacobi differential equation \cite[Table~18.8.1]{NIST} into
\eqref{angular-DN-equation} gives the listed eigenvalues. More explicitly, the angular eigenvalue
before adding the constant $\beta^2$ is
$(m+2s+2k)(m+n+2k)$, and
\[
 (m+2s+2k)(m+n+2k)+\beta^2=(N+m+s+2k)^2.
\]
The A operator uses $m=0$; the T operator uses $m=1$ and subtracts
$d^2$. This gives $\lambda_k$ and $\nu_k$, respectively, and agrees
with \eqref{ap:eq:real-DN-shifts}.

They have trace zero at $\varphi=0$: their value is
$O(\varphi^{2s})$ and their derivative is $O(\varphi^{2s-1})$ there.
Thus their derivative energy near zero is bounded by a constant times
$\int_0^1\varphi^{2s-1}\dd\varphi<\infty$.
At the other endpoint they are $(\cos\varphi)^m$ times a smooth
function of $\cos^2\varphi$, so they satisfy the regular finite-energy
condition used to define the degree-$m$ form domain. To check the
core approximation explicitly, cutting off within distance $\epsilon$
of zero costs $O(\epsilon^{2s})$ in derivative energy, whereas cutting
off near $\pi/2$ costs $O(\epsilon^{n+2m-2})$. Both tend to zero.
Integration by parts against compact interior profiles has no
boundary term; these cutoffs and form-density extend the identity
to all zero-trace form tests. The differential
equation therefore identifies these functions as elements of the
associated selfadjoint operator domain.

For angular degree $m$ the numerator poles occur at
$z=-(N+m+s+2k)^2$ before the real constant shift.
At such a pole, one Gamma factor has argument $-k$ and its derivative
with respect to $z$ is $1/[4(N+m+s+2k)]$.
The Gamma residue formula $\operatorname{Res}_{w=-k}\Gamma(w)=(-1)^k/k!$
and the reflection formula \cite[\S5.2(i) and 5.5.3]{NIST} give
\[
 \frac{(-1)^k}{\Gamma(-s-k)}
 =-\frac{\sin\pi s}{\pi}\Gamma(1+s+k)
\]
and yield \eqref{ap:eq:Adata} and \eqref{ap:eq:Tdata}.

Here is a direct justification that the spectral series exhausts the
function. For either real form let $B_0$ be its zero-trace operator,
let $w_0$ be the trace-one minimizer, and let $(e_k)$ be its complete
orthonormal eigenbasis, with eigenvalues $\ell_k>0$. The trace-one
solution at parameter $z$ is
\[
 w_z=w_0-z(B_0+z)^{-1}w_0.
\]
The correction $w_z-w_0$ has zero trace. Writing
$E_0(\cdot,\cdot)$ for the polarized real-potential form at parameter
zero, the minimizing property of $w_0$ gives
\[
 E_0(w_0,w_z-w_0)=0.
\]
The boundary identity tested against $w_0$ is
$M(z)=d_s\bigl(E_0(w_0,w_z)+z\ip{w_0}{w_z}\bigr)$.
Subtracting $M(0)=d_sE_0(w_0,w_0)$ therefore gives
\begin{align*}
 M(z)-M(0)
 &=d_s z\ip{w_0}{w_z}\\
 &=d_s z\|w_0\|_{\mathcal H}^2
 -d_s z^2\ip{w_0}{(B_0+z)^{-1}w_0}\\
 &=d_s\sum_{k\ge0}
 \frac{\ell_k z}{\ell_k+z}|\ip{w_0}{e_k}|^2.
\end{align*}
The expansion is justified by $w_0\in\mathcal H$ and Parseval's
identity. Its residues are $-d_s\ell_k^2|\ip{w_0}{e_k}|^2$.
Comparison with the Gamma residues above gives exactly
\eqref{ap:eq:Arep} and \eqref{ap:eq:Trep}.

The Gamma-ratio estimates give $\lambda_k,\nu_k\asymp k^2$ and
$r_k,t_k\asymp k^{1+2s}$. On any compact set avoiding the poles,
the summands in \eqref{ap:eq:Arep}--\eqref{ap:eq:Trep} are
$O(k^{2s-3})$; their differentiated summands have at least the same
summable decay. This proves the asserted local uniform convergence
and permits termwise differentiation.

Taking real parts term by term in \eqref{ap:eq:Trep} gives
\eqref{ap:eq:positive-gain-series}. This is the spectral expansion of
the correction energy in \eqref{dn-positive-correction}: if $v_x$ is
the real trace-one minimizer for $E_{T,x}$ and $\Phi$ solves the same
equation with $ib$ added to the potential, then
\[
 \Ree M_T(x+ib)-M_T(x)=d_sE_{T,x}[\Phi-v_x].
\]
In the series each summand is positive when $b\ne0$. For $x=\tau^2$ and $b=2d\tau$,
\eqref{gamma-cylinder-symbols} and complex conjugation give
\[
 a_0(\tau)=M_A(x),\qquad
 t_{0,0}(\tau)=\overline{M_T(x+ib)}.
\]
Taking real parts proves \eqref{ap:eq:exact-R}.
Finally, subtracting the displayed eigenvalues and applying
$g_s(z+1)=(z+s)g_s(z)/z$ gives \eqref{ap:eq:relations}.
\end{proof}

Formula~\eqref{ap:eq:exact-R} isolates the remaining difficulty.
The difference $M_T(x)-M_A(x)$ can be negative when $\delta<0$,
whereas the series contributed by $b=2d\tau$ is nonnegative.
We next show that the latter compensates for the former on
$x\ge5/4$, with the order needed in the scalar target.

\subsection{The complex drift comparison}\label{ap:sec:drift-gate}
We retain the spectral data \eqref{ap:eq:Adata}--\eqref{ap:eq:Tdata}
and the constant $C$ assigned to each row of \eqref{sp-eq:boxes}.
The drift estimate will use the bound
\[
 -\frac{M_A''(x)}{M_A'(x)}\ge\frac{1-s}{x+C},\qquad x\ge5/4.
\]
We obtain it by averaging a rational inequality along the positive
spectral series. We prove that rational inequality first. Write
\[
 c=1+s,\qquad m=N-1,\qquad B=m+c.
\]
For $y\ge B$, $x\ge5/4$, and $a=y^2+x$, define
\[
 A_\pm(y,x)=
 \frac{((y\pm c)^2-m^2)(y\pm1)}{y}
 \frac{2a\pm4y+4}{(a\pm4y+4)^2},\qquad
 A_{\mathrm{av}}(y,x)=\frac{A_+(y,x)+A_-(y,x)}2.
\]

\begin{lemma}[A rational comparison]\label{ap:lem:pointwise}
For each row of \eqref{sp-eq:boxes}, with $s>0$,
\begin{equation}\label{ap:eq:pointwise}
 A_{\mathrm{av}}(y,x)\ge\frac{2(y^2-C)}{y^2+x},
 \qquad y\ge B,\quad x\ge5/4.
\end{equation}
\end{lemma}
\begin{proof}
Put $Y=y^2$, $L=c^2-m^2$, and $a=Y+x$. Direct expansion gives
\[
A_{\mathrm{av}}-\frac{2(Y-C)}a
=\frac{2D(Y,x,c)}{a\bigl((a+4)^2-16Y\bigr)^2},
\]
where
\begin{align*}
D={}&(C+L+2c)a^4+
[-12(1+c)Y+4L+20c+16C]a^3\\
&+32[Y^2-(C+3+2c)Y+3C+2c]a^2\\
&+64(1-Y)[c+4C-(c+4)Y]a
-256(Y-C)(1-Y)^2.
\end{align*}
The denominator is positive: its unsquared factors are
$(y+2)^2+x$ and $(y-2)^2+x$.
We prove $D\ge0$ on $Y\ge B^2$, $x\ge5/4$ by three successive
reductions. First, coefficient bounds show that $D$ is increasing in
$x$. At $x=5/4$, a cubic lower bound for its $Y$ derivative shows
that it is increasing in $Y$. It then suffices to check the three
rational endpoint inequalities at $Y=B^2$. The tables below give
the exact bounds for these steps.

The coefficients of $x^4,x^3,x^2$ are respectively
\begin{align*}
&C+L+2c,\\
&4(C+L-c-3)Y+4L+20c+16C,\\
&[6(C+L-4c)-4]Y^2+
4(3L-c+4C-24)Y+96C+64c.
\end{align*}
They are positive on all three parameter rows. This is checked by
substituting the endpoints into their quadratic polynomials in $c$:
the displayed coefficient polynomials have fixed derivative sign
on each row, except those already positive termwise.
For example, their coefficients are bounded below, in the displayed
order, by $7$, $7Y+175$, and $4Y^2+13Y+1120$.
The coefficient of $x$ is $4J(Y)$, where
\[
J(Y)=(C+L-7c+7)Y^3+(3L-c-4C+16)Y^2
-16(C+4)Y+16(c+4C).
\]
The following lower polynomials suffice:
\[
\begin{array}{c|c|c|c|c|c}
n&J(Y)\text{ lower bound }j(Y)&Y_0&j(Y_0)&j'(Y_0)&j''(Y_0)\\ \hline
7&5Y^3-45Y^2-240Y+720&12&0&840&270\\
8&\frac{43}{10}Y^3-64Y^2-288Y+912&18&\frac{348}{5}&\frac{7938}{5}&\frac{1682}{5}\\
9&4Y^3-89Y^2-360Y+1200&26&1980&3124&446
\end{array}
\]
The coefficient bounds follow from endpoint substitution. For the
coefficient of $Y^3$ its derivative $2c-7$ is negative, and for the
coefficient of $Y^2$ its derivative $6c-1$ is positive.
For example, in the $n=9$ row, $C=37/2$, $m=7/2$, and
$163/100\le c\le209/125$. Thus
\begin{align*}
 \frac{d}{dc}(C+L-7c+7)&=2c-7<0,\\
 C+L-7c+7
 &\ge\frac{37}{2}+\left(\frac{209}{125}\right)^2
       -\frac{49}{4}-7\frac{209}{125}+7
   =\frac{271349}{62500}>4.
\end{align*}
The other coefficient bounds are obtained in the same way.

To read the last three columns of the table, use the exact Taylor
identity for any cubic polynomial $P$:
\[
 P(Y_0+t)=P(Y_0)+P'(Y_0)t+\frac{P''(Y_0)}2t^2
                          +\frac{P^{(3)}(Y_0)}6t^3,\qquad t\ge0.
\]
For each lower polynomial $j$, the displayed values are nonnegative,
and $j'(Y_0)$ and $j^{(3)}$ are positive. Hence $j(Y)>0$ for $Y>Y_0$.
Since $B^2>Y_0$ in every row, this gives $J(Y)>0$ on the required
domain, so $D$ is increasing in $x\ge0$ there.

It remains to consider $Q(Y)=D(Y,5/4,c)=\sum_{i=0}^4q_iY^i$.
Its coefficients are
\begin{align*}
q_4&=C+L-10c+20,&
q_3&=9L-15c-11C-61,\\
q_2&=(683C+195L-852c+2638)/8,\\
q_1&=(-4851C+425L+799c-11991)/16,\\
q_0&=(194481C+2625L+57330c)/256.
\end{align*}
The following rational lower bounds follow again by endpoint
substitution (all four coefficient polynomials are monotone on
each row):
\[
\begin{array}{c|rrrr|rrr}
n&q_4\ge&q_3\ge&q_2\ge&q_1\ge&h(Y_0)&h'(Y_0)&h''(Y_0)\\ \hline
7&15&-245&1031&-4175&18409&10342&2850\\
8&13&-301&1206&-5126&48982&20448&3810\\
9&12&-376&1500&-6532&152588&41688&5232
\end{array}
\]
Here $Y_0$ is as in the preceding table and
$h(Y)=4q_4^{\rm lo}Y^3+3q_3^{\rm lo}Y^2+2q_2^{\rm lo}Y+q_1^{\rm lo}$.
Apply the same cubic Taylor identity to $h$: its leading coefficient
and all three values $h(Y_0),h'(Y_0),h''(Y_0)$ are positive.
Therefore $h(Y)>0$ for $Y\ge Y_0$, and the coefficient bounds give
$Q'(Y)\ge h(Y)>0$ there. It is enough to check $Q(B^2)>0$.

At $y=B=m+c$ the backward term is zero. Inequality~\eqref{ap:eq:pointwise}
at $x=5/4$ is therefore equivalent to
\[
C\ge B^2-2c(B+1)f(B),\qquad
f(B)=\frac{(B^2+5/4)(B^2+5/4+2(B+1))}
{(B^2+5/4+4(B+1))^2}.
\]
This last check can be made with small rational bounds.
Set $z=B-1+9/[4(B+1)]$; then $f=z(z+2)/(z+4)^2$.
Both $z$ and this function of $z$ are increasing for $B\ge7/2$.
At the lower endpoints $B=7/2,107/25,513/100$ respectively,
one finds $f\ge3/10,7/20,2/5$. For each of these three choices
of $f_0$, the quadratic $B^2-2f_0(B-m)(B+1)$ is increasing in $B$.
At the upper endpoint it has the following value:
\[
\begin{array}{c|c|c|c}
n& f_0&B_{\max}&B_{\max}^2-2f_0(B_{\max}-m)(B_{\max}+1)\\ \hline
7&3/10&3551/1000&24349351/2500000<11\\
8&7/20&871/200&5554723/400000<14\\
9&2/5&1293/250&5779349/312500<37/2
\end{array}
\]
This proves strict positivity of $Q(B^2)$, hence $D>0$ throughout
the required domain. The displayed rational identity gives
\eqref{ap:eq:pointwise}.
\end{proof}

\begin{proposition}[A lower bound for the derivative quotient]\label{ap:prop:quotient}
For each parameter row and every $x\ge5/4$,
\[
\frac{-M_A''(x)}{M_A'(x)}\ge\frac{1-s}{x+C}.
\]
If $0<e<1/2$, then also
\[
\frac{-M_A''(x)}{M_A'(x-e)}\ge
\frac{1-s}{x+C}
\left(\frac{B^2+x-e}{B^2+x}\right)^2.
\]
\end{proposition}

\phantomsection\label{ap:sec:quotient-proof}
\begin{proof}
Differentiate the already proved series \eqref{ap:eq:Arep} and write
\[
 S_\ell(x)=\sum_{k\ge0}\frac{r_k}{(\lambda_k+x)^\ell},
 \quad \ell=2,3,\qquad M_A'(x)=S_2(x),\quad -M_A''(x)=2S_3(x).
\]
Put $D_k=g_s(k+1)g_s(k+N)$ and $F_k=k(k+N-1)D_k$.
The Gamma recurrence gives
\[
 F_{k+1}-F_k=c(B+2k)D_k.
\]
Set $a_k=\lambda_k+x$. Summation by parts, once using
\[
 \sum_{k\ge0}(F_{k+1}-F_k)a_k^{-2}
 =\sum_{k\ge0}F_{k+1}(a_k^{-2}-a_{k+1}^{-2})
\]
and once using $\sum_{k\ge1}F_k(a_{k-1}^{-2}-a_k^{-2})$, gives
\begin{equation}\label{ap:eq:residue-averages}
 cS_2(x)
 =\sum_{k\ge0}\frac{r_kA_+(B+2k,x)}{(\lambda_k+x)^2}
 =\sum_{k\ge0}\frac{r_kA_-(B+2k,x)}{(\lambda_k+x)^2}.
\end{equation}
The common positive factor in $r_k$ cancels in these identities.
The boundary terms vanish because $F_k=O(k^{2+2s})$ and
$a_k^{-2}=O(k^{-4})$, with $s<1$; the backward term at $k=0$
is zero since $F_0=0$ and $A_-(B,x)=0$.

Average the two identities in \eqref{ap:eq:residue-averages} and
apply Lemma~\ref{ap:lem:pointwise}. Since $y^2=\lambda_k$,
\[
 cS_2(x)\ge2S_2(x)-2(x+C)S_3(x).
\]
As $2-c=1-s$, division by the positive $S_2(x)$ gives
$-M_A''(x)/M_A'(x)\ge(1-s)/(x+C)$.
Finally, $\lambda_k\ge\lambda_0=B^2$ implies term by term
\[
 M_A'(x-e)\le
 \left(\frac{B^2+x}{B^2+x-e}\right)^2M_A'(x).
\]
Combining this with the first inequality proves the second claim.
\end{proof}

The quotient estimate is now available for the comparison of the two
extension energies. We keep the possible real energy loss and the
positive drift contribution in a single bound for $R_{0,s}$.

\begin{proposition}[The high-frequency zero-mode gap]\label{ap:prop:high-frequency}
For a row of \eqref{sp-eq:boxes}, put $e=-\delta=d^2-(n-1)$. On the part of that row where
$0<e<1/2$, there exists $\varepsilon_n>0$, independent of $s$,
such that
\[
 R_{0,s}(\sqrt x)\ge\varepsilon_neM_A'(x-e)>0,
 \qquad x\ge5/4.
\]
\end{proposition}
\begin{proof}
Let $e=-\delta\in(0,1/2)$ and $b=2d\sqrt x$. The difference of the
two real potentials is
\[
 W_0=(n-1)\sec^2\varphi-d^2=-e+(n-1)\tan^2\varphi\ge-e.
\]
The T form domain is contained in the A form domain. For $x\ge1$,
comparison of their trace-one minima therefore gives
$M_T(x)\ge M_A(x-e)$. The positive series \eqref{ap:eq:Arep} shows that
$M_A$ is increasing and concave, so
\[
 M_T(x)-M_A(x)\ge-eM_A'(x-e).
\]
Set $\rho=1+(n-2-e)^{-1}$. The coefficient relations \eqref{ap:eq:relations}
imply $t_k\ge r_k$ and $\nu_k+x\le\rho(\lambda_k+x)$, since
\[
 \frac{4k+2(1+s)+e}{(2k+N+s)^2}\le\frac1{n-2-e}.
\]
Indeed, after multiplication the difference is the square
$[2k+N+s-(n-2-e)]^2$; use
$2(N+s)-2(1+s)-e=n-2-e$.
Thus \eqref{ap:eq:positive-gain-series} gives
\[
 \Ree M_T(x+ib)-M_T(x)
 \ge\frac{b^2[-M_A''(x)]}
 {2\rho^3[1+b^2/(\nu_0+x)^2]}.
\]
Here we bounded the extra denominator factor by
$1+b^2/(\nu_0+x)^2$ and used
$-M_A''(x)=2\sum r_k/(\lambda_k+x)^3$. Consequently
\begin{equation}\label{ap:eq:master-lower}
 R_{0,s}(\sqrt x)\ge-eM_A'(x-e)+
 \frac{4d^2x[-M_A''(x)]}
 {2\rho^3[1+4d^2x/(\nu_0+x)^2]},\qquad x\ge1.
\end{equation}

Apply Proposition~\ref{ap:prop:quotient} to the second term and
factor out the positive $eM_A'(x-e)$. With $\nu_0=4(1+s)N$, set
\begin{equation}\label{ap:eq:gate}
 \mathfrak g(s,x)=\frac{2d^2(1-s)}e\,\rho^{-3}\,
 \frac{x}{x+C}\,
 \left(1+\frac{4d^2x}{(\nu_0+x)^2}\right)^{-1}
 \left(\frac{B^2+x-e}{B^2+x}\right)^2.
\end{equation}
The resulting scalar estimate is
\begin{equation}\label{ap:eq:normalized-drift-lower}
 R_{0,s}(\sqrt x)\ge
 eM_A'(x-e)\bigl(\mathfrak g(s,x)-1\bigr),\qquad x\ge5/4.
\end{equation}
It remains to show a uniform margin $\mathfrak g(s,x)>1$ on each row.
This function increases in both $x\ge5/4$ and $s$ on the part of a
parameter row where $e>0$.
For $s$, the derivative of the logarithm of $(1-s)d^2/e$ is
\[
-\frac1{1-s}-\frac2d+\frac{2d}e>0
\quad\Longleftrightarrow\quad
2(n-1)(1-s)>de.
\]
On the displayed rows, $d\le3$, $e<1/2$, $1-s\ge41/125$, and
$n-1\ge6$, which give this strict inequality.
The factors $\rho^{-3}$ and the final square also increase with $s$;
the drift denominator decreases because $d$ decreases and $\nu_0$
increases. For $x$, with $a_{\mathrm d}=4d^2$ and $\nu=\nu_0$, the derivative
of
\[
\frac{x}{x+C}\left(1+\frac{a_{\mathrm d}x}{(\nu+x)^2}\right)^{-1}
\]
has the sign of
$C(\nu+x)^3+a_{\mathrm d}x^2(x+2C-\nu)$.
It is positive because $2C\ge\nu_0$ on each row.
The final square in~\eqref{ap:eq:gate} increases with $x$ as well.
Thus only $x=5/4$ and the lower endpoint of $s$ require checking
(with the limit $s\downarrow0$ in the first row).

At $x=5/4$ and the lower endpoints of the first two rows, the five
factors in \eqref{ap:eq:gate}, in their displayed order, give
\begin{align*}
 n=7:\quad &50\cdot\frac12\cdot\frac5{49}
                  \cdot\frac45\cdot\frac9{10}=\frac{90}{49}>1,\\
 n=8:\quad &26\cdot\frac35\cdot\frac5{61}
                  \cdot\frac9{10}\cdot\frac{19}{20}
                     =\frac{6669}{6100}>1.
\end{align*}
Both products are strict lower bounds for $\mathfrak g$, with
$s\downarrow0$ in the first row.
The closest case, $n=9$ and $s=63/100$, has the small rational bounds
\[
\begin{gathered}
\frac{2d^2(1-s)}e>\frac{128}5,\qquad
\rho^{-3}>\frac{33}{50},\qquad \frac{x}{x+C}=\frac5{79},\\
\left(1+\frac{4d^2x}{(\nu_0+x)^2}\right)^{-1}>\frac{957}{1000},\qquad
\left(\frac{B^2+x-e}{B^2+x}\right)^2>\frac{49}{50}.
\end{gathered}
\]
Their product is $6189876/6171875>1$.
These three strict margins give $\varepsilon_n>0$ such that
$\mathfrak g(s,x)\ge1+\varepsilon_n$ throughout each row where $e>0$
and $x\ge5/4$. Returning to \eqref{ap:eq:normalized-drift-lower}, we obtain
\begin{equation}\label{ap:eq:strict-drift-lower}
 R_{0,s}(\sqrt x)\ge eM_A'(x-e)\bigl(\mathfrak g(s,x)-1\bigr)
       \ge\varepsilon_n eM_A'(x-e),\qquad x\ge5/4.
\end{equation}
\end{proof}

\subsection{Coercivity and interior estimates}\label{full-hardy-comparison}\label{assembly-section}\label{interior-completion}
\begin{proposition}[Zero-mode comparison]\label{sp-prop:scalar}
For $(n,s)$ in \eqref{sp-eq:boxes}, with $s>0$ and $A_{n,s}>H_{n,s}$,
there exists $\kappa_0(n,s)>0$ such that
\begin{equation}\label{sp-eq:explicit-zero-gap}
 R_{0,s}(\tau)\ge\kappa_0(n,s)(1+\tau^2)^{s-1},\qquad \tau\in\R.
\end{equation}
\end{proposition}

\begin{proof}
If $\delta\ge0$, Corollary~\ref{endpoint} gives the conclusion.
Suppose $\delta<0$, and put $e=-\delta\in(0,1/2)$.
On $\tau^2\le5/4$, Proposition~\ref{sp-prop:low-direct} gives
\[
 R_{0,s}(\tau)\ge A_{n,s}-H_{n,s}>0.
\]
On the complementary interval, Proposition~\ref{ap:prop:high-frequency}
gives $R_{0,s}(\sqrt x)\ge\varepsilon_n eM_A'(x-e)$.
To recover the required order, differentiate \eqref{ap:eq:Arep}:
\[
 M_A'(x-e)=\sum_{k\ge0}\frac{r_k}{(\lambda_k+x-e)^2}.
\]
For fixed $n,s$, the real-axis Gamma-ratio growth in
Lemma~\ref{uniform-gamma-symbols} gives
\[
 r_k\ge c_{n,s}(k+1)^{1+2s},\qquad
 \lambda_k\le C_{n,s}(k+1)^2,\qquad k\ge0,
\]
with the finitely many initial terms absorbed in the positive constants.
For $x\ge4$, put $K=\lceil\sqrt x\rceil$ and retain only
$K\le k\le2K$. There are at least $\sqrt x$ terms; each numerator
is at least $c_{n,s}x^{s+1/2}$ and each squared denominator is at most
$C_{n,s}x^2$. Consequently
\[
 M_A'(x-e)\ge c_{n,s}x^{s-1},\qquad x\ge4.
\]
For $5/4\le x\le4$, the positive $k=0$ term gives a lower bound, so
\[
 M_A'(x-e)\ge c_{n,s}(1+x)^{s-1},\qquad x\ge5/4.
\]
Thus the high-frequency estimate has the required coercive order.
On the low-frequency interval $(1+\tau^2)^{s-1}\le1$; taking the
smaller constant proves the scalar bound for all $\tau\in\R$.
\end{proof}

We now combine the zero mode with the higher angular modes.

\begin{theorem}[Coercivity for all angular degrees]\label{sp-thm:main}
For integer $n\ge2$ and $0<s<1$, the following are equivalent:
\begin{enumerate}
 \item $A_{n,s}>H_{n,s}$;
 \item there exists $\kappa_{n,s}>0$ such that
 \begin{equation}\label{sp-eq:full-gap}
 R_{j,s}(\tau)\ge\kappa_{n,s}(1+j^2+\tau^2)^{s-1}
 \quad\text{for every }j\ge0\text{ and }\tau\in\R.
 \end{equation}
\end{enumerate}
Moreover, under these conditions there exist
$0<\alpha<\min\{1,2s\}$ and $\kappa>0$, depending only on $n,s$, such that
\begin{equation}\label{positive-radial-gap}
 \Ree t_{j,\alpha}(\tau)-a_j(\tau)
 \ge\kappa(1+j^2+\tau^2)^{s-1}
 \qquad(j\ge0,\ \tau\in\R).
\end{equation}
\end{theorem}

\begin{proof}[Proof of Theorem~\ref{sp-thm:main}]
Necessity follows by setting $j=0$, $\tau=0$ in \eqref{sp-eq:full-gap},
since $R_{0,s}(0)=A_{n,s}-H_{n,s}$.

For sufficiency, Corollary~\ref{endpoint} settles $\delta\ge0$.
Otherwise Subsection~\ref{sp-sec:parameters} places $(n,s)$ in
\eqref{sp-eq:boxes}. Proposition~\ref{sp-prop:scalar} supplies the zero-mode
lower bound.

For the higher modes, $\delta>-1/2$ gives
$\delta+2j\ge\delta+2>0$ for every $j\ge1$.
Proposition~\ref{analytic-gap}, applied directly with $\vartheta=0$,
therefore gives
\begin{equation}\label{sp-eq:existing-higher}
 R_{j,s}(\tau)\ge \kappa^{\mathrm{ang}}_{n,s}(\delta+2j)(1+j^2+\tau^2)^{s-1},
 \qquad j\ge1,\quad \kappa^{\mathrm{ang}}_{n,s}>0.
\end{equation}
Since $\delta>-1/2$, \eqref{sp-eq:existing-higher} gives the lower bound for
all $j\ge1$, with constant at least $\kappa^{\mathrm{ang}}_{n,s}(\delta+2)>0$.
Taking the smaller of these two constants proves \eqref{sp-eq:full-gap}.
Finally, the weight-continuity estimate of Lemma~\ref{shift-gap} gives
\eqref{positive-radial-gap}: choose a positive admissible $\alpha$ for
which the perturbation is at most half the unweighted lower bound.
\end{proof}

As $s\uparrow1$, the difference of the two symbols tends to
\[
 R_{j,1}(\tau)=2j+\frac{(n-2)(10-n)}4.
\]
Thus the limiting critical dimension is ten, as for the classical Laplacian.

% MODULE assembly
\begin{proof}[Proof of Theorem~\ref{full}]
For $n\ge2$, the positive-weight conclusion
\eqref{positive-radial-gap} of Theorem~\ref{sp-thm:main} supplies
\eqref{pde-input} throughout
\eqref{dimension-range}, with constants depending only on $n,s$.
Theorem~\ref{pde-interface} then gives \eqref{full-bound}.

For $n=1$, Theorem~\ref{explicit-exponents} gives the conclusion.

For any exponent supplied by Theorem~\ref{full} or
Theorem~\ref{explicit-exponents}, the corresponding estimate for a
solution in $B_{2R}(x_0)$ is
\begin{equation}\label{scaled-interior-bound}
 \|u\|_{L^\infty(B_R(x_0))}
 +R^\alpha[u]_{C^\alpha(B_R(x_0))}
 +R^{s-n/2}[u]_{H^s(B_R(x_0))}
 \le C \|u\|_{L_s^1(x_0,R)},
\end{equation}
where the rescaled weighted norm is defined in
\eqref{rescaled-tail-definition}. Indeed, put $v(x)=u(x_0+Rx)$.
Its nonlinearity is $R^{2s}f$, and stability follows by changing
variables in \eqref{stability}. Directly,
\[
 \|v\|_{L_s^1(\R^n)}=\|u\|_{L_s^1(x_0,R)},\qquad
 [v]_{C^\alpha(B_1)}=R^\alpha[u]_{C^\alpha(B_R(x_0))},\qquad
 [v]_{H^s(B_1)}=R^{s-n/2}[u]_{H^s(B_R(x_0))}.
\]
Applying the normalized estimate to $v$ proves
\eqref{scaled-interior-bound}.
\end{proof}

\section{Weak solutions and optimality}\label{part:weak}\label{weak-section}

We now extend Theorem~\ref{full} to finite-energy weak solutions of
the Dirichlet problem. The solution class will also contain the
classical logarithmic singular solution used to prove optimality.

\subsection{The weak Dirichlet class}\label{weak-preliminaries}

Let $\Omega\subset\R^n$
be smooth and bounded, $0<s<1$, and put
\[
 Q_\Omega=\R^{2n}\setminus(\Omega^c\times\Omega^c),\qquad
 \mathcal E_\Omega(w)=\frac{c_{n,s}}2\iint_{Q_\Omega}
       \frac{|w(x)-w(y)|^2}{|x-y|^{n+2s}}\dd x\dd y.
\]
Write $\mathcal E_\Omega(v,w)$ for its polarization, so
$\mathcal E_\Omega(w)=\mathcal E_\Omega(w,w)$. Following \cite{SPapprox}, define
\[
 H^s_\Omega=\{w:w|_\Omega\in L^2(\Omega),\ \mathcal E_\Omega(w)<\infty\},\qquad
 H^s_{\Omega,0}=\{w\in H^s_\Omega:w=0\text{ a.e.\ on }\Omega^c\}.
\]
The norm on $H^s_\Omega$ is
$\|w\|_{H^s_\Omega}^2=\|w\|_{L^2(\Omega)}^2+\mathcal E_\Omega(w)$.
This space requires square integrability only in the domain, rather
than membership in $L^2(\R^n)$. Its energy includes the interactions
between interior and exterior points, but omits those between two
exterior points. In contrast, $H^s(\Omega)$ uses only the interaction
over $\Omega\times\Omega$. For $w=0$ on $\Omega^c$, the two energy
forms satisfy $\mathcal E_\Omega(w)=\mathcal E_s(w,w)$.

The exterior datum $g$ satisfies the condition used in
\cite[Definition~1.1]{SPapprox}: for some $\alpha_g$ and $C_g<\infty$,
\begin{equation}\label{we-eq:g}
 \max\{0,2s-1\}<\alpha_g<s,\qquad
 |g(x)-g(z)|\le C_g|x-z|^{\alpha_g}
 \quad(x\in\Omega^c,\ z\in\partial\Omega).
\end{equation}
This condition allows $g$ to be unbounded on $\Omega^c$.
Recall the boundary distance $d_\Omega$, the weighted norm, and
the space $L_s^1(\R^n)$
from Section~\ref{prelim-notation}.

An \emph{$L^1$-weak Dirichlet solution}, in the sense of
\cite[Definition~1.1]{SPapprox}, satisfies
\begin{equation}\label{we-eq:weakclass}
 u\in L^1(\Omega),\qquad u=g\text{ on }\Omega^c,\qquad
 f(u)\in L^1(\Omega,d_\Omega^s\dd x),
\end{equation}
and
\begin{equation}\label{we-eq:transposition}
 \int_\Omega u(-\Delta)^s\varphi
 +\int_{\Omega^c}g\,\mathcal N_s\varphi
 =\int_\Omega f(u)\varphi
\end{equation}
for the Dirichlet test functions $\varphi=0$ on $\Omega^c$
such that $\varphi$ and $(-\Delta)^s\varphi$ are bounded in $\Omega$.
Here our sign convention is
\[
 \mathcal N_s\varphi(x)=(-\Delta)^s\varphi(x)
 =-c_{n,s}\int_\Omega\frac{\varphi(y)}{|x-y|^{n+2s}}\dd y
 \qquad\text{for a.e. }x\in\Omega^c.
\]
In particular, $\mathcal N_s\varphi\le0$ for nonnegative $\varphi$.

\begin{lemma}[Dirichlet test functions]\label{we-lem:dirichlet-tests}
Every test function $\varphi$ in \eqref{we-eq:transposition} satisfies
$|\varphi(x)|\le C_\varphi d_\Omega(x)^s$ in $\Omega$ and belongs to
$H^s_{\Omega,0}$.
\end{lemma}
\begin{proof}
For the bounded zero-exterior test class in
\eqref{we-eq:transposition}, the boundary regularity stated in
\cite[Definition~1.1, footnote~4]{SPapprox} gives
$|\varphi(x)|\le C_\varphi d_\Omega(x)^s$.
We establish energy membership by comparison with a variational
solution. Let $z\in H^s_{\Omega,0}$ solve
\[
 (-\Delta)^sz=(-\Delta)^s\varphi\quad\text{in }\Omega,
 \qquad z=0\quad\text{on }\Omega^c.
\]
Existence follows from the coercive Dirichlet form, since the source
is bounded. Proposition~1.1 of \cite{RSboundary}, applied to this
variational solution, gives $z\in C^s(\R^n)$ and
$|z|\le Cd_\Omega^s$. The bounded difference $\varphi-z$ is
distributionally harmonic, hence has a smooth interior representative
by Lemma~\ref{harmonic-weyl}. Both boundary estimates make this
representative tend to zero at every boundary point. The maximum
principle gives $\varphi=z\in H^s_{\Omega,0}$.
\end{proof}

Stability means $f'(u)\in L^1_{\mathrm{loc}}(\Omega)$ and
\begin{equation}\label{we-eq:stability}
 \int_\Omega f'(u)\xi^2\le\mathcal E_\Omega(\xi)
 \qquad(\xi\in C_c^\infty(\Omega)).
\end{equation}
Throughout this section we additionally require $u\in H^s_\Omega$.
Thus we consider $L^1$-weak solutions with finite energy. We keep
this terminology distinct from \emph{energy solution} in
\cite[Definition~1.2]{SPapprox}, which also requires
$f(u)\in L^2(\Omega)$.
Membership in $L_s^1(\R^n)$ is automatic here: $u\in L^2(\Omega)$,
and \eqref{we-eq:g} gives $|g(x)|\le C(1+|x|^{\alpha_g})$ outside,
with $\alpha_g<2s$. We retain the tail norm in the statements because
it is the quantity controlling the estimates.

Theorem~\ref{full} gives a uniform interior estimate for every
nonnegative convex $C^1_{\mathrm{loc}}(\R)$ reaction, including positive scalar multiples
arising under dilation. To apply it to a weak solution, we first
construct bounded stable approximations converging to $v\le u$.
The uniform estimate makes $v$ locally bounded. Stability and
convexity then imply a linear equation for $u-v$, which transfers
local boundedness to $u$.
We use the following Green estimates to construct a bounded harmonic
extension of the exterior datum and to prove the linear boundedness
statement needed for this transfer.

For a smooth bounded domain, let $z$ be the zero-exterior solution
of $Lz=g$ given by the Green operator. If $n>2s$, then
\begin{equation}\label{prelim-linear-Lp}
 \|z\|_{L^q(\Omega)}\le C\|g\|_{L^p(\Omega)},\qquad
 1<p<\frac n{2s},\qquad \frac1q=\frac1p-\frac{2s}n.
\end{equation}
If $p>n/(2s)$, the solution is bounded. At equality, applying
\eqref{prelim-linear-Lp} with slightly smaller $p$ yields every
finite $q$.

When $n\le2s$, $g\in L^p(\Omega)$ with $p>1$ already
gives boundedness. These are \cite[Proposition~C.5 and
Remark~C.6]{SPapprox}; the subcritical estimate follows by comparison
with $I_{2s}(|g|\mathbf1_\Omega)$ and the Riesz-potential embedding.

\begin{lemma}[Bounded harmonic extension]\label{we-lem:harmonic-lift}
Let $g$ satisfy \eqref{we-eq:g}, and suppose some $w_0\in H^s_\Omega$
equals $g$ on $\Omega^c$. Minimizing $\mathcal E_\Omega$ in
$w_0+H^s_{\Omega,0}$ gives an $s$-harmonic function $h$ with
\[
 h=g\text{ on }\Omega^c,\qquad
 h\in H^s_\Omega\cap L^\infty(\Omega)\cap L_s^1(\R^n).
\]
\end{lemma}
\begin{proof}
The direct method and the fractional Poincar\'e inequality give the
minimizer and its variational equation. This also gives the
$L^1$-weak formulation, as in \cite[proof of Proposition~C.1]{SPapprox}.

We verify boundedness directly. Condition \eqref{we-eq:g} implies
that $g$ is bounded on bounded subsets of $\Omega^c$ and
$|g(x)|\le C(1+|x|^{\alpha_g})$ there, with $\alpha_g<s$.
Take $\chi\in C_c^\infty(\R^n)$, with $0\le\chi\le1$, equal to one on a neighborhood
of $\overline\Omega$. Then $\bar h=\chi h$ has bounded exterior
data $\chi g$ and satisfies
\[
 (-\Delta)^s\bar h=b(x)
 =c_{n,s}\int_{\Omega^c}
       \frac{(1-\chi(y))g(y)}{|x-y|^{n+2s}}\dd y
 \quad\text{in }\Omega.
\]
The source $b$ is bounded by separation from $\Omega$ and the weighted
tail of $g$. The same separation and $\alpha_g<s$ show
$\bar h\in H^s_\Omega$, since
$\int_{\Omega^c}|g(y)|^2(1+|y|^{n+2s})^{-1}\dd y<\infty$.

Let $\tau$ be the bounded zero-exterior solution of
$(-\Delta)^s\tau=1$ in $\Omega$; its boundedness follows from
\eqref{prelim-linear-Lp} and its stated boundedness alternatives,
also for $n\le2s$. Comparison with
$\pm(\|\chi g\|_{L^\infty(\Omega^c)}+\|b\|_\infty\tau)$ gives
\[
 |\bar h|\le\|\chi g\|_{L^\infty(\Omega^c)}
                    +\|b\|_\infty\tau\quad\text{in }\Omega.
\]
Thus $h=\bar h$ is bounded in $\Omega$. Its interior boundedness
and the growth of $g$ also give $h\in L_s^1(\R^n)$.
\end{proof}

\begin{lemma}[Linear boundedness]\label{we-lem:linear-boundedness}
The following qualitative assertions hold.

\textup{(i)} Let $h\in H^s_\Omega\cap L^\infty(\Omega)$ be
$s$-harmonic in $\Omega$. Suppose $U-h\in H^s_{\Omega,0}$ and $U$
is an $L^1$-weak solution of $LU=G$ in $\Omega$, where
$G\in L^2(\Omega)$ and $|G|\le C(1+|U|)$ almost everywhere.
Then $U\in L^\infty(\Omega)$, with a bound that may depend on
$C$, $h$, and $\|U\|_{L^2(\Omega)}$.

\textup{(ii)} Let $w\in L_s^1(\R^n)\cap L^2_{\mathrm{loc}}(\Omega)$
satisfy $Lw=a(x)w$ distributionally in $\Omega$, with
$a\in L^\infty_{\mathrm{loc}}(\Omega)$. Then
$w\in L^\infty_{\mathrm{loc}}(\Omega)$.
\end{lemma}
\begin{proof}
For \textup{(i)}, set $z=U-h$. Uniqueness for the linear
$L^1$-weak Dirichlet problem \cite[Proposition~C.1]{SPapprox}
identifies $z$ with the variational zero-exterior solution with source
$G$. Since $h$ is bounded in $\Omega$,
$|G|\le C_h(1+|z|)$.

For \textup{(ii)}, choose nested balls compactly contained in
$\Omega$. On a current ball $B$, suppose $w\in L^p(B)$, initially
with $p=2$. Let $z_B$ be the zero-exterior Green solution in $B$
with source $aw\in L^p(B)$. The remainder $h_B=w-z_B$ belongs to
$L_s^1(\R^n)$ and is distributionally $s$-harmonic in $B$.
Lemma~\ref{harmonic-weyl} makes it bounded on a smaller ball
$B'\Subset B$.

Both assertions now use the same increase in integrability. The
source inherits each available $L^p$ bound from the unknown:
from $|G|\le C_h(1+|z|)$ in \textup{(i)}, and from the boundedness
of $a$ on the current ball in \textup{(ii)}. If
$1<p<n/(2s)$, \eqref{prelim-linear-Lp} gives the next exponent
\[
 \frac1q=\frac1p-\frac{2s}{n}.
\]
This is an $L^q(\Omega)$ bound for $z$ in \textup{(i)} and an
$L^q(B')$ bound for $w=z_B+h_B$ in \textup{(ii)}.
If $p>n/(2s)$, the boundedness alternative gives the desired bound.
At $p=n/(2s)$, use a slightly smaller exponent to reach a finite
$q>n/(2s)$, and apply that alternative at the next step.

For $n>2s$, starting at $p_0=2$, the subcritical exponents satisfy
\[
 \frac1{p_{j+1}}=\frac1{p_j}-\frac{2s}{n}.
\]
Thus the threshold is reached or crossed in finitely many steps.
For \textup{(ii)}, take enough nested balls for these steps and
finish with a finite cover of any compact subset. When $n\le2s$,
the first boundedness alternative already applies with $p=2$.
\end{proof}

We next use stability and convexity to obtain a linear equation
for the difference of two ordered solutions.

\begin{lemma}\label{we-lem:ordered-boundedness}
Let $0<s<1$, let $\Omega\subset\R^n$ be smooth and bounded, and
let $g$ satisfy \eqref{we-eq:g}. Suppose that $f:\R\to[0,\infty)$
is convex and belongs to $C^1_{\mathrm{loc}}(\R)$, and that
$u,v\in H^s_\Omega\cap L_s^1(\R^n)$
are $L^1$-weak solutions of $(-\Delta)^sz=f(z)$ in $\Omega$,
with $z=g$ on $\Omega^c$, for $z=u,v$.
Assume that $u$ is stable and that
\[
 m\le v\le u\quad\text{a.e. in }\Omega
\]
for some $m\in\R$. Then, with $w=u-v$,
\[
 \mathcal E_\Omega(w)
 =\int_\Omega(f(u)-f(v))w
 =\int_\Omega f'(u)w^2,
\]
and $f$ is affine on $[v(x),u(x)]$ for almost every $x$ with
$u(x)>v(x)$. In particular,
\begin{equation}\label{we-eq:linearized-difference}
 (-\Delta)^s w=f'(v)w\quad\text{in }\mathcal D'(\Omega),
 \qquad w=0\quad\text{on }\Omega^c.
\end{equation}
Consequently, local boundedness of $v$ in $\Omega$
implies local boundedness of $u$, and boundedness of $v$ on all of
$\Omega$ implies boundedness of $u$ on all of $\Omega$.
If $f$ is strictly convex, then $u=v$.
\end{lemma}
\begin{proof}
Step~1. Set $w=u-v\ge0$. Then $w\in H^s_{\Omega,0}$. Since $f'(u)\ge f'(m)$,
the negative part of $f'(u)$ is bounded. Approximation of
$\xi\in H^s_{\Omega,0}$ by smooth compactly supported functions and Fatou's
lemma extend \eqref{we-eq:stability} to $H^s_{\Omega,0}$; in particular,
\begin{equation}\label{we-eq:stabilityextended}
 f'(u)w^2\in L^1(\Omega),\qquad
 \int_\Omega f'(u)w^2\le\mathcal E_\Omega(w).
\end{equation}
This density argument is also recorded in
\cite[Remark~2.1]{SPapprox}.

Convexity also gives
\begin{equation}\label{we-eq:integrability}
 f'(m)w^2
 \le (f(u)-f(v))w
 \le f'(u)w^2.
\end{equation}
The middle expression is therefore integrable.

Step~2. To test the difference equation with $w$, put
\[
 G=f(u)-f(v),\qquad \lambda=\max\{0,-f'(m)\},\qquad
 \mu=G+\lambda w\ge0.
\]
Then $\mu\in L^1(\Omega,d_\Omega^s\dd x)$
and $\mu w\in L^1(\Omega)$ by \eqref{we-eq:integrability}.
Let $\mu_j=\min\{\mu,j\}$ and solve, variationally,
\[
 ((-\Delta)^s+\lambda)z_j=\mu_j\quad\text{in }\Omega,
 \qquad z_j=0\quad\text{on }\Omega^c.
\]
The $L^1$-weak comparison principle
\cite[Proposition~C.2]{SPapprox}, with the nonnegative zeroth-order
coefficient $\lambda$, gives $0\le z_j\le z_{j+1}\le w$. Testing with $z_j$ gives
\[
 \mathcal E_\Omega(z_j)+\lambda\|z_j\|_2^2
   =\int_\Omega\mu_jz_j\le\int_\Omega\mu w.
\]
Thus $z_j$ is bounded in $H^s_{\Omega,0}$, converges in $L^2$ to its
monotone limit, and converges weakly in $H^s_{\Omega,0}$.
Passing to the $L^1$-weak equation, with the weighted
dominated convergence for $\mu_j$, identifies this limit as
$w$ by uniqueness for $(-\Delta)^s+\lambda$: apply the same
comparison principle to both signs of the difference of two solutions
with the same source and zero exterior data.

Now test the bounded-source equation for $z_j$ with the
fixed energy function $w$. Weak convergence on the left and
monotone convergence on the right give
\[
 \mathcal E_\Omega(w)+\lambda\|w\|_2^2=\int_\Omega\mu w,
 \quad\text{hence}\quad \mathcal E_\Omega(w)=\int_\Omega G w.
\]

Step~3. Consequently,
\begin{equation}\label{we-eq:equality}
 \mathcal E_\Omega(w)=\int_\Omega(f(u)-f(v))w
       \le\int_\Omega f'(u)w^2\le\mathcal E_\Omega(w).
\end{equation}
Both inequalities must therefore be equalities. At almost every point where
$w(x)>0$,
\[
 f(u(x))-f(v(x))=f'(u(x))(u(x)-v(x)).
\]
For a differentiable convex function, this equality implies that
$f$ is affine on $[v(x),u(x)]$: the nonnegative continuous
function $f'(u(x))-f'(t)$ has integral zero on that interval.

Step~4. On every affine interval found above, continuity of $f'$ gives
$f'(u)=f'(v)$. Where $w=0$, both source differences vanish.
Thus $f(u)-f(v)=f'(v)w$ almost everywhere. Subtracting the original
transposition identities cancels the common exterior term and gives
\[
 \int_\Omega w(-\Delta)^s\varphi
       =\int_\Omega f'(v)w\varphi
\]
for every Dirichlet test in \eqref{we-eq:transposition}. In particular it proves
\eqref{we-eq:linearized-difference}.
If $v$ is locally bounded, the coefficient $f'(v)$ is locally
bounded, and Lemma~\ref{we-lem:linear-boundedness}\textup{(ii)} gives local boundedness
of $w$, hence of $u=v+w$.
If $v$ is bounded on all of $\Omega$, then $f'(v)$ is bounded
there and $f'(v)w\in L^2(\Omega)$. Apply
Lemma~\ref{we-lem:linear-boundedness}\textup{(i)} to $U=w$ and $h=0$ to obtain
global boundedness of $w$ and then of $u$.

Finally, strict convexity excludes a nontrivial affine interval,
and hence gives $w=0$ almost everywhere.
\end{proof}

\subsection{Approximation and interior regularity}\label{weak-approximation}

We combine bounded stable approximations with the preceding comparison lemma
to extend the interior estimate to finite-energy weak solutions.

\begin{theorem}[Interior regularity of finite-energy weak solutions]\label{we-thm:upgrade}
Let $\Omega\subset\R^n$ be a smooth bounded domain, let $0<s<1$,
and assume either $n=1$ or $n\ge2$ and
$A_{n,s}>H_{n,s}$.
Let $f:\R\to[0,\infty)$ be convex and $C^1_{\mathrm{loc}}(\R)$,
let $g$ satisfy
\eqref{we-eq:g}, and let $u\in H^s_\Omega\cap L_s^1(\R^n)$ be a stable
$L^1$-weak solution of
\[
 (-\Delta)^su=f(u)\quad\text{in }\Omega,\qquad u=g\quad\text{on }\Omega^c.
\]
Then $u$ is locally bounded and has a $C^\alpha_{\mathrm{loc}}(\Omega)$
representative. For every $B_{2\rho}(a)\subset\Omega$, setting
$u_{a,\rho}(x)=u(a+\rho x)$ gives
\begin{equation}\label{we-eq:scaled}
 \|u_{a,\rho}\|_{C^\alpha(B_1)}
       +[u_{a,\rho}]_{H^s(B_1)}
 \le C \|u_{a,\rho}\|_{L_s^1(\R^n)},
\end{equation}
with the constants $\alpha,C$ from Theorem~\ref{full}, depending
only on $n,s$.
Consequently, for every $\Omega'\Subset\Omega$,
\begin{equation}\label{we-eq:interior}
 \|u\|_{C^\alpha(\Omega')}+[u]_{H^s(\Omega')}
 \le C_{\Omega,\Omega',n,s}\|u\|_{L_s^1(\R^n)}.
\end{equation}
\end{theorem}
\begin{proof}
Step~1. Apply Lemma~\ref{we-lem:harmonic-lift} with $w_0=u$ to obtain
$h\in H^s_\Omega\cap L^\infty(\Omega)\cap L_s^1(\R^n)$ satisfying
\[
 (-\Delta)^sh=0\quad\text{in }\Omega,\qquad h=g\quad\text{on }\Omega^c.
\]
Since $f(u)\ge0$, the comparison principle
\cite[Proposition~C.2]{SPapprox} yields $h\le u$ in $\Omega$.
Put $m=\operatorname*{ess\,inf}_\Omega h>-\infty$.

Step~2. Choose integers $k\to\infty$ with
$k>\max\{0,f'(m)\}$ and cap the derivative above $m$:
\begin{equation}\label{we-eq:trunc}
 f_k(t)=
 \begin{cases}
 f(t),&t\le m,\\
 f(m)+\displaystyle\int_m^t\min\{f'(r),k\}\dd r,&t>m.
 \end{cases}
\end{equation}
The derivative agrees with $f'$ below $m$ and equals
$\min\{f',k\}$ above $m$. It is continuous and nondecreasing,
including at $m$ because $k>f'(m)$. Thus $f_k$ is convex and
$C^1_{\mathrm{loc}}(\R)$, and it is Lipschitz on $[m,\infty)$, with
\[
 f'(m)\le f_k'(t)\le k\quad(t\ge m),\qquad
 f_k\le f_{k+1}\le f,\qquad f_k'\le f'.
\]
The function order follows by integrating the ordered derivative
caps from $m$; below $m$ the functions agree. It is also nonnegative.
Until $f'$ reaches the positive value $k$,
the function $f_k$ equals $f$; after a first such point $a$,
$f_k(t)=f(a)+k(t-a)\ge0$ for $t\ge a$.
If $f'$ never reaches $k$, then $f_k=f$. This construction therefore
also covers the case in which $f'\le0$ on $[m,\infty)$.

Use $h$ as a subsolution and $u$ as a supersolution for $f_k$.
The iteration in \cite[Proposition~C.7 and Remark~C.8]{SPapprox}
gives the solution $u_k$ minimal above $h$, with
\begin{equation}\label{we-eq:order}
 h\le u_k\le u_{k+1}\le u\quad\text{in }\Omega,\qquad
 u_k=g\quad\text{on }\Omega^c.
\end{equation}
For the fixed shift $\ell=\max\{0,-f'(m)\}$, the function
$f_k(t)+\ell t$ is nondecreasing on $[m,\infty)$ by the derivative
bound above. Thus the same monotone iteration applies with operator
$(-\Delta)^s+\ell$. The source terms are integrable because $f_k$
has at most linear growth on that interval.

The bounds in \eqref{we-eq:order} give $u_k\in L^2(\Omega)$ and
$f_k(u_k)\in L^2(\Omega)$. The variational solution with this source
and exterior datum lies in $h+H^s_{\Omega,0}$ and agrees with $u_k$
by uniqueness of the linear $L^1$-weak problem
\cite[Proposition~C.1]{SPapprox}. Hence $u_k\in H^s_\Omega$, and
Lemma~\ref{we-lem:linear-boundedness}\textup{(i)} gives
$u_k\in L^\infty(\Omega)$.
Finally,
\[
 f_k'(u_k)\le f'(u_k)\le f'(u)
\]
and the original stability inequality show that every $u_k$ is
stable for its own equation; its derivative is bounded on $[m,\infty)$.
This follows the monotone approximation method in the proof of
\cite[Theorem~1.9]{SPapprox}, using derivative caps and the harmonic
lift in place of a constant subsolution.

Step~3. By monotonicity let $u_k\uparrow v$ a.e.\ in $\Omega$ and set
$v=g$ outside. Dominated convergence gives $u_k\to v$ in
$L^2(\Omega)$, and $h\le v\le u$. Test the equation for $u_k$
with $u-u_k\in H^s_{\Omega,0}$. This is legitimate because
$f_k(u_k)\in L^2(\Omega)$. Nonnegativity of $f_k$ gives
\begin{equation}\label{we-eq:energybound}
 \mathcal E_\Omega(u_k)
 =\mathcal E_\Omega(u_k,u)-\int_\Omega f_k(u_k)(u-u_k)
 \le\mathcal E_\Omega(u_k,u)\le\sqrt{\mathcal E_\Omega(u_k)\mathcal E_\Omega(u)}.
\end{equation}
Thus $u_k$ is bounded in $H^s_\Omega$ and its weak limit is $v$.
In particular $v\in H^s_\Omega$ with the prescribed exterior datum.

To pass to the limit in the equation, we use the weighted
integrability required by the weak formulation. By convexity,
for $m\le u_k(x)\le u(x)$,
\begin{equation}\label{we-eq:sourcebound}
 0\le f_k(u_k(x))\le f(u_k(x))
       \le \max\{f(m),f(u(x))\}\le f(m)+f(u(x)).
\end{equation}
For almost every $x$, $u(x)$ is finite and $f'$ is bounded on
$[m,u(x)]$. For all sufficiently large $k$, the cap in
\eqref{we-eq:trunc} is therefore inactive on that whole interval.
Since $u_k(x)\in[m,u(x)]$ and $u_k(x)\to v(x)$, it follows that
$f_k(u_k(x))=f(u_k(x))\to f(v(x))$.
Dominated convergence in $L^1(\Omega,d_\Omega^s\dd x)$ permits
passage to the limit in \eqref{we-eq:transposition}. Therefore $v$
solves the original equation.

It is stable, since
\[
 f'(m)\le f'(v)\le f'(u).
\]

Step~4. To apply the interior estimate uniformly, observe that the pointwise bound
\[
 |u_k|\le |u|+|h|\quad\text{on }\R^n
\]
gives a uniform bound for the weighted $L^1$ norm on every fixed
rescaled interior ball.
Apply Theorem~\ref{full} to $u_k(a+\rho\,\cdot)$, whose
nonlinearity is $\rho^{2s}f_k$. The resulting uniform bounds and
the a.e.\ convergence show that $v\in L^\infty_{\mathrm{loc}}(\Omega)$
(indeed $v\in C^\alpha_{\mathrm{loc}}(\Omega)$).
The auxiliary bound involving $\|h\|_{L_s^1(\R^n)}$ is needed only for this
qualitative boundedness conclusion.

Step~5. The solutions $v\le u$ have the same exterior datum and satisfy
$m\le v\le u$ in $\Omega$. Since $v$ is locally bounded,
Lemma~\ref{we-lem:ordered-boundedness} gives local boundedness of $u$.

Apply Theorem~\ref{full} directly to $u(a+\rho\,\cdot)$ to obtain
the asserted rescaled estimate; its right-hand side involves only
the original rescaled solution.
A finite interior cover gives the corresponding estimate on each
compact subset. For the local
energy, close pairs are controlled on the covering balls and
pairs with separation bounded away from zero are controlled
by the amplitude estimate.
\end{proof}

The same argument transfers all exponents in
Theorem~\ref{explicit-exponents} to this weak class.

\subsection{The logarithmic profile and optimality}\label{weak-optimality}

\begin{proposition}[The logarithmic singular solution]\label{we-prop:log}
Let $n>2s$ and $\Omega=B_1$. Set
\[
 u_*(x)=-2s\log|x|\quad(x\ne0),\qquad f(t)=A_{n,s}e^t,
\]
where $A_{n,s}$ and $H_{n,s}$ are the constants from the introduction
and the value at $0$ is immaterial. Then
\[
 u_*\in L_s^1(\R^n)\cap H^s_{B_1},\qquad
 (-\Delta)^su_*=A_{n,s}|x|^{-2s}
 \quad\text{in }\mathcal D'(\R^n).
\]
The exterior datum $g=u_*|_{B_1^c}$ satisfies \eqref{we-eq:g},
and $u_*$ is an $L^1$-weak solution in the finite-energy class
defined above for every $n>2s$. Its source belongs to $L^2(B_1)$
exactly when $n>4s$, in which case it also satisfies the narrower
energy-solution definition in \cite[Definition~1.2]{SPapprox}.
It is stable in $B_1$ exactly when $A_{n,s}\le H_{n,s}$,
including equality, and it is unbounded near the origin.
\end{proposition}

\begin{proof}
Direct integration of the logarithmic profile gives $u_*\in L^2(B_1)$
and $u_*\in L_s^1(\R^n)$. For a fixed unit vector $e$, the number
\[
 C_{n,s}^{\log}=\int_{\R^n}
       \frac{(\log|z|)^2}{|e-z|^{n+2s}}\dd z
\]
is finite. Near $0$ logarithms are integrable; near $e$ one has
$|\log|z||\le C|z-e|$, so $s<1$ suffices; at infinity the
integrand is controlled by $|z|^{-n-2s}\log^2|z|$.
Scaling and rotation yield the exact identity
\begin{equation}\label{we-eq:logenergy}
 \int_{\R^n}\frac{|u_*(x)-u_*(y)|^2}{|x-y|^{n+2s}}\dd y
 =4s^2 C_{n,s}^{\log}|x|^{-2s}\qquad(x\ne0).
\end{equation}
Integration over $x\in B_1$ is finite since $n>2s$.
The unnormalized double integral over $Q_{B_1}$ is at most
twice this integral, so the normalized energy is finite as well.
This proves finite energy in the stated Dirichlet space, including
the interactions with the unbounded logarithmic exterior datum.

It remains to verify the equation across the origin. We use the distributional
homogeneous-power identity \eqref{power-multiplier}, for $0<\varepsilon<n-2s$,
\begin{align}
 (-\Delta)^s|x|^{-\varepsilon}
   &=c_0(\varepsilon)|x|^{-\varepsilon-2s},\label{we-eq:power}\\
 c_0(\varepsilon)
   &=4^s\frac{\Gamma((n-\varepsilon)/2)
                    \Gamma((\varepsilon+2s)/2)}
                   {\Gamma(\varepsilon/2)
                    \Gamma((n-\varepsilon-2s)/2)}.\nonumber
\end{align}
Subtract $(-\Delta)^s1=0$ and multiply by $2s/\varepsilon$.
For any fixed $0<\varepsilon_0<n-2s$, the difference quotient
is dominated near $0$ by
$C|x|^{-\varepsilon_0}|\log|x||$, and for $|x|\ge1$ by
$2s\log|x|$. Thus
\[
 \frac{2s(|x|^{-\varepsilon}-1)}{\varepsilon}
       \longrightarrow u_*\quad\text{in }L_s^1(\R^n).
\]

Meanwhile $2s c_0(\varepsilon)/\varepsilon\to A_{n,s}$ and
$|x|^{-\varepsilon-2s}\to|x|^{-2s}$ in $L^1_{\mathrm{loc}}$.
Testing against $C_c^\infty(\R^n)$ proves the claimed
distributional identity. In particular, no additional point mass is present at the origin.

For $x\in B_1^c$ and $z\in\partial B_1$, and any $0<a<1$,
\[
 |g(x)-g(z)|=2s\log|x|
 \le\frac{2s}{a}(|x|-1)^a
 \le\frac{2s}{a}|x-z|^a.
\]
Here $\log(1+t)\le t^a/a$ follows by integrating
$(1+t)^{-1}\le t^{a-1}$. Choose
$a\in(\max\{0,2s-1\},s)$ to obtain \eqref{we-eq:g}.

Put $F(x)=A_{n,s}|x|^{-2s}$ and $V=H^s_{B_1,0}$.
For $\psi\in C_c^\infty(B_1)$, Cauchy--Schwarz and the Hardy
inequality \eqref{prelim-hardy-inequality} give
\begin{equation}\label{we-eq:logduality}
 \begin{split}
 \left|\int_{B_1}F\psi\right|
 &\le A_{n,s}\left(\int_{B_1}|x|^{-2s}\dd x\right)^{1/2}
              \left(\int_{B_1}|x|^{-2s}\psi^2\dd x\right)^{1/2}\\
 &\le C_{n,s}\mathcal E_{B_1}(\psi)^{1/2}.
 \end{split}
\end{equation}
The first integral is finite because $n>2s$.
Both sides of the distributional identity are therefore continuous
on $V$, and density yields
\begin{equation}\label{we-eq:logvariational}
 \mathcal E_{B_1}(u_*,\psi)=\int_{B_1}F\psi
 \qquad(\psi\in V).
\end{equation}

We pass from this identity to the exact transposition formulation
with the logarithmic exterior datum. Let $h$ be its harmonic lift
from Lemma~\ref{we-lem:harmonic-lift}, put $F_j=\min\{F,j\}$, and let
$U_j-h\in V$ solve $LU_j=F_j$ with exterior value $g$.
These bounded-source variational solutions also satisfy
\eqref{we-eq:transposition}, by the linear Dirichlet theory.
The same inequality gives, for $\psi\in V$,
\[
 \begin{split}
 \left|\int_{B_1}(F-F_j)\psi\right|
 &\le \varepsilon_j\mathcal E_{B_1}(\psi)^{1/2},\\
 \varepsilon_j&=
 \frac{A_{n,s}}{\sqrt{H_{n,s}}}
 \left(\int_{\{x\in B_1:F(x)>j\}}|x|^{-2s}\dd x\right)^{1/2}
 \longrightarrow0.
 \end{split}
\]
Subtract the variational equations and test with
$u_*-U_j\in V$. Then
$\mathcal E_{B_1}(u_*-U_j)^{1/2}\le\varepsilon_j$,
and fractional Poincar\'e gives $U_j\to u_*$ in $L^2(B_1)$.

For each admissible transposition test $\varphi$, we may now pass to
the limit in the identity for $U_j$: its interior term converges
because $L\varphi$ is bounded, and its source term because
$F_j\to F$ in $L^1(B_1)$ and $\varphi$ is bounded.

The exterior term is unchanged. To check its absolute integrability,
use $|\varphi(x)|\le C d_{B_1}(x)^s$ and
$d_{B_1}(x)\le|x-y|$ for $x\in B_1$, $y\in B_1^c$.
The formula for $\mathcal N_s\varphi$ then gives
\begin{align*}
 |\mathcal N_s\varphi(y)|
 &\le C\int_{B_1}\frac{d_{B_1}(x)^s}{|x-y|^{n+2s}}\dd x\\
 &\le C\int_{|z|\ge d_{B_1}(y)}|z|^{-n-s}\dd z
 \le C d_{B_1}(y)^{-s}.
\end{align*}
For $|y|\ge2$, the same kernel is bounded by $C|y|^{-n-2s}$.
Since $g$ is bounded near the sphere and grows logarithmically at infinity,
\begin{align*}
 \int_{1<|y|<2}|g\,\mathcal N_s\varphi|\dd y
 &\le C\int_0^1t^{-s}\dd t<\infty,\\
 \int_{|y|\ge2}|g\,\mathcal N_s\varphi|\dd y
 &\le C\int_2^\infty (\log r)r^{-1-2s}\dd r<\infty.
\end{align*}
Thus
\[
 \int_{B_1}u_*L\varphi+\int_{B_1^c}g\,\mathcal N_s\varphi
       =\int_{B_1}F\varphi.
\]
This proves the $L^1$-weak formulation for every $n>2s$.

Finally $f'(u_*)=A_{n,s}|x|^{-2s}$.
The sharp fractional Hardy inequality \eqref{prelim-hardy-inequality}
proves stability for
$A_{n,s}\le H_{n,s}$. Its optimal constant is still $H_{n,s}$
when the competitors are restricted to $C_c^\infty(B_1)$:
any compactly supported competitor in $\R^n$ can be dilated
into $B_1$, and the quotient in this inequality is invariant under dilation.
Thus stability fails when $A_{n,s}>H_{n,s}$.
\end{proof}

For every $n\ge2$ and $0<s<1$, one has $n>2s$.
Hence, whenever $A_{n,s}\le H_{n,s}$,
Proposition~\ref{we-prop:log} gives an unbounded stable solution
in exactly the finite-energy weak class of
Theorem~\ref{we-thm:upgrade}. Therefore $A_{n,s}>H_{n,s}$ is the
optimal condition for local boundedness and interior H\"older regularity
in this class: both fail whenever $A_{n,s}\le H_{n,s}$.

% MODULE applications

\section{Global estimates and a Liouville theorem}\label{applications-section}

We first consider the Dirichlet problem
\begin{equation}\label{zero-exterior-equation}
 (-\Delta)^su=f(u)\quad\hbox{in }\Omega,\qquad
 u=0\quad\hbox{in }\R^n\setminus\Omega,
\end{equation}
on a smooth bounded domain $\Omega$. We then use the interior estimate
at large scales to obtain a Liouville theorem for entire solutions.
Here $f:\R\to[0,\infty)$ is convex and belongs to
$C^1_{\mathrm{loc}}(\R)$; the constants remain independent of the nonlinearity.
The bounded-solution, extremal, and Liouville conclusions use
Theorem~\ref{full}; the finite-energy corollary also uses
Theorem~\ref{we-thm:upgrade}.

We first specify the notions of Dirichlet solution used below. For zero
exterior data, a classical solution is a function
$u\in C(\overline\Omega)$, extended by zero to $\Omega^c$, whose
fractional Laplacian is defined by the principal-value integral and
equals $f(u)$ at every point of $\Omega$; see
\cite[Definition~1.4]{SPapprox}.

The $L^1$-weak formulation and the finite-energy spaces are those
of Section~\ref{weak-section}, specialized here to $g=0$.

\begin{corollary}[A global bound in convex domains]\label{convex-global}
Let $0<s<1$, let $n$ satisfy \eqref{dimension-range}, and let
$\Omega\subset\R^n$ be a smooth bounded convex domain. Suppose that
$f:\R\to[0,\infty)$ is convex and belongs to $C^1_{\mathrm{loc}}(\R)$.
Every bounded stable solution of
\eqref{zero-exterior-equation}
understood in either of the preceding Dirichlet senses, satisfies
\begin{equation}\label{convex-global-bound}
 \|u\|_{L^\infty(\Omega)}
 \le C_{\Omega,n,s}\|u\|_{L^1(\Omega)}.
\end{equation}
\end{corollary}

\begin{proof}[Proof of Corollary~\ref{convex-global}]
For bounded solutions these two notions of solution agree; see
\cite[Remark 1.5]{SPapprox}. The maximum principle gives $u\ge0$, and
the strong maximum principle gives either $u\equiv0$ or $u>0$ in
$\Omega$. The first case is immediate. In the second case, the boundary
estimate of Ros-Oton and Serra \cite[Proposition 1.8]{RSext} supplies
$\delta>0$ and $C$ depending only on the domain (allowing also dependence
on $n,s$ here), such that
\[
 \|u\|_{L^\infty(\{x\in\Omega:
                  \operatorname{dist}(x,\partial\Omega)<\delta\})}
 \le C\|u\|_{L^1(\Omega)}.
\]
This boundary estimate requires only a locally Lipschitz nonlinearity
and a positive bounded solution.

It therefore remains to bound the
solution on a compact subset of the domain. Choose
$\rho>0$ depending only on $\Omega$ with
$2\rho<\delta$ and cover that set by balls $B_\rho(a)$ with
$B_{2\rho}(a)\subset\Omega$. The rescaled function
$v(x)=u(a+\rho x)$ solves the stable equation
$(-\Delta)^sv=\rho^{2s}f(v)$ in $B_2$. Since $u$ vanishes outside
$\Omega$,
\[
 \|v\|_{L_s^1(\R^n)}=\rho^{-n}\int_\Omega
 \frac{|u(y)|}{1+|(y-a)/\rho|^{n+2s}}\dd y
 \le\rho^{-n}\|u\|_{L^1(\Omega)}.
\]
Applying Theorem~\ref{full} on each ball bounds $u$ on $B_\rho(a)$
by the same right-hand side. Combining the interior and boundary estimates proves the result.
\end{proof}

\begin{corollary}[Finite-energy solutions with zero exterior data]\label{we-cor:zero}
Under the assumptions of Theorem~\ref{we-thm:upgrade}, suppose that
$g=0$. Then $u\ge0$ and, for every $\Omega'\Subset\Omega$,
\[
 \|u\|_{C^\alpha(\Omega')}+[u]_{H^s(\Omega')}
 \le C_{\Omega,\Omega',n,s}\|u\|_{L^1(\Omega)}.
\]
If $\Omega$ is convex, then
\[
 \|u\|_{L^\infty(\Omega)}
 \le C_{\Omega,n,s}\|u\|_{L^1(\Omega)}.
\]
\end{corollary}
\begin{proof}
The maximum principle gives $u\ge0$, and the interior estimate
follows from Theorem~\ref{we-thm:upgrade} since
$\|u\|_{L_s^1(\R^n)}\le\|u\|_{L^1(\Omega)}$.

For a convex domain, apply Corollary~\ref{convex-global} to the
bounded approximants constructed in the proof of
Theorem~\ref{we-thm:upgrade}. Here $h=0$, so $0\le u_k\le u$ and
\[
 \|u_k\|_{L^\infty(\Omega)}
 \le C_{\Omega,n,s}\|u\|_{L^1(\Omega)}.
\]
Their limit $v$ is therefore bounded on the whole domain.
Lemma~\ref{we-lem:ordered-boundedness}, applied with $m=0$, gives
$u\in L^\infty(\Omega)$. Once boundedness is established,
Corollary~\ref{convex-global}, applied directly to $u$, gives the
asserted estimate.
\end{proof}

\begin{corollary}[Extremal solutions]\label{extremal}
Let $\Omega$ and $n,s$ be as in Corollary~\ref{convex-global}. Let
$f:\R\to[0,\infty)$ be convex and $C^1_{\mathrm{loc}}(\R)$,
and nondecreasing on $[0,\infty)$. Assume
\[
 f(0)>0,\qquad \lim_{t\to\infty}\frac{f(t)}t=\infty.
\]
Let $u_\lambda$ be the minimal solutions of
\[
 (-\Delta)^su_\lambda=\lambda f(u_\lambda)\quad\hbox{in }\Omega,
 \qquad u_\lambda=0\quad\hbox{in }\R^n\setminus\Omega,
 \qquad 0<\lambda<\lambda^*,
\]
and let $u^*=\lim_{\lambda\uparrow\lambda^*}u_\lambda$ be their extremal
limit. Then
\[
 \|u^*\|_{L^\infty(\Omega)}
 \le C_{\Omega,n,s}\|u^*\|_{L^1(\Omega)}<\infty.
\]
In particular, for $s=1/2$ this conclusion holds through dimension eight.
\end{corollary}

\begin{proof}[Proof of Corollary~\ref{extremal}]
Existence of the minimal branch, its stability, and the facts that
$0\le u_\lambda\le u^*\in L^1(\Omega)$ follow from
\cite[Proposition 1.2]{RSext}.

Apply Corollary~\ref{convex-global} to $u_\lambda$ with
nonlinearity $\lambda f$. Its constant is independent of $\lambda$, so
\[
 \|u_\lambda\|_{L^\infty(\Omega)}
 \le C_{\Omega,n,s}\|u_\lambda\|_{L^1(\Omega)}
 \le C_{\Omega,n,s}\|u^*\|_{L^1(\Omega)}.
\]
Passing to the monotone limit gives the assertion.
\end{proof}

At large scales, the interior estimate gives a Liouville theorem
under a one-sided growth condition.

\begin{corollary}\label{liouville}
Let $n,s$ satisfy \eqref{dimension-range}, let $\alpha$ be any
exponent available in Theorem~\ref{full} or Theorem~\ref{explicit-exponents},
and let $u\in L^\infty_{\rm loc}(\R^n)\cap L_s^1(\R^n)$ be an entire stable
distributional solution of $(-\Delta)^su=f(u)$, with the nonlinearity
assumptions of Theorem~\ref{full}.
If
\begin{equation}\label{liouville-growth}
 \|u_-(R\,\cdot)\|_{L_s^1(\R^n)}=o(R^\alpha)\qquad(R\to\infty),
\end{equation}
then $u$ is constant. In particular, $u$ is constant under either
one-sided lower bound
\begin{equation}\label{liouville-lower-log}
 u(x)\ge-C_0\log(2+|x|)\qquad(x\in\R^n)
\end{equation}
or
\[
 u(x)\ge-C_1(1+|x|)^p\qquad(x\in\R^n),\qquad 0\le p<\alpha,
\]
where $C_0,C_1\ge0$.
\end{corollary}
\begin{proof}
We first use the nonnegativity of the source to control the weighted
$L^1$ norm of the solution by that of its negative part. Use the centered ball Poisson kernel
$\Pi_a(0,y)$ from \eqref{prelim-fractional-poisson}.
Since $(-\Delta)^su\ge0$, comparison with the $s$-harmonic replacement
in $B_a$ gives
\[
 u(0)\ge\int_{\R^n}\Pi_a(0,y)u(y)\dd y.
\]
Write $h_a$ for this replacement. It is well defined because $u$ is locally
continuous and has finite weighted $L^1$ norm. The function $u-h_a$
is continuous, zero outside the ball, and a supersolution, so the
maximum principle gives $u-h_a\ge0$.

To remove the boundary singularity of the Poisson kernel, average
these inequalities using the probability density $na^{n-1}/R^n$ on $(0,R)$ and set
\[
 \overline\Pi_R(y)=\frac n{R^n}\int_0^R a^{n-1}\Pi_a(0,y)\dd a.
\]
The averaged kernel is positive and has mass one. If $r=|y|\le R$, the change
$a=rb$ gives exactly
\[
 \overline\Pi_R(y)=\frac{np_{n,s}}{R^n}
       \int_0^1 b^{n+2s-1}(1-b^2)^{-s}\dd b.
\]

For $r\ge2R$, direct bounds in the same integral give
$\overline\Pi_R(y)\asymp R^{2s}r^{-n-2s}$. On $R<r<2R$, its upper bound follows
by extending the integral to $a=r$, and its lower bound by restricting
to $R/4<a<R/2$. Consequently
\begin{equation}\label{poisson-average-tail}
 \overline\Pi_R(y)\asymp
 \frac{R^{-n}}{1+(|y|/R)^{n+2s}},
 \qquad \int_{\R^n}u(y)\overline\Pi_R(y)\dd y\le u(0),
\end{equation}
with constants depending only on $n,s$. The comparison of kernels gives absolute integrability and justifies
averaging the Poisson inequalities in the second assertion.

Writing $u_R(x)=u(Rx)$ and using $|u|=u+2u_-$, we obtain
\[
 \begin{split}
 \|u_R\|_{L_s^1(\R^n)}&\le C\int_{\R^n}|u(y)|\overline\Pi_R(y)\dd y\\
 &\le C|u(0)|+2C\int_{\R^n}u_-(y)\overline\Pi_R(y)\dd y\\
 &\le C\bigl(|u(0)|+\|u_-(R\,\cdot)\|_{L_s^1(\R^n)}\bigr).
 \end{split}
\]
Thus \eqref{liouville-growth} implies $\|u_R\|_{L_s^1(\R^n)}=o(R^\alpha)$.

The equation for $u_R$ has nonlinearity $R^{2s}f$ and remains stable.
The scaled estimate \eqref{scaled-interior-bound} gives, for fixed
$x,y$ and sufficiently large $R$,
\[
 |u(x)-u(y)|\le C|x-y|^\alpha R^{-\alpha}\|u_R\|_{L_s^1(\R^n)}
 \longrightarrow0.
\]
Hence $u$ is constant.

Finally, \eqref{liouville-lower-log} gives
$\|u_-(R\,\cdot)\|_{L_s^1(\R^n)}=O(1+\log R)$ for $R\ge1$, because the weighted $L^1$ measure has a finite
logarithmic moment.

The polynomial lower bound
gives $\|u_-(R\,\cdot)\|_{L_s^1(\R^n)}=O(R^p)$, since $p<\alpha<2s$.
Both conditions imply \eqref{liouville-growth}.
\end{proof}

\begingroup
\renewcommand{\thesection}{}
\pretitlemark*{section}{Acknowledgments}
\section*{Acknowledgments}
\sectionmark{Acknowledgments}
Generative AI tools were used to assist with the development of proof strategies, mathematical calculations, the formulation and checking the arguments, and the preparation of the exposition. The authors take full responsibility for all mathematical claims, proofs, and references in this paper.
\endgroup

% MODULE bibliography
\begingroup\footnotesize
\renewcommand{\thesection}{}
\pretitlemark*{section}{\refname}

\endgroup
\end{document}